\documentclass[a4paper,10pt]{article}
\usepackage[utf8]{inputenc}
\usepackage[T1]{fontenc}
\usepackage[colorlinks=true,linkcolor=blue]{hyperref}
\usepackage{amsmath,amssymb,amsthm}
\usepackage{amsfonts}
\usepackage[english]{babel} 
\usepackage{dsfont}
\usepackage{lmodern}
\usepackage{enumerate}
\usepackage{listings}
\usepackage{graphicx}
\usepackage{array}
\usepackage{caption}
\usepackage{subcaption}
\usepackage{xcolor}
\usepackage{pgf,tikz}
\usetikzlibrary{decorations.pathreplacing}
\usepackage{stmaryrd}
\usepackage{xr}
\usepackage[a4paper]{geometry}
\newtheorem{rem}{Remark}[section]
\newtheorem{thm}{Theorem}[section]
\newtheorem{lem}{Lemma}[section]
\newtheorem{propo}{Proposition}[section]
\newtheorem{cor}{Corollary}[section]
\newtheorem{dfn}{Definition}[section]
\newtheorem{prop}{Property}[section]

\numberwithin{equation}{section}
\newcommand*{\R}{\mathbb{R}} % L'ensemble des réels
\newcommand*{\N}{\mathbb{N}} % L'ensemble des entiers
\newcommand*{\Z}{\mathbb{Z}} % L'ensemble des relatifs
\newcommand*{\C}{\mathbb{C}} % L'ensemble des complexes
\newcommand*{\PP}{\mathbb{P}}% Probabilité
\newcommand*{\bfP}{\mathbf{P}}% Probabilité
\newcommand*{\E}{\mathbb{E}}% Espérance
\newcommand*{\W}{\mathcal{W}}% espace bruit
\newcommand*{\X}{\mathcal{X}}% espace du processus X
\DeclareGraphicsExtensions{.jpg,.png,.pdf,.bmp,.eps}

\title{Rate of convergence to equilibrium for discrete-time stochastic dynamics with memory}

\author{Maylis Varvenne\thanks{Institut de Math\'ematiques de Toulouse, Universit\'e Paul Sabatier, 118 route de Narbonne, 31062 Toulouse, France. E-mail: maylis.varvenne@math.univ-toulouse.fr}}

{\scriptsize\date\today}

\begin{document}

\setlength{\parindent}{0pt}
\maketitle

\begin{abstract}
The main objective of the paper is to study the long-time behavior of general discrete dynamics driven by an ergodic stationary Gaussian noise. In our main result, we prove existence and uniqueness of the invariant distribution and exhibit some upper-bounds on the rate of convergence to equilibrium in terms of the asymptotic behavior of the covariance function of the Gaussian noise (or equivalently to its moving average representation). Then, we apply our general results to fractional dynamics (including the Euler Scheme associated to fractional driven Stochastic Differential Equations). When the Hurst parameter $H$ belongs to $(0,1/2)$ we retrieve, with a slightly more explicit approach due to the discrete-time setting, the rate exhibited by Hairer in a continuous time setting \cite{hairer2005ergodicity}. In this fractional setting, we also emphasize the significant dependence of the rate of convergence to equilibrium on the local behaviour of the covariance function of the Gaussian noise.
\end{abstract}

\bigskip

\textit{Keywords:} Discrete stochastic dynamics; Rate of convergence to equilibrium; Stationary Gaussian noise; Total variation distance; Lyapunov function; Toeplitz operator.

\nocite{arnold2013random}

\section{Introduction}

Convergence to equilibrium for Stochastic dynamics is one of the most natural and most studied problems in probability theory. Regarding Markov processes, this topic has been deeply undertaken through various approaches: spectral analysis, functional inequalities or coupling methods. However, in many applications (Physics, Biology, Finance...) the future evolution  of a quantity may depend on its own history, and thus, noise with independent increments does not accurately reflect reality. A classical way to overcome this problem is to consider dynamical systems driven by processes with stationary increments like fractional Brownian motion (fBm) for instance which is widely used in applications (see e.g \cite{guasoni2006no,jeon2011vivo,kou2008stochastic,odde1996stochastic}). In a continuous time framework, Stochastic Differential Equations (SDEs) driven by Gaussian processes with stationary increments have been introduced to model random evolution phenomena with long range dependence properties. Consider SDEs of the following form
\begin{equation}\label{SDE_Gaussian}
dX_t=b(X_t)dt+\sigma(X_t)dZ_t
\end{equation}
where $(Z_t)_{t\geqslant0}$ is a Gaussian process with ergodic stationary increments and $\sigma:\R^d\to\mathcal{M}_d(\R)$, $b:\R^d\to\R^d$ are functions defined in a such a way that existence of a solution holds. The ergodic properties of such processes have been a topic of great interest over the last decade. For general Gaussian processes, existence and approximation of stationary solutions are provided in \cite{cohen2011approximation} in the additive case (i.e. when $\sigma$ is constant). The specific situation where $(Z_t)_{t\geqslant0}$ is a fractional Brownian motion has received significant attention since in the seminal paper \cite{hairer2005ergodicity} by Hairer, a definition of invariant distribution is given in the additive case through the embedding of the solution to an infinite-dimensional markovian structure. This point of view led to some probabilistic uniqueness criteria (on this topic, see $e.g.$ \cite{hairer2007ergodic,hairer2013regularity}) and to some coupling methods in view of the study of the rate of convergence to equilibrium. More precisely, in \cite{hairer2005ergodicity}, some coalescent coupling arguments are also developed and lead to the convergence of the process in total variation to the stationary regime with a rate upper-bounded by $C_\varepsilon t^{-(\alpha_H-\varepsilon)}$ for any $\varepsilon>0$, with 
%As concerns long-time behavior, different properties have been studied like existence and uniqueness of stationary solution (see \cite{hairer2005ergodicity,hairer2007ergodic,hairer2013regularity} for fractional SDEs for instance), approximation of stationary solution in \cite{cohen2011approximation, cohen2014approximation} or the rate of convergence to an equilibrium distribution. For this last property, the case where $(Z_t)_{t\geqslant0}$ is fractional Brownian motion has received significant attention from Hairer \cite{hairer2005ergodicity}, Fontbona and Panloup \cite{fontbona2017rate}, Deya, Panloup and Tindel \cite{deya2016rate} over the last decade. They used coalescent coupling strategy to compute the rate of convergence. In the additive noise setting, Hairer proved that the process converges in total variation to the stationary regime 
\begin{equation}
\alpha_H=\left\{\begin{array}{ll}
\frac{1}{8}&\text{ if }\quad H\in(\frac{1}{4},1)\backslash\left\{\frac{1}{2}\right\}\\
H(1-2H)&\text{ if }\quad H\in(0,\frac{1}{4}).
\end{array}\right.
\end{equation}
In the multiplicative noise setting (i.e. when $\sigma$ is not constant), Fontbona and Panloup in \cite{fontbona2017rate} extended those results under selected assumptions on $\sigma$ to the case where $H\in(\frac{1}{2},1)$ and finally Deya, Panloup and Tindel obtained in \cite{deya2016rate} this type of results in the rough setting $H\in(\frac{1}{3},\frac{1}{2})$.\\

In this paper, we focus on a general class of recursive discrete dynamics of the following form: we consider $\R^d$-valued sequences $(X_n)_{n\geqslant0}$ satisfying
\begin{equation}\label{SDSintro}
X_{n+1}=F(X_n,\Delta_{n+1})
\end{equation} 
where $(\Delta_n)_{n\in\Z}$ is an ergodic stationary Gaussian sequence and $F:\R^d\times\R^d\to\R^d$ is a deterministic function. 
%As a typical example, we can think about discretization of \eqref{SDE_Gaussian}. Stochastic dynamics like \eqref{SDSintro}, which are not Markovian in general, have not been much discussed except in the linear case like Autoregressive Moving-average (ARMA) models \cite{brockwell2013time} whose main objective is the prediction of stationary processes. When $F$ is linear, \eqref{SDSintro} is truely related to ARMA processes through the so-called Wold's decomposition theorem which implies that we can see $(\Delta_n)_{n\in\Z}$ as a moving-average of infinite order (see \cite{brockwell2013time} to get more details). Here, we investigate the problem of the long-time behavior of \eqref{SDSintro} for a general class of functions $F$. \\
As a typical example for $F$, we can think about Euler discretization of \eqref{SDE_Gaussian} (see Subsection \ref{subsection:euler_scheme} for a detailed study) or to autoregressive processes in the particular case where $F$ is linear. Note that such dynamics can be written as \eqref{SDSintro} through the so-called Wold's decomposition theorem which implies that we can see $(\Delta_n)_{n\in\Z}$ as a moving-average of infinite order (see \cite{brockwell2013time} to get more details). To the best of our knowledge, in this linear setting the litterature mainly focuses on the statistical analysis of the model (see for instance \cite{brouste2014asymptotic, brouste2012kalman}) and on mixing properties of such Gaussian processes when it is in addition stationary (on this topic see e.g. \cite{bradley2005basic, kolmogorov1960strong, rosenblatt1972central}).
%Several papers in the litterature of autoregressive models with memory have been devoted to its statistical analysis (see for instance \cite{brouste2014asymptotic}). [...] 
Back to the main topic, namely ergodic properties of \eqref{SDSintro} for general functions $F$, Hairer in \cite{hairerergodic} has provided technical criteria using ergodic theory to get existence and uniqueness of invariant distribution for this kind of a priori non-Markovian processes. Here, the objective is to investigate the problem of the long-time behavior of \eqref{SDSintro}. To this end, we first explain how it is possible to define invariant distributions in this non-Markovian setting and to obtain existence results.
%Except this linear case where specific convergence methods could be derived from the fact that $(X_n)_{n\geqslant0}$ is a Gaussian process, the ergodicity of stochastic dynamics like \eqref{SDSintro}, which are not Markovian in general, has not been much discussed. Therefore, the objective of this paper is to investigate the problem of the long-time behavior of \eqref{SDSintro} for a general class of functions $F$.\\
More precisely, with the help of the moving average representation of the noise process $(\Delta_n)_{n\in\Z}$, we prove that $(X_n,(\Delta_{n+k})_{k\leqslant0})_{n\geqslant0}$ can be realized through a Feller transformation $\mathcal{Q}$. In particular, an initial distribution of the dynamical system $(X,\Delta)$ is a distribution $\mu$ on $\R^d\times(\R^d)^{\Z^-}$. Rephrased in more probabilistic terms, an initial distribution is the distribution of the couple $(X_0,(\Delta_{k})_{k\leqslant0})_{n\geqslant0}$. Then, such an initial distribution is called an invariant distribution if it is invariant by the transformation $\mathcal{Q}$. The first part of the main theorem establishes the existence of such an invariant distribution.\\
Then, our main contribution is to state a general result about the rate of convergence to equilibrium in terms of the covariance structure of the Gaussian process. To this end, we use a coalescent coupling strategy. \\
Let us briefly explain how this coupling method is organized in this discrete-time framework. First, one considers two processes $(X^1_{n},(\Delta^1_{n+k})_{k\leqslant0})_{n\geqslant0}$ and $(X^2_{n},(\Delta^2_{n+k})_{k\leqslant0})_{n\geqslant0}$ following \eqref{SDSintro} starting respectively from $\mu_0$ and $\mu_\star$ (an invariant distribution). As a preliminary step, one waits that the two paths get close. Then, at each trial, the coupling attempt is divided in two steps. First, one tries in Step 1 to stick the positions together at a given time. Then, in Step 2, one attempts to ensure that the paths stay clustered until $+\infty$. Actually, oppositely to the Markovian setting where the paths remain naturally fastened together (by putting the same innovation on each marginal), the main difficulty here is that, staying together has a cost. In other words, this property can be ensured only with a non trivial coupling of the noises. Finally, if one of the two previous steps fails, one begins Step 3 by putting the same noise on each coordinate until the ``cost'' to attempt Step 1 is not too big. More precisely, during this step one waits for the memory of the coupling cost to decrease sufficiently and for the two trajectories to be in a compact set with lower-bounded probability.\\
In the main theorem previously announced, as a result of this strategy, we are able to prove that the law of the process $(X_{n+k})_{k\geqslant0}$ following \eqref{SDSintro} converges in total variation to the stationary regime with a rate upper-bounded by $Cn^{-v}$. The quantity $v$ is directly linked to the assumed exponential or polynomial asymptotic decay of the sequence involved in the moving-average representation of $(\Delta_n)_{n\in\Z}$, see \eqref{eq:moving_average} (or equivalently in its covariance function, see Remark \ref{rem:memory} and \ref{rem:covariance_moving-average_memory}). \\
Then, we apply our main theorem to fractional memory (including the Euler Scheme associated to fractional Stochastic Differential Equations). We first emphasize that with covariance structures with the same order of memory but different local behavior, we can get distinct rates of convergence to equilibrium.
Secondly, we highlight that the computation of the asymptotic decay of the sequence involved in the inversion of the Toeplitz operator (related to the moving-average representation) can be a very technical task (see proof of Proposition \ref{propo:b_rate}). \\
Now, let us discuss about specific contributions of this discrete-time approach. Above all, our result is quite general since, for instance, it includes discretization of \eqref{SDE_Gaussian} for a large class of Gaussian noise processes. Then, in several ways, we get a further understanding of arguments used in the coupling procedure. We better target the impact of the memory through the sequence both appearing in the moving-average representation and the covariance function of the Gaussian noise. Regarding Step 1 of the coupling strategy, the ``admissibility condition'' (which means that we are able to attempt Step 1 with a controlled cost) is rather more explicit than in the continuous-time setting. Finally, this paper, by deconstructing the coupling method through this explicit discrete-time framework, may weigh in favour of the sharpness of Hairer's approach. \\

The following section gives more details on the studied dynamics, describes the assumptions required to get the main result, namely Theorem \ref{thm:principal} and discuss about the application of our main result to the case of fractional memory in Subsection \ref{subsection:fractional_sequence}. Then, the proof of Theorem \ref{thm:principal} is achieved in Sections \ref{section:existence_inv_distribution}, \ref{section:gen_coupling_strategy}, \ref{section:coupling_under_H}, \ref{section:admissibility} and \ref{section:proof_thm_principal}, which are outlined at the end of Section \ref{section:setting_result}.

\section{Setting and main results}\label{section:setting_result}

\subsection{Notations}

The usual scalar product on $\R^d$ is denoted by $\langle~,~\rangle$ and $|~.~|$ stands either for the Euclidean norm on
$\R^d$ or the absolute value on $\R$. For a given $K>0$ we denote by $B(0,K)$ the $\R^d$ closed ball centered in $0$ with radius $K$. Then, the \textit{state space} of the process $X$ and the \textit{noise space} associated to $((\Delta_{n+k})_{k\leqslant0})_{n\geqslant0}$ are respectively denoted by $\X:=\R^d$ and $\W:=(\R^d)^{\Z^-}$. For a given measurable space $(\X_1,\mathcal{A}_1)$, $~\mathcal{M}_1(\X_1)$ will denote the set of all probability measures on $\X_1$. Let $(\X_2,\mathcal{A}_2)$ be an other measurable space and $f:\X_1\to\X_2$ be a measurable mapping. Let $\mu\in\mathcal{M}_1(\X_1)$, we denote by $f^*\mu$ the pushforward measure given by : $$\forall B\in\mathcal{A}_2,\quad f^*\mu(B):=\mu(f^{-1}(B)).$$
$\Pi_\X:\X\times\W\to\X$ and $\Pi_\W:\X\times\W\to\W$ stand respectively for the projection on the marginal $\X$ and $\W$.
For a given differentiable function $f:\R^d\to\R^d$ and for all $x\in\R^d$ we will denote by $J_f(x)$ the Jacobian matrix of $f$ valued at point $x$.
Finally, we denote by $\|~.~\|_{TV}$ the classical total variation norm: let $\nu,\mu\in\mathcal{M}_1(\X_1)$,
$$\|\mu-\nu\|_{TV}:=\sup\limits_{A\subset\X_1}|\mu(A)-\nu(A)|.$$

\subsection{Dynamical system and Markovian structure}\label{subsection:setting}

Let $X:=(X_n)_{n\geqslant0}$ denote an $\R^d$-valued sequence defined by: $X_0$ is a random variable with a given distribution and
\begin{equation}\label{SDS}
\forall n\geqslant0,\quad X_{n+1}=F(X_n,\Delta_{n+1}),
\end{equation} 
where $F:\R^d\times\R^d\to\R^d$ is continuous and $(\Delta_n)_{n\in\Z}$ is a stationary and purely non-deterministic Gaussian sequence with $d$ independent components. Hence, by Wold's decomposition theorem \cite{brockwell2013time} it has a moving average representation
\begin{equation}\label{eq:moving_average}
\forall n\in\Z, ~~ \Delta_n=\sum_{k=0}^{+\infty}a_k\xi_{n-k}~~ 
\end{equation}
with
\begin{align}\label{eq:a_k}
\left\{\begin{array}{l}
              (a_k)_{k\geqslant0}\in\R^\N ~\text{ such that }~ a_0\neq0 ~\text{ and }~ \sum_{k=0}^{+\infty}a_k^2<+\infty \\
              (\xi_k)_{k\in\Z} ~\text{ an i.i.d sequence such that }\xi_1\sim\mathcal{N}(0,I_d).
             \end{array}\right.
\end{align}
Without loss of generality, we assume that $a_0=1$. Actually, if $a_0\neq1$, we can come back to this case by setting 
\[\tilde{\Delta}_{n}=\sum_{k=0}^{+\infty}\tilde{a}_k\xi_{n-k}\]
with $\tilde{a}_k:=\frac{a_k}{a_0}$. 
%\textcolor{red}{Moreover, the independency of the components of $(\Delta_n)$ is not a restrictive assumption. Indeed, if the components are not independent, one can build a matrix $A$ and an other stationary Gaussian sequence $(\tilde{\Delta}_n)$ with $d$ independent components such that $(\Delta_n)\overset{(d)}{=}(A\tilde{\Delta}_n)$. This case is then fulfilled by including $A$ in the functional $F$ in \eqref{SDS} and by considering the noise $(\tilde{\Delta}_n)$ instead of $(\Delta_n)$. }
\begin{rem}\label{rem:memory}
The asymptotic behavior of the sequence $(a_k)_{k\geqslant0}$ certainly plays a key role to compute the rate of convergence to equilibrium of the process $(X_n)_{n\geqslant0}$. Actually, the memory induced by the noise process is quantified by the sequence $(a_k)_{k\geqslant0}$ through the identity
\[\forall n\in\Z,\forall k\geqslant0,\quad c(k):=\E\left[\Delta_n\Delta_{n+k}\right]=\sum_{i=0}^{+\infty}a_ia_{k+i}.\]
\end{rem}

The stochastic dynamics described in \eqref{SDS} is clearly non-Markovian. 
Let us see how it is possible to introduce a Markovian structure and how to define invariant distribution. This method is inspired by \cite{hairerergodic}.
The first idea is to look at $(X_n,(\Delta_{n+k})_{k\leqslant0})_{n\geqslant0}$ instead of $(X_n)_{n\geqslant0}$. Let us introduce the following concatenation operator 
\begin{align}
\sqcup: \W\times\R^d &\to\W\\
             (w,w')&\mapsto w\sqcup w' \nonumber
\end{align}
where $(w\sqcup w')_0=w'$ and $\forall k<0, ~~ (w\sqcup w')_k=w_{k+1}$.
Then \eqref{SDS} is equivalent to the system 
\begin{equation}\label{SDS_equiv}
(X_{n+1},(\Delta_{n+1+k})_{k\leqslant0})=\varphi((X_{n},(\Delta_{n+k})_{k\leqslant0}),\Delta_{n+1})
\end{equation}
where
\begin{align*}\varphi: (\X\times\W)\times\R^d&\to\X\times\W\\
                           ((x,w),w')&\mapsto (F(x,w'),w\sqcup w').
\end{align*}
Therefore, $(X_n,(\Delta_{n+k})_{k\leqslant0})_{n\geqslant0}$ can be realized through the Feller Markov transition kernel $\mathcal{Q}$ defined by
\begin{equation}\label{transition_kernel_Q}
\int_{\X\times\W}g(x',w')\mathcal{Q}((x,w),({\rm d}x',{\rm d}w'))=\int_{\R^d}g(\varphi((x,w),\delta))\mathcal{P}(w,{\rm d}\delta)
\end{equation}
where $\mathcal{P}(w,{\rm d}\delta):=\mathcal{L}(\Delta_{n+1}|(\Delta_{n+k})_{k\leqslant0}=w)$ does not depend on $n$ since $(\Delta_n)$ is a stationary sequence, and $g:\X\times\W\to\R$ is a measurable function. 
\begin{dfn}\label{invariant_dist}
A measure $\mu\in\mathcal{M}_1(\X\times\W)$ is said to be an invariant distribution for \eqref{SDS_equiv} and then for \eqref{SDS} if it is invariant by $\mathcal{Q}$, i.e.
\[\mathcal{Q}\mu=\mu.\]
\end{dfn}
However, the concept of uniqueness will be slightly different from the classical setting. Indeed, denote by $\mathcal{S}\mu$ the distribution of $(X^{\mu}_n)_{n\geqslant0}$ when we realize $(X^\mu_n,(\Delta_{n+k})_{k\leqslant0})_{n\geqslant0}$ through the transition $\mathcal{Q}$ and with initial distribution $\mu$. Then, we will speak of uniqueness of the invariant distribution up to the equivalence relation: $\mu\sim\nu~\Longleftrightarrow~\mathcal{S}\mu=\mathcal{S}\nu$.\\
Moreover, here uniqueness will be deduced from the coupling procedure. There exist some results about uniqueness using ergodic theory, like in \cite{hairerergodic}, but they will be not outlined here. 

\subsection{Preliminary tool: a Toeplitz type operator}\label{subsection:toeplitz}

The moving-average representation of the Gaussian sequence $(\Delta_n)_{n\in\Z}$ naturally leads us to define an operator related to the coefficients $(a_k)_{k\geqslant0}$. First, set 
\begin{equation*}
\ell_a(\Z^-,\R^d):=\left\{w\in(\R^d)^{\Z^-}~\left|~\forall k\geqslant0,~\left|\sum_{l=0}^{+\infty}a_lw_{-k-l}\right|<+\infty\right\}\right.
\end{equation*}
and define $\mathbf{T_a}$ on $\ell_a(\Z^-,\R^d)$ by
\begin{equation}
\mathbf{T_a}(w)=\left(\sum_{l=0}^{+\infty}a_lw_{-k-l}\right)_{k\geqslant0}.
\end{equation}

Due to the Cauchy-Schwarz inequality, we can check that for instance $\ell^2(\Z^-,\R^d):=\{w\in(\R^d)^{\Z^-}~|~\sum_{k=0}^{+\infty}|w_{-k}|^2<+\infty\}$ is included in $\ell_a(\Z^-,\R^d)$ due to the assumption $\sum_{k\geqslant0}a_k^2<+\infty$. This Toeplitz type operator $\mathbf{T_a}$ links $(\Delta_n)_{n\in\Z}$ to $(\xi_n)_{n\in\Z}$. The following proposition spells out the reverse operator.

\begin{propo}\label{prop:inverse_T}
Let $\mathbf{T_b}$ be the operator defined on $\ell_b(\Z^-,\R^d)$ in the same way as $\mathbf{T_a}$ but with the following sequence $(b_k)_{k\geqslant0}$ 
\begin{equation}\label{eq:def_b}
b_0=\frac1{a_0} \quad\text{ and }\quad \forall k\geqslant1,~~ b_k=-\frac{1}{a_0}\sum_{l=1}^{k}a_lb_{k-l}.
\end{equation}
Then,
\[\forall w\in\ell_a(\Z^-,\R^d),\quad \mathbf{T_b}(\mathbf{T_a}(w))=w
\quad\text{ and }\quad\forall w\in\ell_b(\Z^-,\R^d),\quad \mathbf{T_a}(\mathbf{T_b}(w))=w\]
that is $\mathbf{T_b}=\mathbf{T_a}^{-1}$ and $\ell_b(\Z^-,\R^d)=\mathbf{T_a}(\ell_a(\Z^-,\R^d))$.
\end{propo}

\begin{proof} Let $w\in\ell_a(\Z^-,\R^d)$.
Then let $n\geqslant0$,
\begin{align*}
(\mathbf{T_b}(\mathbf{T_a}(w)))_{-n}
=\sum_{k=0}^{+\infty}b_k\sum_{l=0}^{+\infty}a_lw_{-n-k-l}
=\sum_{k=0}^{+\infty}\sum_{i=k}^{+\infty}b_ka_{i-k}w_{-n-i}
=\sum_{i=0}^{+\infty}\underbrace{\left(\sum_{k=0}^{i}b_ka_{i-k}\right)}_{\substack{=0\text{ except }\\\text{for }i=0}}w_{-n-i}=w_{-n}.
\end{align*}
We show in the same way that for $w\in\ell_b(\Z^-,\R^d)$, we have $\mathbf{T_a}(\mathbf{T_b}(w))=w$.
\end{proof}

\begin{rem}\label{rem:reverse_relation_b}
The sequence $(b_k)_{k\geqslant0}$ is of first importance in the sequel. The sketch of the proof of Theorem \ref{thm:principal} will use an important property linked to the sequence $(b_k)_{k\geqslant0}$: if two sequences $u$ and $v$ are such that \[\forall n\geqslant1,\quad u_n=\sum_{k=0}^{n-1}a_kv_{n-k}\]
then,
\[\forall n\geqslant1,\quad v_n=\sum_{k=0}^{n-1}b_ku_{n-k}.\]
This reverse identity and the asymptotic behavior of $(b_k)_{k\geqslant0}$ play a significant role in the computation of the rate of convergence.
\end{rem}

The following section is devoted to outline assumptions on $(a_k)_{k\geqslant0}$ and $(b_k)_{k\geqslant0}$ and then on $F$ to get the main result.

\subsection{Assumptions and general theorem}

First of all, let us introduce assumptions on $(a_k)_{k\geqslant0}$ and $(b_k)_{k\geqslant0}$.
All along the paper, we will switch between two types of assumptions called respectively the \textit{polynomial case} and the \textit{exponential case}.\\

\textbf{Hypothesis} $(\mathbf{H_{poly}})$: The following conditions hold,
\begin{itemize}
\item[$\bullet$] there exist $\rho,\beta>1/2$ and $C_\rho,C_\beta>0$ such that 
\[\forall k\geqslant0,~~ |a_k|\leqslant C_\rho(k+1)^{-\rho}\quad\text{ and }\quad
\forall k\geqslant0,~~ |b_k|\leqslant C_\beta(k+1)^{-\beta}.\]
\item[$\bullet$] there exist $\kappa\geqslant\rho+1$ and $C_\kappa>0$ such that
\[\forall k\geqslant0,~~ |a_k-a_{k+1}|\leqslant C_\kappa(k+1)^{-\kappa}.\]
\end{itemize}

\textbf{Hypothesis} $(\mathbf{H_{exp}})$: There exist $\lambda,\zeta>0$ and $C_\lambda,C_\zeta>0$ such that,
\[\forall k\geqslant0,~~ |a_k|\leqslant C_\lambda ~e^{-\lambda k}\quad\text{ and }\quad
\forall k\geqslant0,~~ |b_k|\leqslant C_\zeta ~e^{-\zeta k}.\]

\begin{rem}\label{rem:covariance_moving-average_memory}
$\rhd$ $(\mathbf{H_{poly}})$ and $(\mathbf{H_{exp}})$ are general parametric hypothesis which apply to a large class of Gaussian driven dynamics. These assumptions implicitly involve the covariance function of the noise process $(c(k))_{k\in\N}$ (see Remark \ref{rem:memory}) : there exists $\tilde{C}_\lambda>0$ and for all $\varepsilon\in[0,\rho]$ such that $\rho+\varepsilon>1$, there exists $\tilde{C}_{\rho,\varepsilon}>0$ such that
\[\forall k\geqslant0,\quad|c(k)|\leqslant\left\{\begin{array}{lll}
\tilde{C}_{\rho,\varepsilon}(k+1)^{-\rho+\varepsilon}&\text{ under }(\mathbf{H_{poly}})\\
\tilde{C}_\lambda e^{-\lambda k}&\text{ under }(\mathbf{H_{exp}})\\
\end{array}\right..\]
 $\rhd$ $(\mathbf{H_{poly}})$ and $(\mathbf{H_{exp}})$ also involve the coefficients appearing in the reverse
Toeplitz operator $\mathbf{T_a}^{-1}$ (see Proposition \ref{prop:inverse_T}). Even though $(a_k)_{k\geqslant0}$ and $(b_k)_{k\geqslant0}$ are related by \eqref{eq:def_b}, there is no general rule which connects $\rho$ and $\beta$. This fact will be highlighted in Subsection \ref{subsection:fractional_sequence}. Moreover, for the sake of clarity, we have chosen to state our main result when $(a_k)$ and $(b_k)$ belong to the same family of asymptotic decay rate.\\
$\rhd$ Due to the strategy of the proof (coalescent coupling in a non Markovian setting) we also need a bound on the discrete derivative of $(a_k)_{k\geqslant0}$.
\end{rem}

Let us now introduce some assumptions on the function $F$ which defines the dynamics \eqref{SDS}.
Throughout this paper $F:\X\times\R^d\to\X$ is a continous function and the following hypothesis $(\mathbf{H_1})$ and $(\mathbf{H_2})$ are satisfied. \\

\textbf{Hypothesis} $(\mathbf{H_1})$: There exists a continous function $V:\X\to\R^*_+$ satisfying 
$\lim\limits_{|x|\to+\infty}V(x)=+\infty$ and $\exists\gamma\in(0,1)$ and $C>0$ such that for all $(x,w)\in\X\times\R^d$, 
\begin{equation}\label{ineq:Lyapunov_type_condition}
V(F(x,w))\leqslant\gamma V(x)+C(1+|w|).
\end{equation}

\begin{rem} $\rhd$ As we will see in Section \ref{section:existence_inv_distribution}, this condition ensures the existence of a Lyapunov function $V$ and then of an invariant distribution. %(in a sense made precise below)
Such a type of assumption also appears in the litterature of discrete Markov Chains (see e.g. equation $(7)$ in \cite{down1995exponential}) but in an integrated form. More precisely, in our non-Markovian setting, a pathwise control is required to ensure some control on the moments of the trajectories before the successful coupling procedure (this fact is detailed in Subsection \ref{subsection:compact_return}). \\
$\rhd$ This assumption is also fulfilled if we have a function $V$ with $V(F(x,w))\leqslant\gamma V(x)+C(1+|w|^p)$ for a given $p>1$ instead of \eqref{ineq:Lyapunov_type_condition} since the function $V^{1/p}$ satisfies $(\mathbf{H_1})$ (with the help of the elementary inequality $|u+v|^{1/p}\leqslant|u|^{1/p}+|v|^{1/p}$ when $p>1$).
\end{rem}

%\begin{rem} $\rhd$ As we will see in Section \ref{section:existence_inv_distribution}, this condition ensures the existence of \textcolor{red}{a Lyapunov function $V$ and then of} an invariant distribution. %(in a sense made precise below)
%\textcolor{red}{Such a type of assumption also appears in the litterature of Markov Chains (see e.g. equation $(7)$ in \cite{down1995exponential}) but in an integrated form. More precisely, in our non-Markovian setting, a pathwise control is required (let us note that if the noises were iid random variables, we would get the same condition as in the Markovian setting by integrating \eqref{ineq:Lyapunov_type_condition}) if some trial in the coupling procedure fails. Indeed during Step 3, as mentioned in the introduction, we want to wait enough in order that the two trajectories belong to a compact set with a lower bounded probability. Such a lower bound is obtained under $(\mathbf{H_1})$ (see Subsection \ref{subsection:compact_return}).}\\
%$\rhd$ {\color{red} This assumption is also fulfilled if we have a function $V$ with $V(F(x,w))\leqslant\gamma V(x)+C(1+|w|^p)$ for a given $p>1$ instead of \eqref{ineq:Lyapunov_type_condition} since the function $V^{1/p}$ satisfies $(\mathbf{H_1})$.}
%\end{rem}

We define $\tilde{F}:\X\times\R^d\times\R^d\to\X~~$ by $\tilde{F}(x,u,y)=F(x,u+y)$.  We assume that $\tilde{F}$ satisfies the following conditions:\\

\textbf{Hypothesis} $(\mathbf{H_2})$: Let $K>0$. We assume that there exists $\tilde{K}>0$ such that for every $\mathbf{x}:=(x,x',y,y')$ in $B(0,K)^4$, there exist $\Lambda_{\mathbf{x}}:\R^d\to\R^d,~M_{K}>0$ and $C_{\tilde{K}}$
such that the following holds 
\begin{itemize}
\item[$\bullet$] $\Lambda_{\mathbf{x}}$ is a bijection from $\R^d$ to $\R^d$. Moreover, it is a $\mathcal{C}^1$-diffeomorphism between two open sets $U$ and $D$ such that $\R^d\backslash U$ and $\R^d\backslash D$ are negligible sets.
% a measurable, invertible and almost everywhere differentiable function 
% such that $\Lambda_{\mathbf{x}}^{-1}$ is also measurable.
\item[$\bullet$] for all $u\in B(0,\tilde{K})$, 
\begin{align}
\tilde{F}(x,u,y)&=\tilde{F}(x',\Lambda_{\mathbf{x}}(u),y')\label{eq:controlability}\\
\text{and }\quad|\det(J_{\Lambda_{\mathbf{x}}}&(u))|\geqslant C_{\tilde{K}}.\label{ineq:lambda_jacobian_bound}
\end{align}
\item[$\bullet$] for all $u\in \R^d$, 
\begin{equation}\label{ineq:lambda_unif_upperbound}
|\Lambda_{\mathbf{x}}(u)-u|\leqslant M_{K}.
\end{equation}
\end{itemize} 

\begin{rem} $\rhd$ Let us make a few precisions on the arguments of $\tilde{F}$: $x$ is the position of the process, $u$ the increment of the innovation process and $y$ is related to the past of the process (see next item for more details). The boundary $C_{\tilde{K}}$ and $M_{K}$ are independent from $x,x',y$ and $y'$. This assumption can be viewed as a kind of controlability assumption in the following sense: the existence of $\Lambda_{\mathbf{x}}$ leads to the coalescence of the positions by \eqref{eq:controlability}. 
Rephrased in terms of our coalescent coupling strategy, this \textit{ad hoc} assumption is required to achieve the first step. More precisely, as announced in the introduction, we take two trajectories $(X^1,\Delta^1)$ and $(X^2,\Delta^2)$ following \eqref{SDS} and we want to stick $X^1$ and $X^2$ at a given time $n+1$. Through the function $\Lambda_{\mathbf{x}}$ in $(\mathbf{H}_2)$, we can build a couple of Gaussian innovations $(\xi^1_{n+1},\xi^2_{n+1})$ with marginal distribution $\mathcal{N}(0,I_d)$ to achieve this goal (with lower-bounded probability), so that: $X^1_{n+1}=X^{2}_{n+1}$ which is equivalent to $\tilde{F}\left(x,~\xi^1_{n+1},~y\right)=\tilde{F}\left(x',~\xi^2_{n+1},~y'\right)$ with $(x,x',y,y')=\left(X^1_n,X^2_n,\sum_{k=1}^{+\infty}a_k\xi^1_{n+1-k},~\sum_{k=1}^{+\infty}a_k\xi^2_{n+1-k}\right)$ (see Subsection \ref{subsubsection:hitting_step_lem}).\\
$\rhd$ Assumption $(\mathbf{H_2})$ can be applied to a large class of functions $F$, as for example: $F(x,w)=f(b(x)+\sigma(x)w)$ where $\sigma$ is continuously invertible and $\sigma^{-1}$ and $b$ are continuous functions (we do not need any assumption on $f$ as we will see in the appendix Remark \ref{rem:relax_sigma_assumption}). Actually, Condition \eqref{eq:controlability} can be obtained through an application of the implicit function theorem: if we assume that there exists a point $(0,u'_{\mathbf{x}})$ such that $\tilde{F}(x,0,y)=\tilde{F}(x',u'_{\mathbf{x}},y')$ and denote by $G_{\mathbf{x}}:(u,u')\mapsto\tilde{F}(x,u,y)-\tilde{F}(x',u'+u'_{\mathbf{x}},y')$, then if $(\partial_{u'}G_{\mathbf{x}})(0,0)\neq0$, the implicit function theorem yields \eqref{eq:controlability}.
As we will see in Subsection \ref{subsection:euler_scheme}, Condition \eqref{ineq:lambda_unif_upperbound} can be also easily fulfilled (see proof of Theorem \ref{thm:euler_scheme}). 
\end{rem}

%$\rhd$ \textcolor{red}{This assumption is not as restrictive as it may appear at first sight. Actually, condition \eqref{eq:controlability} is an application of the implicit function theorem. If we assume that there exists a point $(0,u'_{\mathbf{x}})$ such that $\tilde{F}(x,0,y)=\tilde{F}(x',u'_{\mathbf{x}},y')$ and denote by $G_{\mathbf{x}}:(u,u')\mapsto\tilde{F}(x,u,y)-\tilde{F}(x',u'+u'_{\mathbf{x}},y')$, then if $(\partial_{u'}G_{\mathbf{x}})(0,0)\neq0$, the implicit function theorem gives \eqref{eq:controlability}. Moreover, to get \eqref{ineq:lambda_unif_upperbound}, the idea is to take a translation for $\Lambda_{\mathbf{x}}$ outside a well chosen ball and at the end, to extend $\Lambda_{\mathbf{x}}$ to $\R^d$. Finally, such a condition is quite general and certainly applies to dynamics linear in the noise as discretization of \eqref{SDE_Gaussian} (see Subsection \ref{subsection:euler_scheme}) but not exclusively.}

We are now in position to state our main result. %Let $\mathcal{L}((X^{\mu_0}_n)_{n\geqslant0})$ denote the distribution of the process $X$ starting from an initial condition $\mu_0$ (see Subsection \ref{subsection:markov_structure} below for detailed definitions of initial condition and invariant distribution) and for an invariant distribution $\mu_\star$ denote by $\mathcal{S}\mu_\star$ the law of the stationary solution. Finally, we denote by $\|.\|_{TV}$ the classical total variation norm.

\begin{thm}\label{thm:principal} Assume $(\mathbf{H_1})$ and $(\mathbf{H_2})$. Then,
\begin{itemize}
\item[(i)] There exists an invariant distribution $\mu_\star$ associated to \eqref{SDS}.
\item[(ii)] Assume that $(\mathbf{H_{poly}})$ is true with $\rho,\beta>1/2$ and $\rho+\beta>3/2$. Then, uniqueness holds for the invariant distribution $\mu_\star$. Furthermore, for every initial distribution $\mu_0$ for which $\int_\X V(x)\Pi^*_\X\mu_0(dx)<+\infty$ and for all $\varepsilon>0$, there exists $C_\varepsilon>0$ such that 
\[\|\mathcal{L}((X^{\mu_0}_{n+k})_{k\geqslant0})-\mathcal{S}\mu_\star \|_{TV}\leqslant C_\varepsilon~n^{-(v(\beta,\rho)-\varepsilon)}.\] 
where the function $v$ is defined by 
\[v(\beta,\rho)=\sup\limits_{\alpha\in\left(\frac{1}{2}\vee\left(\frac{3}{2}-\beta\right),\rho\right)}\min\{1,2(\rho-\alpha)\}(\min\{\alpha,~\beta,~\alpha+\beta-1\}-1/2).\] 
\item[(iii)] Assume that $(\mathbf{H_{exp}})$ is true, then uniqueness holds for the invariant distribution $\mu_\star$. Furthermore, for every initial distribution $\mu_0$ for which $\int_\X V(x)\Pi^*_\X\mu_0(dx)<+\infty$ and for all\\ $p>0$, there exists $C_p>0$ such that 
\[\|\mathcal{L}((X^{\mu_0}_{n+k})_{k\geqslant0})-\mathcal{S}\mu_\star \|_{TV}\leqslant C_p~n^{-p}.\]
\end{itemize}
\end{thm}

\begin{rem}\label{rem:thm_exp}$\rhd$ In view of Theorem \ref{thm:principal} (iii), one can wonder if we could obtain exponential or subexponential rates of convergence in this case. We focus on this question in Remark \ref{rem:exp_rate}.\\ %and in particular we look at the case of a finite moving average. \\
$\rhd$ The rates obtained in Theorem \ref{thm:principal} hold for a large class of dynamics. This generality implies that the rates are not optimal in all situations. In particular, when $F$ have ``nice'' properties an adapted method could lead to better rates. For example, let us mention the particular case where the dynamical system \eqref{SDS} is reduced to: $X_{n+1}=AX_{n}+\sigma\Delta_{n+1}$ where $A$ and $\sigma$ are some given matrices. As for (fractional) Ornstein-Uhlenbeck processes in a continuous setting, the study of linear dynamics can be achieved with specific methods. Here, the sequence $X$ benefits of a Gaussian structure so that the convergence in distribution could be studied through the covariance of the process. One can also remark that since for two paths $X$ and $\tilde{X}$ built with the same noise, we have: $\tilde{X}_{n+1}-X_{n+1}=A(\tilde{X}_{n}-X_{n})$, a simple induction leads to $\E[|\tilde{X}_{n}-X_{n}|^2]\leqslant |A^n|^2\E[|\tilde{X}_{0}-X_{0}|^2]$. So without going into the details, if $\rho(A):=\lim\limits_{n\to+\infty}|A^n|^{1/n}<1$, such bounds lead to geometric rates of convergence in Wasserstein distance and also in total variation distance (on this topic, see e.g. \cite{panloup2018sub}). 
%To conclude, it is worth noting that our coupling strategy (and the related rates in Theorem \ref{thm:principal}) is more adapted to a setting where the drift term in $F$ does not contract everywhere as in linear dynamics.}
\end{rem}

In the following subsection, we test the assumptions of our main result Theorem \ref{thm:principal} (especially $(\mathbf{H_{1}})$ and $(\mathbf{H_{2}})$) on the Euler scheme of SDEs like \eqref{SDE_Gaussian}.

\subsection{The Euler Scheme}\label{subsection:euler_scheme}
Recall that $\X=\R^d$.
In this subsection, set 
\begin{align} \label{function:euler_scheme}
F_h: \quad \X\times\R^d&\to\X\nonumber\\
                           (x,w)&\mapsto x+hb(x)+\sigma(x)w.
\end{align}

where $h>0$, $b:\X\to\X$ is continuous and $\sigma: \X\to\mathcal{M}_d(\R)$ is a continuous and bounded function on $\X$. For all $x\in\X$ we suppose that $\sigma(x)$ is invertible and we denote by $\sigma^{-1}(x)$ the inverse. Moreover, we assume that $\sigma^{-1}$ is a continuous function and that $b$ satisfies a Lyapunov type assumption that is:
\begin{itemize}
 \item[(\textbf{L1})]$\exists C>0$ such that
 \begin{equation}\label{eq:sublinearity_b}
 \forall x\in\X,~~|b(x)|\leqslant C(1+|x|).
 \end{equation}
 \item[(\textbf{L2})]$\exists \tilde{\beta}\in\R \text{ and }\tilde{\alpha}>0$ such that
  \begin{equation}\label{eq:Lyapunov_b}
\forall x\in\X,~~\langle x,b(x)\rangle\leqslant \tilde{\beta}-\tilde{\alpha}|x|^2.
 \end{equation}
\end{itemize}

\begin{rem} This function $F_h$ corresponds to the Euler scheme associated to SDEs like \eqref{SDE_Gaussian}. The conditions on the function $b$ are classical to get existence of invariant distribution. 
\end{rem}

In this setting the function $\tilde{F}_h$ (introduced in Hypothesis $(\mathbf{H_2})$) is given by 
\begin{align*} 
\tilde{F}_h: \quad \X\times\R^d\times\R^d&\to\X\\
                           (x,u,y)&\mapsto x+hb(x)+\sigma(x)(u+y).
\end{align*}

\begin{thm}\label{thm:euler_scheme} Let $h>0$. Let $F_h$ be the function defined above. Assume that $b:\X\to\X$ is a continuous function satisfying $(\mathbf{L1})$ and $(\mathbf{L2})$ and $\sigma: \X\to\mathcal{M}_d(\R)$ is a continous and bounded function such that for all $x\in\X$, $\sigma(x)$ is invertible and $x\mapsto\sigma^{-1}(x)$ is a continuous function. Then, $(\mathbf{H_1})$ and $(\mathbf{H_2})$ hold for $F_h$ as soon as $~0<h<\min\left\{\sqrt{1+\frac{\tilde{\alpha}}{2C^2}}-1,~\frac{1}{\tilde{\alpha}}\right\}$ where $\tilde{\alpha}$ and $C$ are given by $(\mathbf{L1})$ and $(\mathbf{L2})$.
\end{thm}

\begin{proof} 
For the sake of conciseness, the proof is detailed in Appendix \ref{appendix:proof_euler_scheme}. Regarding $(\mathbf{H_1})$, it makes use of ideas developped in \cite{cohen2011approximation, cohen2014approximation}. For $(\mathbf{H_2})$, the construction of $\Lambda_{\mathbf{x}}$ is explicit. The idea is to take $\Lambda_{\mathbf{x}}(u):=\sigma^{-1}(x')\sigma(x)u+\sigma^{-1}(x')(x-x'+h(b(x)-b(x')))+\sigma^{-1}(x')\sigma(x)y-y'$ inside $B(0,K)$ (which ensures \eqref{eq:controlability} with $\tilde{K}=K$), to set $\Lambda_{\mathbf{x}}(u):=u$ outside an other ball $B(0,K_1)$ with a well chosen $K_1$ (which almost gives \eqref{ineq:lambda_unif_upperbound}) and to extend $\Lambda_{\mathbf{x}}$ into $\R^d$ by taking into account the various hypothesis on $\Lambda_{\mathbf{x}}$.
\end{proof}

\begin{rem} The coefficient $\sigma$ is assumed to be bounded but we can relax a bit this assumption: if $|\sigma(x)|\leqslant C(1+|x|^{\kappa})$ for some $\kappa\in(0,1)$, then Theorem \ref{thm:euler_scheme} is true again. We provide in Remark \ref{rem:relax_sigma_assumption} the necessary adjustments of the proof to get this result. \\
%$\rhd$ A natural example of function $F$ can be derived from the Euler Scheme: we can choose $F(x,w):=b(x)+\sigma(x)w~$ with $\sigma$ as previously and $b$ being a sublinear function : $|b(x)|\leqslant C+\alpha|x|$ with $\alpha<1$. In that case, $F$ verifies $(\mathbf{H_1})$ and $(\mathbf{H_2})$. %Going a bit further, we can add a sublinear function $\varphi$ (such that $|\varphi(x)|\leqslant C(1+|x|)$ with $C<1/\alpha$) so that $F(x,w)=\varphi\left(b(x)+\sigma(x)w\right)$ also satisfies $(\mathbf{H_1})$ and $(\mathbf{H_2})$. %In order to be less general, let us cite a precise arbitrary example in dimension $1$:
\end{rem}

The two following subsections are devoted to outline examples of sequences which satisfy hypothesis $(\mathbf{H_{exp}})$ or $(\mathbf{H_{poly}})$. In particular, Subsection \ref{subsection:fractional_sequence} includes the case where the process $(\Delta_n)_{n\in\Z}$ corresponds to fractional Brownian motion increments.

\subsection{Two explicit cases which satisfy $(\mathbf{H_{exp}})$}\label{subsection:exponantial_rate}

First, let us mention the explicit exponential case with the following definition for the sequence $(a_k)_{k\geqslant0}$
\begin{equation}\label{eq:a_k_expo}
a_0=1\quad\text{ and }\quad\forall k\in\N^*,~a_k=C_ae^{-\lambda k}
\end{equation}
with $C_a\in\R$. Let us recall that $b_0=1$ (since $a_0=1$) and for all $k\geqslant1$, we can get the following general expression of $b_k$ (see Appendix \ref{appendix:formula_b} for more details):

\begin{equation}\label{expo_b_k}
b_k=\sum_{p=1}^k\frac{(-1)^p}{a_0^{p+1}}\left(\sum_{\substack{k_1,\dots,k_p\geqslant1\\ k_1+\dots+k_p=k}}
\prod_{i=1}^pa_{k_i}\right).
\end{equation}

A classical combinatorial argument shows that $\sharp\{(k_1,\dots,k_p)\in\N^*~|~k_1+\dots+k_p=k\}=\binom{k-1}{p-1}$.
As a consequence, when the sequence $(a_k)_{k\geqslant0}$ is defined by \eqref{eq:a_k_expo}, we can easily prove that for $k\geqslant1$, \[b_k=-C_a(1-C_a)^{k-1}e^{-\lambda k}.\]  
Hence, to satisfy $(\mathbf{H_{exp}})$, we only need $C_a$ to be such that $\zeta:=\lambda-\ln|1-C_a|>0$ and then for all $k\in\N^*$, we get
\begin{equation}\label{domination_b_k}
|b_k|\leqslant C_be^{-\zeta k}
\end{equation}
with $C_b>0$ a constant depending on $C_a$.

\begin{rem} $\rhd$ In this setting where everything is computable, it's interesting to see that the asymptotic decrease of the sequence $(|b_k|)$ is not only related to the one of the sequence $(|a_k|)$. For instance, if we take $C_a<0$, the simple fact that $a_0>0$ and $a_k<0$ for all $k>0$ makes $(b_k)$ diverge to $+\infty$ and nevertheless, $(|a_k|)$ decreases to $0$ at an exponential rate.\\
$\rhd$ If we take $C_a=1$, we can reduce $(\Delta_n)_{n\in\Z}$ to the following induction: 
$
~\forall n\in\Z,~\Delta_{n+1}=\xi_{n+1}+e^{-\lambda}\Delta_n.$
\end{rem}

Let us finally consider finite moving averages, i.e. when $a_k=0$ for all $k> m$ (for some $m\geqslant1$). In this setting, one can expect $(\mathbf{H_{exp}})$ to be satisfied since the memory is finite. This is actually the case
when the finite moving average is invertible, namely: $P(\lambda):=1+\sum_{k=1}^ma_k\lambda^k$ has all its roots outside the unit circle (see \cite{brockwell2013time} Theorem 3.1.2).
In that case, there exists $\lambda\in\C$ such that $|\lambda|>1$ and $\frac{1}{P(\lambda)}=\sum_{k=0}^{+\infty}b_k\lambda^k<+\infty$ (to get more details on this equality, see Appendix \ref{appendix:formula_b}). Then, there exists $C>0$ such that for all $k\geqslant0$,$$|b_k|\leqslant C\left(\frac{1}{|\lambda|}\right)^k$$ and finally $(\mathbf{H_{exp}})$ holds true.\\
When the invertibility is not fulfilled, the situation is more involved but $(\mathbf{H_{exp}})$ is still true up to another Wold decomposition. More precisely, one can find another white noise $\xi$ and another set of coefficients $a_k$ such that the invertibility holds true, on this topic see e.g. \cite{brockwell2013time} Proposition 4.4.2.

\subsection{Polynomial case: from a general class of examples to the fractional case}\label{subsection:fractional_sequence}

A natural example of Gaussian sequence $(\Delta_n)_{n\in\Z}$ which leads to polynomial memory is to choose $a_k=(k+1)^{-\rho}$ for a given $\rho>1/2$. In that case, we have the following result.

\begin{propo}\label{propo:polynomial_case} Assume $(\mathbf{H_1})$ and $(\mathbf{H_2})$. Let $\rho>1/2$. If for all $k\geqslant0$, $a_k=(k+1)^{-\rho}$, then we have $|b_k|\leqslant(k+1)^{-\rho}$. Moreover, if $\rho>3/4$ Theorem \ref{thm:principal} (ii) holds with the rate 
\[v(\rho,\rho)=\left\{\begin{array}{ll}
2(\rho-3/4)^2&\text{ if }\quad \rho\in(3/4,1]\\
\frac{1}{2}(\rho-1/2)^2&\text{ if }\quad \rho>1.
\end{array}\right.\]
\end{propo}

\begin{rem}
The main novelty here comes from the proof of the inequality $|b_k|\leqslant(k+1)^{-\rho}$ for all $k\geqslant0$ which is outlined in Appendix \ref{appendix:a_log_convex} and which is based on results of \cite{ford2014decay}. The key argument in this proof is the log-convexity property of the sequence $(a_k)_{k\in\N}$, which means that for all $k\in\N$, $a_k\geqslant0$ and for $k\geqslant1$, $a_k^2-a_{k-1}a_{k+1}\leqslant0$. 
\end{rem}

With Proposition \ref{propo:polynomial_case} in hand, the purpose of the remainder of this section is to
focus on Gaussian sequences of \textit{fractional type}, i.e. when the sequence $(a_k)_{k\geqslant0}$ satisfies:
\begin{equation}\label{eq:fractional_type_a}
\forall k\geqslant0,\quad|a_k|\leqslant C_\rho(k+1)^{-\rho}\quad\text{ and }\quad|a_k-a_{k+1}|\leqslant\tilde{C}_\rho(k+1)^{-(\rho+1)}
\end{equation}
with $\rho:=3/2-H$ and $H\in(0,1)$ is the so-called Hurst parameter. In particular, through this class of sequences, we provide an explicit example which shows that computing the rate of convergence of the sequence $(b_k)_{k\geqslant0}$ is a hard task and strongly depends on the variations of $(a_k)_{k\geqslant0}$. Condition \eqref{eq:fractional_type_a} includes both cases of Proposition \ref{propo:polynomial_case} with $\rho:=3/2-H$ and when $(\Delta_n)_{n\in\Z}$ corresponds to the fractional Brownian motion (fBm) increments (as we will see below), we therefore decided to use the terminology ``\textit{fractional type}''.
%above all to provide an explicit example which shows that computing the rate of convergence of the sequence $(b_k)_{k\geqslant0}$ is a hard task and strongly depends on the variations of $(a_k)_{k\geqslant0}$. To this end, we  This condition includes both cases of Proposition \ref{propo:polynomial_case} with $\rho:=3/2-H$ and when $(\Delta_n)_{n\in\Z}$ corresponds to the fractional Brownian motion (fBm) increments (as we will see below), we therefore decided to use the terminology ``\textit{fractional type}''.}
Recall that a $d$-dimensional fBm with Hurst parameter $H\in(0,1)$ is a centered Gaussian process $(B_t)_{t\geqslant0}$ with stationary increments satisfying
\begin{equation*}
\forall t,s\geqslant0,~\forall i,j\in\{1,\dots,d\},\quad\E\left[(B^i_t-B^i_s)(B^j_t-B^j_s)\right]=\delta_{ij}|t-s|^{2H}.
\end{equation*}
%The study by a coupling argument of the rate of convergence to equilibrium for this kind of dynamics has been undertaken by Hairer \cite{hairer2005ergodicity}, Fontbona and Panloup \cite{fontbona2017rate}, Deya, Panloup and Tindel \cite{deya2016rate}, respectively in the additive noise, multiplicative noise with $H>1/2$ and multiplicative noise with $H\in(1/3,1/2)$. 
In our discrete-time setting, we are thus concerned by the long time behavior of \eqref{SDS} if we take for $h>0$
\begin{equation}
(\Delta_{n})_{n\in\Z}=(B_{nh}-B_{(n-1)h})_{n\in\Z}
\end{equation}
which is a stationary Gaussian sequence. It can be realized through a moving average representation with coefficients $(a^H_k)_{k\geqslant0}$ defined by (see \cite{ostry2006synthesis}):
\begin{equation}
a^H_0=h^{H}\kappa(H)2^{1/2-H} \quad \text{ and for }k\geqslant1,\quad a^H_k=h^{H}\kappa(H)\left(\left(k+\frac{1}{2}\right)^{H-1/2}-\left(k-\frac{1}{2}\right)^{H-1/2}\right)
\end{equation}
where \[\kappa(H)=\frac{\sqrt{\sin (\pi H)\Gamma(2H+1)}}{\Gamma(H+1/2)}.\]

One can easily check that $a^H_k\underset{k\to+\infty}{\sim}C_{h,H}(k+1)^{-(3/2-H)}$ and $|a^H_k-a^H_{k+1}|\leqslant C'_{h,H}(k+1)^{-(5/2-H)}$.\\ Hence $(a_k^H)_{k\geqslant0}$ is of \textit{fractional type} in the sense of \eqref{eq:fractional_type_a}.
Now, the question is: how does the corresponding $(b^H_k)$ behave ? When $H$ belongs to $(0,1/2)$, only $a_0^H$ is positive and then $(a^H_k)$ is not log-convex. Therefore, we cannot use this property to get the asymptotic behavior of $(b^H_k)$ as we did in Proposition \ref{propo:polynomial_case}. However, thanks to simulations (see Figure \ref{fig:b_mbf_H1} and \ref{fig:b_mbf_H3}), we conjectured and we proved  Proposition \ref{propo:b_rate}.

\begin{propo}\label{propo:b_rate}
There exists $C''_{h,H}>0$ such that for all $H\in(0,1/2)$
\begin{equation}\label{eq:rate_b_fbm}
\forall k\geqslant0,\quad |b^H_k|\leqslant C''_{h,H}(k+1)^{-(H+1/2)}.
\end{equation}
Then, if we assume $(\mathbf{H_1})$ and $(\mathbf{H_2})$, Theorem \ref{thm:principal} (ii) holds with the rate 
\[v(\rho,2-\rho)=v(3/2-H,H+1/2)=\left\{\begin{array}{ll}
H(1-2H)&\text{ if }\quad H\in(0,1/4]\\
\frac{1}{8}&\text{ if }\quad H\in(1/4,1/2).
\end{array}\right.\]
\end{propo}

\begin{rem} For the sake of conciseness, we provided the details of the proof in Appendix \ref{appendix:proof_propo_b_rate}. The proof of \eqref{eq:rate_b_fbm} is based on a technical induction which involves very sharp inequalities.
\end{rem}
In Proposition \ref{propo:polynomial_case} (with $\rho=3/2-H$ and $H\in(0,1/2)$) and in the above proposition, dealing with the same order of memory, we get really different orders of rate of convergence: one easily checks that for all $H\in(0,1/2)$, $v(3/2-H,~H+1/2)<v(3/2-H,~3/2-H)$. Finally, we have seen that managing the asymptotic behavior of $(b_k)$ is both essential and a difficult task.\\

To end this section, let us briefly discuss on the specific statements on fBm increments and compare with the continuous time setting (see \cite{hairer2005ergodicity,fontbona2017rate,deya2016rate}). For this purpose, we introduce the following conjecture (based on simulations, see Figure \ref{fig:b_mbf_H6} and \ref{fig:b_mbf_H9}) when $H$ belongs to $(1/2,1)$:\\

\underline{\textbf{Conjecture:}}\quad There exists $C''_{h,H}>0$ and $\beta_H>1$ such that
\begin{equation}\label{eq:conjecture}
\forall k\geqslant0,\quad |b^H_k|\leqslant C''_{h,H}(k+1)^{-\beta_H}.
\end{equation}

\begin{rem} 
We do not have a precise idea of the expression of $\beta_H$ with respect to $H$. But, we can note that if $\rho<1$ and $\beta>1$ in $(\mathbf{H_{poly}})$, then the rate of convergence in Theorem \ref{thm:principal} is $v(\rho,\beta)=\frac{(2\rho-1)^2}{8}$ and does not depend on $\beta$. Hence, if $H\in(1/2,1)$, $\rho=3/2-H$ and $\beta_H>1$, we fall into this case and then the dependence of $\beta_H$ in terms of $H$ does not matter.
\end{rem}
If the previous conjecture is true we get the following rate of convergence for $H\in(1/2,1)$ in Theorem \ref{thm:principal}:
\[v(\rho,\beta_H)=v(3/2-H,\beta_H)=\frac{(1-H)^2}{2}.\]

\begin{center}
\begin{figure}[!ht]
    \centering
    \begin{subfigure}[h]{0.4\textwidth}
        \includegraphics[width=\textwidth]{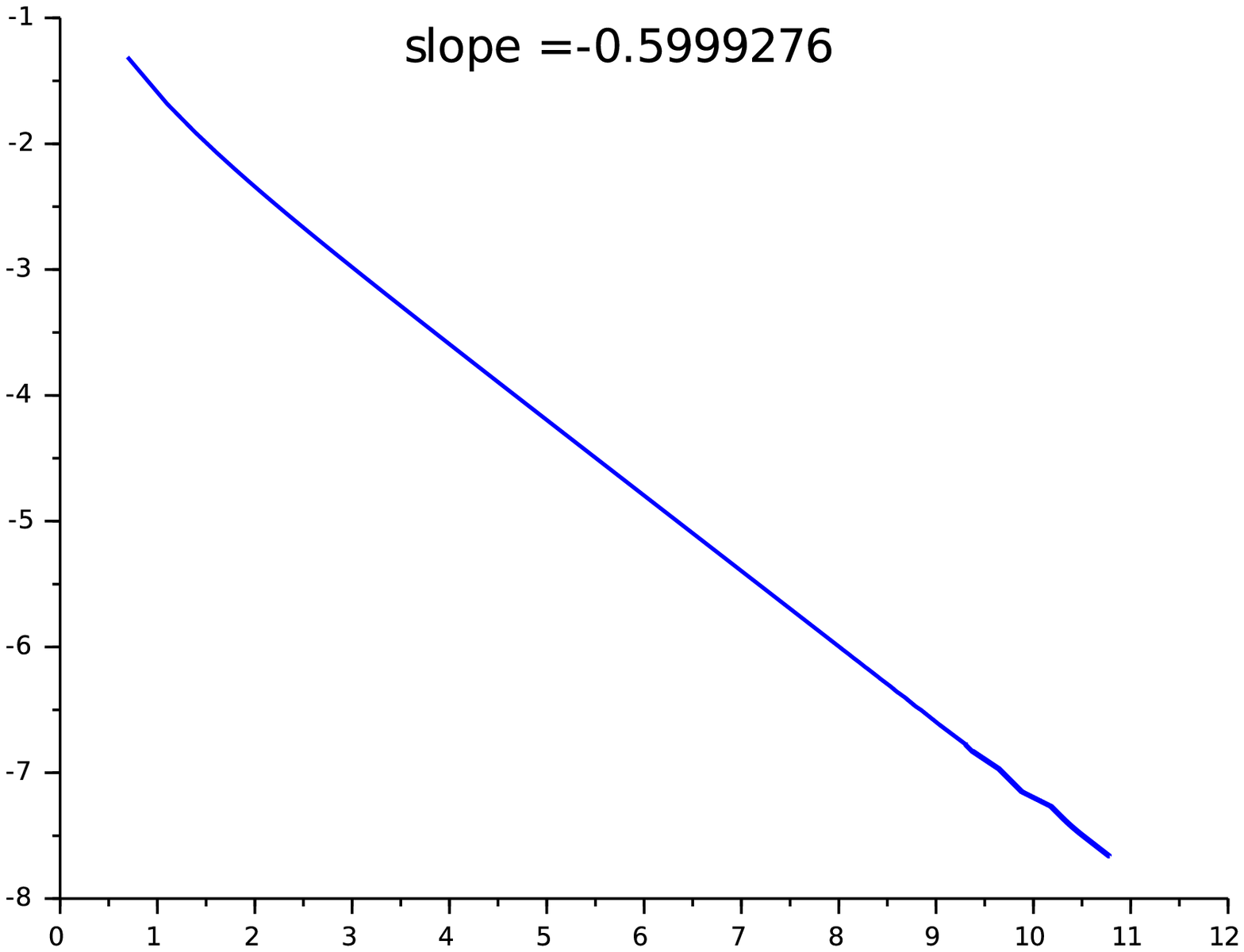}
        \caption{parameter: $H=0.1$}
        \label{fig:b_mbf_H1}
    \end{subfigure}
    \begin{subfigure}[h]{0.4\textwidth}
        \includegraphics[width=\textwidth]{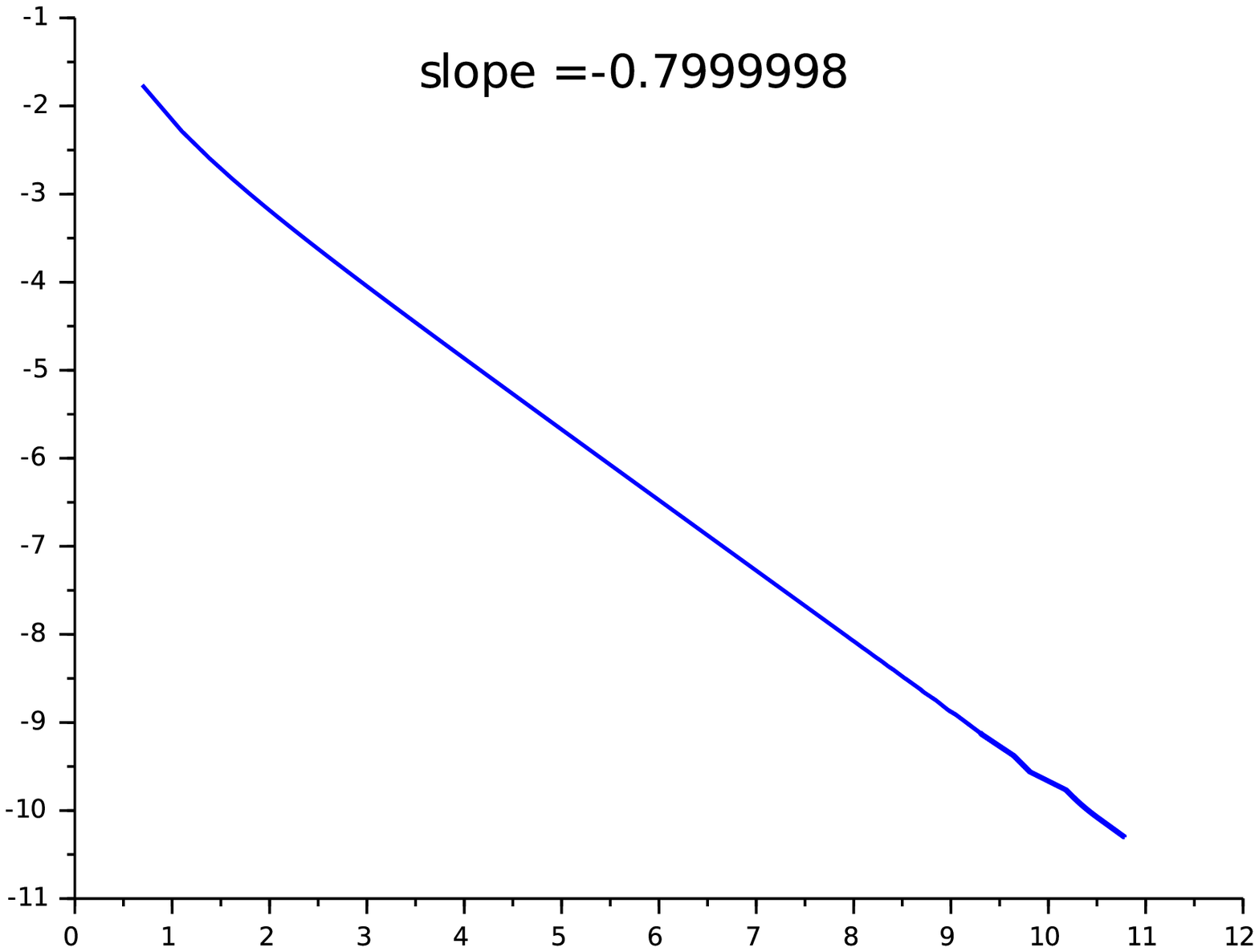}
        \caption{parameter: $H=0.3$}
        \label{fig:b_mbf_H3}
    \end{subfigure}
    \begin{subfigure}[h]{0.4\textwidth}
        \includegraphics[width=\textwidth]{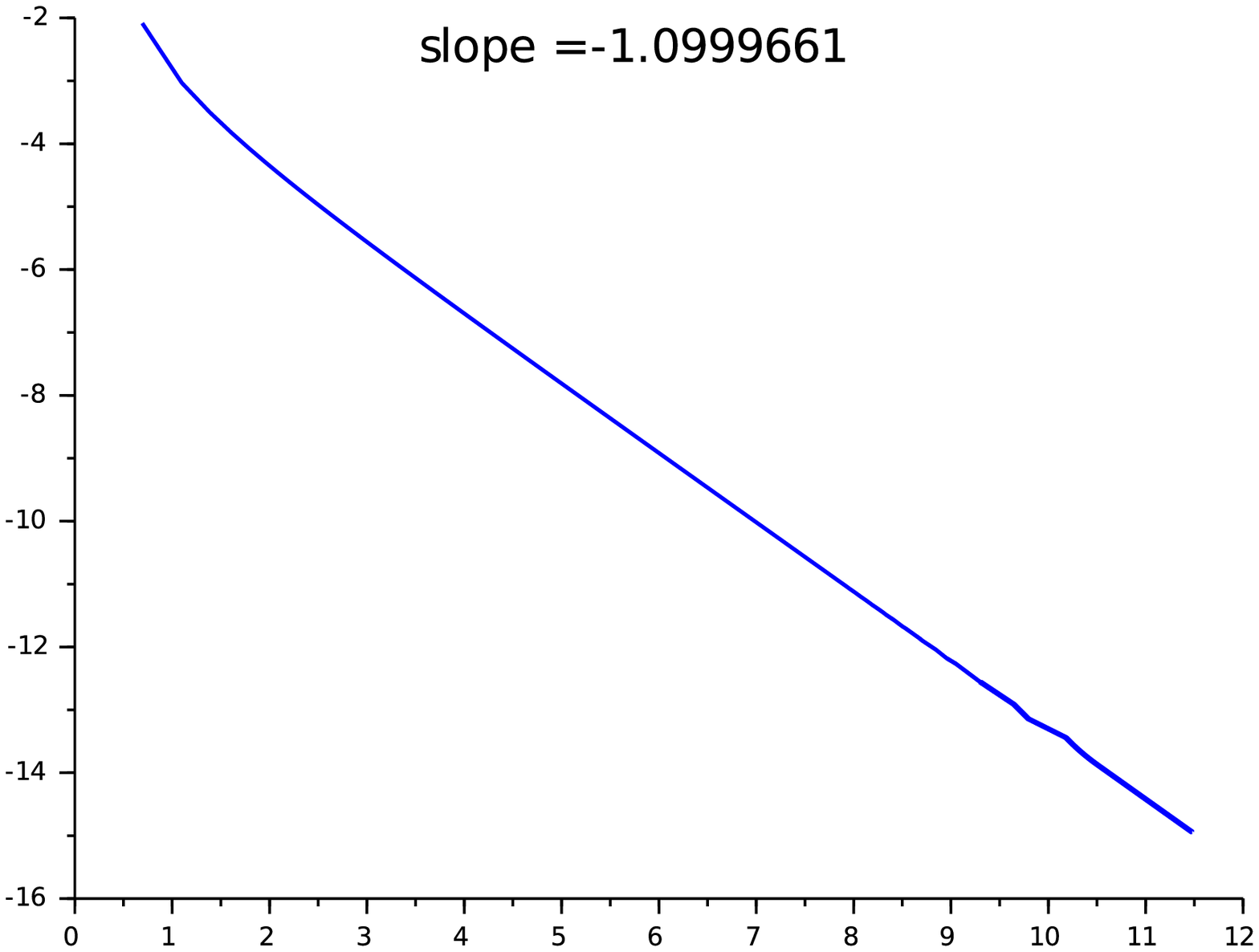}
        \caption{parameter: $H=0.6$}
        \label{fig:b_mbf_H6}
    \end{subfigure}
    \begin{subfigure}[h]{0.4\textwidth}
        \includegraphics[width=\textwidth]{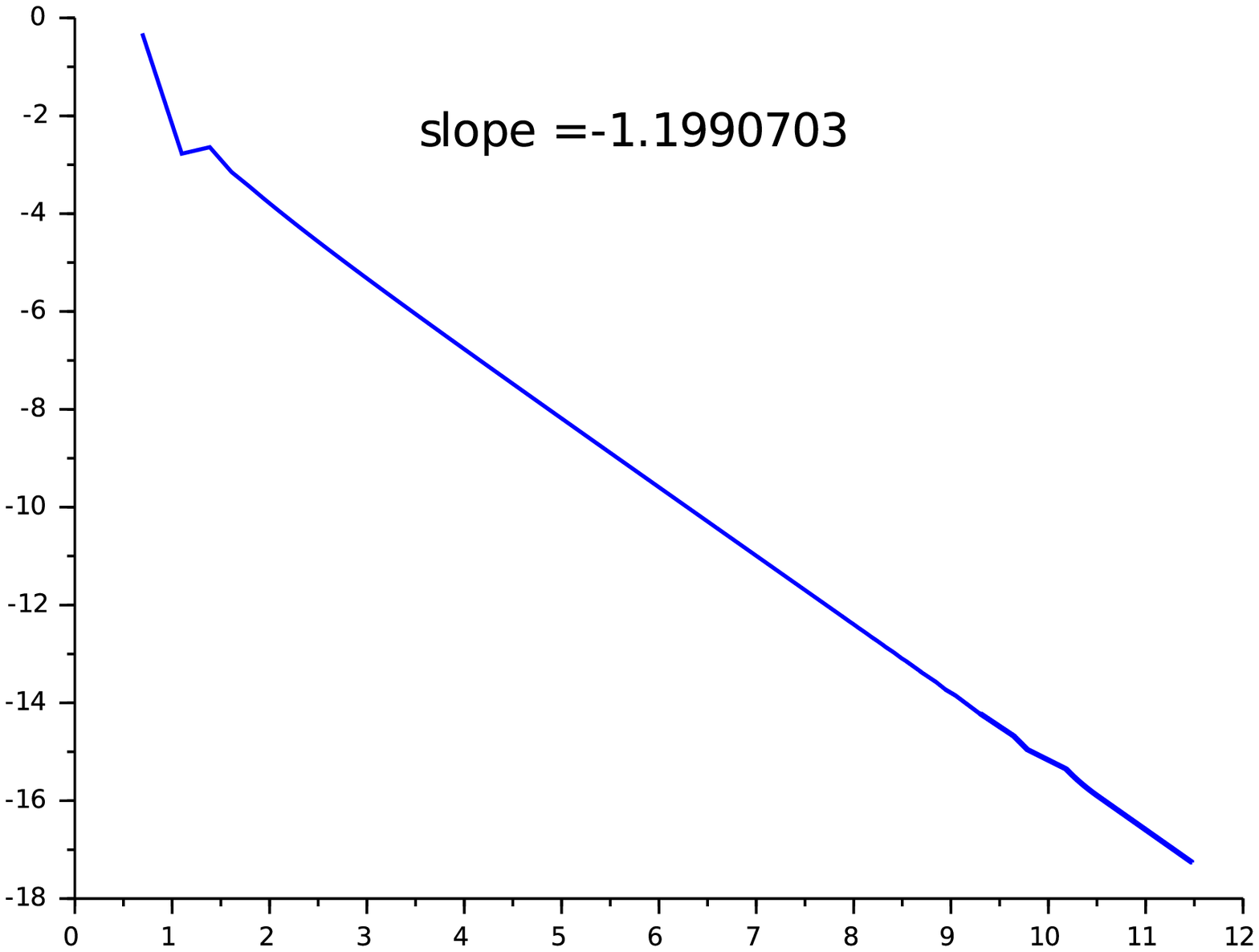}
        \caption{parameter: $H=0.9$}
        \label{fig:b_mbf_H9}
    \end{subfigure}
    \caption{$(\log|b^H_k|)$ according to $(\log(k+1))$ with different Hurst parameters $H$.}\label{fig:mBf}
\end{figure}
\end{center}

Then, when $H$ belongs to $(0,1/2)$ Proposition \ref{propo:b_rate} gives exactly the same rate of convergence obtained in \cite{hairer2005ergodicity,deya2016rate}. However, when $H>1/2$ it seems that we will get a smaller rate than in a continuous time setting. The reason for this may be that Theorem \ref{thm:principal} is a result with quite general hypothesis on the Gaussian noise process $(\Delta_n)_{n\in\Z}$. In the case of fBm increments, the moving average representation is explicit. Hence, we may use a more specific approach and significantly closer to Hairer's, especially with regard to Step 2 in the coupling method (see Subsection \ref{subsection:step2}) by not exploiting the technical lemma \ref{lem:technical} for instance. This seems to be a right track in order to improve our results on this precise example.\\

\quad We are now ready to begin the proof of Theorem \ref{thm:principal}. In Section \ref{section:existence_inv_distribution}, we establish the first part of the theorem, i.e. (i). Then, in Section \ref{section:gen_coupling_strategy} we explain the scheme of coupling before implementing this strategy in Sections \ref{section:coupling_under_H} and \ref{section:admissibility}. Finally, in Section \ref{section:proof_thm_principal}, we achieve the proof of (ii) and (iii) of Theorem  \ref{thm:principal}.

\section{Existence of invariant distribution}\label{section:existence_inv_distribution}

Denote by $\PP_w$ the law of $(\Delta_k)_{k\leqslant0}$. Since $(\Delta_n)_{n\in\Z}$ is stationary we immediately get the following property:

\begin{prop} \label{prop:stability}
If a measure $\mu\in\mathcal{M}_1(\X\times\W)$ is such that $\Pi^*_\W\mu=\PP_w$ then 
$\Pi^*_\W\mathcal{Q}\mu=\PP_w$.
\end{prop}

We can now define the notion of Lyapunov function.

\begin{dfn}\label{def:lyapunov_function}A function
$\psi:\X\to[0,+\infty)$ is called a Lyapunov function for $\mathcal{Q}$ if $\psi$ is continuous and if the following holds:
\begin{itemize}
 \item[(i)] $\psi^{-1}([0,a])$ is compact for all $a\in[0,+\infty)$.
 \vspace{2mm}
 
 \item[(ii)] $\exists \beta>0$ and $\alpha\in(0,1)$ such that:
 \[\int_{\X\times\W}\psi(x)\mathcal{Q}\mu({\rm d}x,{\rm d}w)\leqslant\beta+\alpha\int_{\X}\psi(x)(\Pi_\X^*\mu)({\rm d}x)\]
 for all $\mu\in\mathcal{M}_1(\X\times\W)$ such that $\Pi_\W^*\mu=\PP_w$ and
$\int_{\X}\psi(x)(\Pi_\X^*\mu)({\rm d}x)<+\infty$.
\end{itemize}
\end{dfn}

The following result ensures the existence of invariant distribution for $\mathcal{Q}$.

\begin{thm}\label{thm:existence_invariant_dist}
If there exists a Lyapunov function $\psi$ for $\mathcal{Q}$, then $\mathcal{Q}$ has at least one invariant distribution $\mu_\star$, in other words $\mathcal{Q}\mu_\star=\mu_\star$. 
\end{thm}

A detailed proof of this result is given in Appendix \ref{appendix:proof_thm_existence}.
Finally, we get the first part (i) of Theorem \ref{thm:principal} about the existence of an invariant distibution by setting $\psi:=V$ (with $V$ the function appearing in $(\mathbf{H_1})$) and by saying that $\psi$ is a Lyapunov function for $\mathcal{Q}$.

\section{General coupling procedure}\label{section:gen_coupling_strategy}

We now turn to the proof of the main result of the paper, i.e. Theorem \ref{thm:principal} (ii) and (iii) about the convergence in total variation. This result is based on a coupling method first introduced in \cite{hairer2005ergodicity}, but also used in \cite{fontbona2017rate} and \cite{deya2016rate}, in a continuous time framework. The coupling strategy is slightly different in our discrete context, the following part is devoted to explain this procedure.

\subsection{Scheme of coupling}\label{subsection:coupling_scheme}

Let $(\Delta_n^1)_{n\in\Z}$ and $(\Delta_n^2)_{n\in\Z}$ be two stationary and purely non-deterministic Gaussian sequences with the following moving average representations 
\[\left\{\begin{array}{c}
\Delta^1_n=\sum\limits_{k=0}^{+\infty}a_k\xi^1_{n-k}\\
\Delta^2_n =\sum\limits_{k=0}^{+\infty}a_k\xi^2_{n-k}
\end{array}\right.\]
with
\begin{align}\label{a_k}
\left\{\begin{array}{l}
              (a_k)_{k\geqslant0}\in\R^\N ~\text{ such that }~ a_0=1 ~\text{ and }~ \sum_{k=0}^{+\infty}a_k^2<+\infty \\
              \xi^i:=(\xi^i_k)_{k\in\Z} ~\text{  an i.i.d sequence such that }\xi^i_1\sim\mathcal{N}(0,I_d)\text{ for }i=1,2.
             \end{array}\right.
\end{align}
We denote by $(X^1,X^2)$ the solution of the system:
\begin{equation}\label{syst_couplage}
\left\{\begin{array}{c}
X_{n+1}^1=F(X_n^1,\Delta_{n+1}^1)\\
X_{n+1}^2=F(X_n^2,\Delta_{n+1}^2)
\end{array}\right.
\end{equation}
with initial conditions $(X^1_0,(\Delta^1_k)_{k\leqslant0})$ and $(X^2_0,(\Delta^2_k)_{k\leqslant0})$.
We assume that $(X^2_0,(\Delta^2_k)_{k\leqslant0})\sim\mu_\star$ where $\mu_\star$ denotes a fixed invariant distribution associated to \eqref{SDS}. The previous section ensures that such a measure exists.
We define the natural filtration associated to \eqref{syst_couplage} by \[(\mathcal{F}_n)_{n\in\N}=(\sigma((\xi^1_k)_{k\leqslant n},(\xi^2_k)_{k\leqslant n},X^1_0,X^2_0))_{n\geqslant0}.\]
To lower the ``weight of the past'' at the beginning of the coupling procedure, we assume that a.s,
\[(\Delta^1_k)_{k\leqslant0}=(\Delta^2_k)_{k\leqslant0}\]
which is actually equivalent to assume that a.s $(\xi^1_k)_{k\leqslant0}=(\xi^2_k)_{k\leqslant0}$ since the invertible Toeplitz operator defined in Subsection \ref{subsection:toeplitz} links $(\Delta^i_k)_{k\leqslant0}$ to $(\xi^i_k)_{k\leqslant0}$ for $i=1,2$. Lastly, we denote by 
$(g_n)_{n\in\Z}$ and $(f_n)_{n\in\Z}$ the random variable sequences defined by
\begin{equation}\label{def_f_g}
\xi_{n+1}^1=\xi_{n+1}^2+g_n\quad\text{ and }
\quad\Delta_{n+1}^1=\Delta_{n+1}^2+f_n.
\end{equation}
They respectively represent the ``drift'' between the underlying noises $(\xi^i_k)$ and the real noises $(\Delta_k^i)$. By assumption, we have $g_n=f_n=0$ for $n<0$.

\begin{rem} From the moving average representations, we deduce immediately the following relation
for all $n\geqslant0$,
\begin{equation}\label{eq:relation_f_g}
f_n=\sum_{k=0}^{+\infty}a_kg_{n-k}=\sum_{k=0}^{n}a_kg_{n-k}.
\end{equation}
\end{rem} 

The aim is now to build $(g_n)_{n\geqslant0}$ and $(f_n)_{n\geqslant0}$ in order to stick $X^1$ and $X^2$. We set 
\begin{equation*}
\tau_\infty=\inf\{n\geqslant0~|~X^1_k=X^2_k,~\forall k\geqslant n\}.
\end{equation*}

In a purely Markovian setting, when the paths coincide at time $n$ then they remain stuck for all $k\geqslant n$ by putting the same innovation into both processes. Due to the memory this phenomenon cannot happen here. Hence, this involves a new step in the coupling scheme: try to keep the paths fathened together (see below). \\
Recall that $\mathcal{L}((X^2_k)_{k\geqslant n})=\mathcal{S}\mu_\star$. The purpose of the coupling procedure is to bound the quantity $\PP(\tau_\infty>n)$ since by a classical result we have 
\begin{equation} \|\mathcal{L}((X^1_k)_{k\geqslant n}) -\mathcal{S}\mu_\star\|_{TV}\leqslant \PP(\tau_\infty>n).
\end{equation}

Hence, we realize the coupling after a series of trials which follows three steps:

\begin{itemize}
 \item[$\ast$] \textbf{Step 1}: Try to stick the positions at a given time with a ``controlled cost''.
 \vspace{2mm}
 
 \item[$\ast$] \textbf{Step 2}: (specific to non-Markov processes) Try to keep the paths fastened together.
 
 \vspace{2mm}
 
 \item[$\ast$] \textbf{Step 3}: If Step 2 fails, we wait long enough so as to allow Step 1 to be realized with a ``controlled cost'' and with a positive probability. During this step, we assume that $g_n=0$.

\end{itemize}

More precisely, let us introduce some notations,

\begin{itemize}
 \item[$\bullet$] Let $\tau_0\geqslant0$. We begin the first trial at time $\tau_0+1$, in other words we try to stick $X^1_{\tau_0+1}$ and $X^2_{\tau_0+1}$. Hence, we assume that 
 \begin{equation}\label{eq:drift_before_coupling}
 \forall n<\tau_0,\quad g_n=f_n=0.
 \end{equation}
 \item[$\bullet$] For $j\geqslant1$, let $\tau_{j}$ denote the end of trial $j$. More specifically,
 \begin{itemize}
  \item[$\rhd$] If $\tau_{j}=+\infty$ for some $j\geqslant1$, it means that the coupling tentative has been successful.
  \item[$\rhd$] Else, $\tau_{j}$ corresponds to the end of Step 3, that is $\tau_{j}+1$ is the beginning of Step 1 of trial $j+1$.
 \end{itemize}
\end{itemize}

The real meaning of ``controlled cost'' will be clarified on Subsection \ref{subsection:adm_condition}. But the main idea is that at Step 1 of trial $j$, the ``cost'' is represented by the quantity $g_{\tau_{j-1}}$ that we need to build to get $X^1_{\tau_{j-1}+1}=X^2_{\tau_{j-1}+1}$ with positive probability. Here the cost does not only depend on the positions at time $\tau_{j-1}$ but also on all the past of the underlying noises $\xi^1$ and $\xi^2$. Hence, we must have a control on $g_{\tau_{j-1}}$ in case of failure and to this end we have to wait enough during Step 3 before beginning a new attempt of coupling.

\subsection{Coupling lemmas to achieve Step 1 and 2}\label{subsection:coupling_lem}

This section is devoted to establish coupling lemmas in order to build $(\xi^1,\xi^2)$ during Step 1 and  Step 2. 
\subsubsection{Hitting step}\label{subsubsection:hitting_step_lem}

If we want to stick $X^1$ and $X^2$ at time $n+1$, we need to build $(\xi^1_{n+1},\xi^2_{n+1})$ in order to get $F(X^1_{n},\Delta^1_{n+1})=F(X^2_{n},\Delta^2_{n+1})$ with positive probability, that is to get 
\begin{align}\label{eq:hitting_step}
F\left(X^1_n,~\xi^1_{n+1}+\sum_{k=1}^{+\infty}a_k\xi^1_{n+1-k}\right)&=F\left(X^2_n,~\xi^2_{n+1}+\sum_{k=1}^{+\infty}a_k\xi^2_{n+1-k}\right)\nonumber\\
\Longleftrightarrow\quad\tilde{F}\left(X^1_n,~\xi^1_{n+1},~\sum_{k=1}^{+\infty}a_k\xi^1_{n+1-k}\right)&=\tilde{F}\left(X^2_n,~\xi^2_{n+1},~\sum_{k=1}^{+\infty}a_k\xi^2_{n+1-k}\right).
\end{align}
The following lemma will be the main tool to achieve this goal.

\begin{lem}\label{lem:coupling_lem_step1} 
Let $K>0$ and $\mu:=\mathcal{N}(0,I_d)$.
Under the controlability assumption $(\mathbf{H_2})$, there exists $\tilde{K}>0$ (given by $(\mathbf{H_2})$), such that for every $\mathbf{x}:=(x,x',y,y')$ in $B(0,K)^4$, we can build a random variable $(Z_1,Z_2)$ with values in $(\R^d)^2$ such that 
\begin{itemize}
\item[(i)] $\mathcal{L}(Z_1)=\mathcal{L}(Z_2)=\mu$,
\item[(ii)] there exists $\delta_{\tilde{K}}>0$ depending only on $\tilde{K}$ such that
\begin{equation}\label{eq1:coupling_lem_step1}
\PP(\tilde{F}(x,Z_1,y)=\tilde{F}(x',Z_2,y'))\geqslant\PP(Z_2=\Lambda_{\mathbf{x}}(Z_1),|Z_1|\leqslant \tilde{K})\geqslant\delta_{\tilde{K}}>0
\end{equation}
where $\Lambda_{\mathbf{x}}$ is the function given by hypothesis $(\mathbf{H_2})$,
\item[(iii)] there exists $M_{K}>0$ given by $(\mathbf{H_2})$ depending only on $K$ such that
\begin{equation}\label{eq2:coupling_lem_step1}
\PP(|Z_2-Z_1|\leqslant M_{K})=1.
\end{equation}
\end{itemize} 
\end{lem}

\begin{proof} Let $\mathbf{x}:=(x,x',y,y')\in B(0,K)^4$.
First, let us denote by $\pi_1$ (resp. $\pi_2$) the projection from $\R^d\times\R^d$ to $\R^d$ of the first (resp. the second) coordinate. Introduce the two following functions defined on $\R^d$ 
\begin{align*} 
\Lambda_1: u_1 &\mapsto (u_1,\Lambda_\mathbf{x}(u_1))\\
\Lambda_2: u_2 &\mapsto (\Lambda_\mathbf{x}^{-1}(u_2),u_2)
\end{align*}
where $\Lambda_\mathbf{x}$ is the function given by $(\mathbf{H_2})$.
Now, we set 
\begin{equation*}
\bfP_1=\frac{1}{2}(\Lambda_1^*\mu\wedge\Lambda_2^*\mu).
\end{equation*}
Let us find a simplest expression for $\bfP_1$.
For every measurable function $f:\R^d\times\R^d\to\R_+$, we have 
\begin{align*}
\Lambda_1^*\mu(f)=\int_{\R^d}f(u_1,\Lambda_\mathbf{x}(u_1))\mu(du_1)=\int_{\R^d\times\R^d}f(u_1,u_2)\delta_{\Lambda_\mathbf{x}(u_1)}(du_2)\mu(du_1)
\end{align*}
and 
\begin{align*}
\Lambda_2^*\mu(f)&=\int_{\R^d}f(\Lambda_\mathbf{x}^{-1}(u_2),u_2)\mu(du_2)\\
&=\frac{1}{(2\pi)^{d/2}}\int_{\R^d}f(\Lambda_\mathbf{x}^{-1}(u_2),u_2)\exp\left(-\frac{|u_2|^2}{2}\right)du_2\\
&=\frac{1}{(2\pi)^{d/2}}\int_{\R^d}f(u_1,\Lambda_\mathbf{x}(u_1))\exp\left(-\frac{|\Lambda_\mathbf{x}(u_1)|^2}{2}\right)|\det(J_{\Lambda_\mathbf{x}}(u_1))|du_1\quad\text{ (by setting }u_1=\Lambda_\mathbf{x}^{-1}(u_2)) \\
&=\int_{\R^d}f(u_1,\Lambda_\mathbf{x}(u_1))\underbrace{\exp\left(\frac{|u_1|^2}{2}-\frac{|\Lambda_\mathbf{x}(u_1)|^2}{2}\right)|\det(J_{\Lambda_\mathbf{x}}(u_1))|}_{=:D_{\Lambda_\mathbf{x}}(u_1)}\mu(du_1)\\
&=\int_{\R^d\times\R^d}f(u_1,u_2)\delta_{\Lambda_\mathbf{x}(u_1)}(du_2)D_{\Lambda_\mathbf{x}}(u_1)\mu(du_1).\\
\end{align*}

By construction, we then have

\begin{equation}\label{measure_P1}
\bfP_1(du_1,du_2)=\frac{1}{2}\delta_{\Lambda_\mathbf{x}(u_1)}(du_2)(D_{\Lambda_\mathbf{x}}(u_1)\wedge1)\mu(du_1).
\end{equation}

Write $S(u_1,u_2)=(u_2,u_1)$ and denote by $\tilde{\bfP}_1$ the ``symmetrized'' non-negative measure induced by $\bfP_1$, 
\begin{equation}\label{measure_P1_tilde}
\tilde{\bfP}_1=\bfP_1+S^*\bfP_1.
\end{equation}
We then define $(Z_1,Z_2)$ as follows:
\begin{equation}\label{couple_law}
\mathcal{L}(Z_1,Z_2)=\tilde{\bfP}_1+\Delta^*(\mu-\pi_1^*\tilde{\bfP}_1)=\bfP_1+\bfP_2
\end{equation}
with $\Delta(u)=(u,u)$ and $\bfP_2=S^*\bfP_1+\Delta^*(\mu-\pi_1^*\tilde{\bfP}_1)$. It remains to prove that $\mathcal{L}(Z_1,Z_2)$ is well defined and satisfies all the properties required by the lemma.\\

\textbf{First step:} Prove that $\bfP_2$ is the sum of two non-negative measures.\\
Using \eqref{measure_P1}, we can check that for all non-negative function f,
\[\pi_1^*\bfP_1(f)\leqslant\frac{1}{2}\mu(f)\]
and 
\[\pi_2^*\bfP_1(f)=\pi_1^*(S^*\bfP_1)(f)\leqslant\frac{1}{2}\mu(f).\] 
By adding the two previous inequalities, we deduce that the measure $\mu-\pi_1^*\tilde{\bfP}_1$ is non-negative. This concludes the first step.\\

\textbf{Second step:} Prove that $\pi_1^*(\bfP_1+\bfP_2)=\pi_2^*(\bfP_1+\bfP_2)=\mu$.\\
This fact is almost obvious. We just need to use the fact that \[\pi_1\circ\Delta=\pi_2\circ\Delta={\rm Id}\] and the symmetry property of $\tilde{\bfP}_1$, \[\text{i.e. }~ \pi_1^*\tilde{\bfP}_1=\pi_2^*\tilde{\bfP}_1.\]

\textbf{Third step:} Prove \eqref{eq1:coupling_lem_step1} and \eqref{eq2:coupling_lem_step1}.\\
Let us first remark that the support of $\bfP_1+\bfP_2$ is included in
\[\{(u,v)\in\R^d\times\R^d~|~v=\Lambda_\mathbf{x}(u)\}\cup\{(u,v)\in\R^d\times\R^d~|~v=\Lambda_\mathbf{x}^{-1}(u)\}\cup\{(u,v)\in\R^d\times\R^d~|~v=u\}.\]
Therefore, by \eqref{ineq:lambda_unif_upperbound} in $(\mathbf{H_2})$ and the fact that 
\[(\forall u\in\R^d,~|\Lambda_\mathbf{x}(u)-u|\leqslant M_{K})~\Longleftrightarrow~(\forall u\in\R^d,~|\Lambda_\mathbf{x}^{-1}(u)-u|\leqslant M_{K})\]
since $\Lambda_\mathbf{x}$ is invertible on $\R^d$,
we finally get \eqref{eq2:coupling_lem_step1}.\\
Then, using again $(\mathbf{H_2})$ where $\tilde{K}$ is defined and the definition of the subprobability $\bfP_1$ we get 
\begin{equation}
\PP(\tilde{F}(x,Z_1,y)=\tilde{F}(x',Z_2,y'))\geqslant\underbrace{\bfP_1(B(0,\tilde{K})\times\Lambda(B(0,\tilde{K})))}_{=\PP(Z_2=\Lambda_{\mathbf{x}}(Z_1),|Z_1|\leqslant \tilde{K})}
\end{equation} 
and \[\PP(Z_2=\Lambda_{\mathbf{x}}(Z_1),|Z_1|\leqslant \tilde{K})=\frac{1}{2}\int_{B(0,\tilde{K})}(D_{\Lambda_\mathbf{x}}(u)\wedge1)\mu(du).\]

It just remains to use \eqref{ineq:lambda_jacobian_bound} and \eqref{ineq:lambda_unif_upperbound} in $(\mathbf{H_2})$ to conclude.
Indeed,
\begin{align*}
\PP(Z_2=\Lambda_{\mathbf{x}}(Z_1),|Z_1|\leqslant \tilde{K})=&~\frac{1}{2}\int_{B(0,\tilde{K})}\left(\exp\left(\frac{|u|^2}{2}-\frac{|\Lambda_\mathbf{x}(u)|^2}{2}\right)|\det(J_{\Lambda_\mathbf{x}}(u))|\right)\wedge1~\mu(du)\\
\geqslant&~\frac{1}{2}\mu(B(0,\tilde{K}))\left[\left(\exp\left(-\frac{(M_{\tilde{K}}+\tilde{K})^2}{2}\right)C_{\tilde{K}}\right)\wedge1\right]=:\delta_{\tilde{K}}>0~\\
\end{align*}
which concludes the proof.
\end{proof}

\subsubsection{Sticking step}

Now, if the positions $X^1_{n+1}$ and $X^2_{n+1}$ are stuck together, we want that they remain fastened together for all $k>n+1$ which means that:
\begin{align}\label{eq1:relation_step2}
&\forall k\geqslant n+1,\quad F(X^1_{k},\Delta^1_{k+1})=F(X^2_{k},\Delta^2_{k+1})\nonumber\\
\Longleftrightarrow\quad &\forall k\geqslant n+1,\quad F(X^1_{k},\xi^1_{k+1}+\sum_{l=1}^{+\infty}a_l\xi^1_{k+1-l})=F(X^1_{k},\xi^2_{k+1}+\sum_{l=1}^{+\infty}a_l\xi^2_{k+1-l})
\end{align}
since $X^1_{k}=X^2_{k}$.
Recall that for all $k\in\Z$, $g_{k}=\xi^1_{k+1}-\xi^2_{k+1}$ is the drift between the underlying noises. Then, if we have 
\begin{align}\label{eq2:relation_step2}
&\forall k\geqslant n+1,\quad \xi^1_{k+1}+\sum_{l=1}^{+\infty}a_l\xi^1_{k+1-l}=\xi^2_{k+1}+\sum_{l=1}^{+\infty}a_l\xi^2_{k+1-l}\nonumber\\
\Longleftrightarrow\quad&\forall k\geqslant n+1,\quad g_k=-\sum_{l=1}^{+\infty}a_lg_{k-l},
\end{align}
the identity \eqref{eq1:relation_step2} is automatically satisfied. 
\begin{rem} The successful $g_k$ defined by relation \eqref{eq2:relation_step2} is $\mathcal{F}_k$-measurable. This explains why we chose to index it by $k$ even if it represents the drift between $\xi^1_{k+1}$ and $\xi^2_{k+1}$. 
\end{rem}

Hence, we will try to get \eqref{eq2:relation_step2} on successive finite intervals to finally get a bound on the successful-coupling probability. The size choice of those intervals will be important according to the hypothesis $(\mathbf{H_{poly}})$ or $(\mathbf{H_{exp}})$ that we made. The two next results will be our tools to get \eqref{eq2:relation_step2} on Subsection \ref{subsection:lower-bound_coupling_proba}.
For the sake of simplicity we set out these results on $\R$. On $\R^d$ we just have to apply them on every marginal. Lemma \ref{lem:coupling_step2} is almost the statement of Lemma 5.13 of \cite{hairer2005ergodicity} or Lemma 3.2 of \cite{fontbona2017rate}.

\begin{lem}\label{lem:coupling_step2}
Let $\mu:=\mathcal{N}(0,1)$. Let $a\in\R$, $b\geqslant |a|$ and $M_b:=\max(4b,-2\log(b/8))$.  
\begin{itemize}
\item[(i)]
For all $b\geqslant|a|$, there exist $\delta_b^1$ and $\delta_b^2\in(0,1)$, such that we can build a probability measure $\mathcal{N}_{a,b}^2$ on $\R^2$ with every marginal equal to $\mu$ and such that:
 \[\mathcal{N}_{a,b}^2(\{(x,y)~|~y=x+a\})\geqslant\delta_b^1\quad\text{ and }\quad
 \mathcal{N}_{a,b}^2(\{(x,y)~|~|y-x|\leqslant M_b\})=1.\]
 \item[(ii)] Moreover, if $b\in(0,1)$, the previous statement holds with $\delta_b^1=1-b$.
\end{itemize}
\end{lem}

The following corollary is an adapted version of Lemma 3.3 of \cite{fontbona2017rate} to our discrete context.

\begin{cor}\label{cor:coupling_step2}
Let $T>0$ be an integer, $b>0$, $g=(g_0,g_1,\dots,g^{}_{T})\in\R^{T+1}$ such that $\|g\|\leqslant b$  where $\|.\|$ is the euclidian norm on $\R^{T+1}$ and set
$M_b:=\max(4b,-2\log(b/8))$. 
\begin{itemize}
 \item[(i)] 
Then, there exists $\delta_b^1\in(0,1)$, for which we can build a random variable $((\xi^1_{k+1})_{k\in\llbracket0,T\rrbracket},(\xi^2_{k+1})_{k\in\llbracket0,T\rrbracket})$ with values in $(\R^{T+1})^2$, with marginal distribution $\mathcal{N}(0,I_{T+1})$ and satisfying:
\[\PP\left(\xi^1_{k+1}=\xi^2_{k+1}+g_{k}~~\forall k\in\llbracket0,T\rrbracket\right)\geqslant\delta_b^1\]
and
\[\PP\left(\|\xi^1-\xi^2\|\leqslant M_b\right)=1.\] 
 \item[(ii)] Moreover, if $b\in(0,1)$, the previous statement holds with $\delta_b^1=1-b$.
\end{itemize}
\end{cor}

\begin{proof}
Let $(u_k)_{k\in\llbracket0,T\rrbracket}$ be an orthonormal basis of $\R^{T+1}$ with $u_0=\frac{g}{\|g\|}$.
We denote by $(U_1,U_2)$ a random variable which has distribution $\mathcal{N}^2_{a,b}$ (with $a=\|g\|$) given in the lemma \ref{lem:coupling_step2}.
Let $(\varepsilon_k)_{k\in\llbracket1,T\rrbracket}$ be an iid random variable sequence with $\varepsilon_1\sim\mathcal{N}(0,1)$ and independent from $(U_1,U_2)$.
Then, for $i=1,2$ we define the isometry:
\begin{equation}\label{eq:W^i}
\begin{array}{llll}
\W^i:& \R^{T+1}&\to&\W^i(\R^{T+1})\subset L^2(\Omega,\mathcal{F},\PP)\\
                           &u_0&\mapsto& U_i\\
                           &u_k&\mapsto&\varepsilon_k ~\text{ for }~k\in\llbracket1,T\rrbracket.
       \end{array}
\end{equation}
And we set for all $n\in\llbracket0,T\rrbracket$, $\xi^i_{n+1}:=\W^i(e_n)$ where $e_n$ is the vector of $\R^{T+1}$ for which every coordinate is $0$ except the $(n+1)^{\text{th}}$ which is $1$.
Since $(u_k)_{k\in\llbracket0,T\rrbracket}$ is an orthonormal basis of $\R^{T+1}$, we then have:
\[e_n=\sum_{k=0}^T\langle e_n,u_k\rangle u_k.\]
Hence,
\begin{equation*}
\xi^i_{n+1}=\W^i\left(\sum_{k=0}^T\langle e_n,u_k\rangle u_k\right)
        =U_i\frac{g_{n}}{\|g\|}+\sum_{k=1}^T\langle e_n,u_k\rangle\varepsilon_k.
\end{equation*}
$\xi^i_{n+1}$ is clearly centered and Gaussian as a linear combination of independent centered Gaussian random variables and using that $\W^i$ is an isometry, we get that $(\xi^i_{k+1})_{k\in\llbracket0,T\rrbracket}$ has distribution $\mathcal{N}(0,I_{T+1})$ for $i=1,2$.
Therefore, we built $\xi^1$ and $\xi^2$ as anounced. Indeed, by Lemma \ref{lem:coupling_step2}
\[\PP\left(\xi^1_{n+1}=\xi^2_{n+1}+g_{n}~~\forall n\in\llbracket0,T\rrbracket\right)=
\PP\left(U_1=U_2+\|g\|\right)
\geqslant\delta_b^1\]
and
\[\PP\left(\|\xi^1-\xi^2\|\leqslant M_b\right)=\PP(|U_1-U_2|\leqslant M_b)=1.\]
$(ii)$ also follows immediately from Lemma \ref{lem:coupling_step2}.

\end{proof}

\section{Coupling under $(\mathbf{H_{poly}})$ or $(\mathbf{H_{exp}})$}\label{section:coupling_under_H}

We can now move on the real coupling procedure to finally get a lower-bound for the successful-coupling probability. In a first subsection, we explain exactly what we called ``controlled cost'' and in a second subsection we spell out our bound.

\subsection{Admissibility condition}\label{subsection:adm_condition}

The ``controlled cost'' is called ``admissibility'' in \cite{hairer2005ergodicity}. Here, we will talk about $(K,\alpha)$-admissibility, as in \cite{fontbona2017rate}, but in the following sense:

\begin{dfn} Let $K>0$ and $\alpha>0$ be two constants and $\tau$ a random variable with values in $\N$. We say that the system is $(K,\alpha)$-admissible at time $\tau$
if $\tau(\omega)<+\infty$ and if\\ $(X^1_{\tau}(\omega),X^2_\tau(\omega),(\xi^1_n(\omega),\xi^2_n(\omega))_{n\leqslant\tau})$ satisfies 
\begin{equation}\label{eq:adm_1}
\forall n\geqslant0,\left|\sum_{k=n+1}^{+\infty}a_kg_{\tau+n-k}(\omega)\right|\leqslant v_n
\end{equation} 
and
\begin{equation}\label{eq:adm_2}
\left|X^i_\tau(\omega)\right|\leqslant K,\quad\left|\sum_{k=1}^{+\infty}a_{k}\xi^i_{\tau+1-k}(\omega)\right|\leqslant K~~ \text{ for } i=1,2
\end{equation}
with \begin{equation}\label{eq:adm_speed}
v_n=(n+1)^{-\alpha}\text{ under $(\mathbf{H_{poly}})$}\quad\text{ and }\quad v_n=e^{-\alpha n}\text{ under $(\mathbf{H_{exp}})$.}
\end{equation}
\end{dfn}

\begin{rem} On the one hand, condition \eqref{eq:adm_1} measures the distance between the past of the noises (before time $\tau$). On the other hand, condition \eqref{eq:adm_2} has two parts: the first one ensures that at time $\tau$ both processes are not far from each other and the second part is a constraint on the memory part of the Gaussian noise $\Delta^i_{\tau+1}$.
\end{rem}

The aim is to prove that under those two conditions, the coupling will be successful with a probability lower-bounded by a positive constant. To this end, we will need to ensure that at every time $\tau^{}_j$, the system will be $(K,\alpha)$-admissible with positive probability.
We set:
\begin{equation}\label{set:adm1}
\Omega^1_{\alpha,\tau}:=\left\{\omega,\tau(\omega)<+\infty,\left|\sum_{k=n+1}^{+\infty}a_kg_{\tau+n-k}(\omega)\right|\leqslant v_n\quad\forall n\in\N\right\}
\end{equation}
and
\begin{equation}\label{set:adm2}
\Omega^2_{K,\tau}:=\left\{\omega,\tau(\omega)<+\infty,\left|X^i_\tau(\omega)\right|\leqslant K\text{ and }\left|\sum_{k=1}^{+\infty}a_{k}\xi^i_{\tau+1-k}(\omega)\right|\leqslant K~~ \text{ for } i=1,2\right\}.
\end{equation}
We define
\begin{equation}\label{set:adm}
\Omega_{K,\alpha,\tau}=\Omega^1_{\alpha,\tau}\cap\Omega^2_{K,\tau}.
\end{equation}
If $\omega\in\Omega_{K,\alpha,\tau}$, we will try to couple at time $\tau+1$. Otherwise, we say that Step 1 fails and one begins Step 3.
Hence, Step 1 of trial $j$ has two ways to fail: either $\omega$ belongs to $\Omega^c_{K,\alpha,\tau^{}_{j-1}}$ and one moves directly to Step 3 or $\omega$ belongs to $\Omega_{K,\alpha,\tau_{j-1}}$, one tries to couple and it fails.

\subsection{Lower-bound for the successful-coupling probability}\label{subsection:lower-bound_coupling_proba}

The main purpose of this subsection is to get a positive lower-bound for the successful-coupling probability which will be independent of $j$ (the number of the tentative), in other words we want to prove the following proposition 

\begin{propo}\label{prop:lower-bound_success_proba} Assume $(\mathbf{H_{1}})$ and $(\mathbf{H_{2}})$.
Let $K>0$, $\alpha>~\frac{1}{2}\vee\left(\frac{3}{2}-\beta\right)$ if we are under $(\mathbf{H_{poly}})$ and $\alpha>0$ different from $\zeta$ if we are under $(\mathbf{H_{exp}})$. In both cases, there exists $\delta_0$ in $(0,1)$ such that for all $j\geqslant1$,
\begin{equation}
\delta_0\leqslant\PP(\Delta\tau^{}_j=+\infty|\Omega_{K,\alpha,\tau^{}_{j-1}})
\end{equation}
where $\Delta\tau^{}_j:=\tau_j-\tau_{j-1}$ and $\tau_j$ is defined in Subsection \ref{subsection:coupling_scheme} as the end of trial $j$.\\
Moreover, we can choose $\delta_1\in(0,1)$ such that
\begin{equation}\label{ineq:lower-bound_fail_proba}
\forall j\geqslant1,\quad \delta_1\leqslant\PP(\tau_j<\infty|\tau_{j-1}<\infty).
\end{equation}
\end{propo}

The second part of Proposition \ref{prop:lower-bound_success_proba} may appear of weak interest but will be of first importance in Subsection \ref{subsection:compact_return}.

\subsubsection{Step 1 (hitting step)}

\begin{lem}\label{lem:success_proba_step1}
Let $K>0$ and $\alpha>0$. Assume $(\mathbf{H_1})$ and $(\mathbf{H_2})$. Let $\tilde{K}>0$ be the constant appearing in $(\mathbf{H_2})$, $\delta_1\in(0,1)$ and $\tau$ be a stopping time with respect to $(\mathcal{F}_n)_{n\in\Z}$ such that $\PP(\Omega_{K,\alpha,\tau})>0$.\\
We can build $(\xi^1_{\tau+1},\xi^2_{\tau+1})$ with $\xi^1_{\tau+1}\sim\mathcal{N}(0,I_d)$
and $\xi^2_{\tau+1}\sim\mathcal{N}(0,I_d)$ such that
\begin{itemize}
\item[(i)] There exist $K_1\in(0,\tilde{K}]$ and $\delta_{K_1}\in(0,1)$ such that
\begin{equation}\label{eq:success_proba_step1}
\PP(X^1_{\tau+1}=X^2_{\tau+1}|\Omega_{K,\alpha,\tau})\geqslant
\PP(\xi^2_{\tau+1}=\Lambda_\mathbf{x}(\xi^1_{\tau+1}),~|\xi^1_{\tau+1}|\leqslant K_1~|\Omega_{K,\alpha,\tau})\geqslant\delta_{K_1}>0
\end{equation}
and 
\begin{equation}\label{eq:fail_proba_step1}
\PP\left(\Omega_{K,\alpha,\tau}^c\cup\left(\{\xi^2_{\tau+1}\neq\Lambda_\mathbf{x}(\xi^1_{\tau+1})~\text{ or }~|\xi^1_{\tau+1}|> K_1\}\cap\Omega_{K,\alpha,\tau}\right)\right)\geqslant\delta_1
\end{equation}
where $\mathbf{x}:=\left(X^1_\tau,~X^2_\tau,~\sum_{k=0}^{+\infty}a_k\xi^1_{\tau+1-k},~\sum_{k=0}^{+\infty}a_k\xi^2_{\tau+1-k}\right)$ and $\Lambda_\mathbf{x}$ comes from $(\mathbf{H_2})$.
\item[(ii)] There exists $M_{K}>0$ such that
\[\left|g_{\tau}\right|=|\xi^1_{\tau+1}-\xi^2_{\tau+1}|\leqslant M_{K}\quad a.s.\]
\end{itemize}
\end{lem}

\begin{rem} The constant $\delta_1$ is chosen independently from $K$ and $\alpha$.
\end{rem}

Before proving this result, let us explain a bit why we add the lower-bound \eqref{eq:fail_proba_step1}. As we already said, we will see further (in Subsection \ref{subsection:compact_return}) that we need the (uniform) bound on the failure-coupling probability given in \eqref{ineq:lower-bound_fail_proba}. Therefore, for every $j\geqslant1$, we will consider that Step 1 of trial $j$ fails if and only if $\omega\in\Omega_{K,\alpha,\tau_{j-1}}^c\cup\left(\{\xi^2_{\tau_{j-1}+1}\neq\Lambda_\mathbf{x}(\xi^1_{\tau_{j-1}+1})~\text{ or }~|\xi^1_{\tau_{j-1}+1}|> K_1\}\cap\Omega_{K,\alpha,\tau_{j-1}}\right)$ and in this case one immediatly begins Step 3. Hence, for all $j\geqslant1$, thanks to Lemma \ref{lem:success_proba_step1} we get the existence of $K_1$ such that:
\begin{align*}\PP(\tau_j<&\infty|\tau_{j-1}<\infty)\\
&\geqslant\PP\left(\Omega_{K,\alpha,\tau_{j-1}}^c\cup\left(\{\xi^2_{\tau_{j-1}+1}\neq\Lambda_\mathbf{x}(\xi^1_{\tau_{j-1}+1})~\text{ or }~|\xi^1_{\tau_{j-1}+1}|> K_1\}\cap\Omega_{K,\alpha,\tau_{j-1}}\right)\right)\geqslant\delta_1
\end{align*}
and then \eqref{ineq:lower-bound_fail_proba} derives from Lemma \ref{lem:success_proba_step1}.
 This construction may seem artificial but it is necessary to prove Proposition \ref{propo:compact_return}. Moreover, this has no impact on the computation of the rate of convergence to equilibrium since it only affects Step 1. We can now move on the proof of Lemma \ref{lem:success_proba_step1}.
 
\begin{proof}
(i) Set $\mathbf{x}:=\left(X^1_\tau,~X^2_\tau,~\sum_{k=1}^{+\infty}a_k\xi^1_{\tau+1-k},~\sum_{k=1}^{+\infty}a_k\xi^2_{\tau+1-k}\right)$. Conditionnally to $\Omega_{K,\alpha,\tau}$
we have $\mathbf{x}\in B(0,K)^4$ and we can build $(Z_1,Z_2)$ as in Lemma \ref{lem:coupling_lem_step1}.
Let $\xi\sim\mathcal{N}(0,1)$ be independent from $(Z_1,Z_2)$ and set
\begin{equation}\label{eq:def_derive_step1}
(\xi^1_{\tau+1},\xi^2_{\tau+1})=(\mathds{1}_{\Omega_{K,\alpha,\tau}}Z_1+\mathds{1}_{\Omega^c_{K,\alpha,\tau}}\xi,\quad\mathds{1}_{\Omega_{K,\alpha,\tau}}Z_2+
\mathds{1}_{\Omega^c_{K,\alpha,\tau}}\xi).
\end{equation}
Therefore, we deduce by Lemma \ref{lem:coupling_lem_step1} and its proof that for all $K_1\in(0,\tilde{K}]$,
\begin{equation}
\PP(X^1_{\tau+1}=X^2_{\tau+1}|\Omega_{K,\alpha,\tau})\geqslant
\underbrace{\PP(Z_2=\Lambda_\mathbf{x}(Z_1),~|Z_1|\leqslant K_1~|\Omega_{K,\alpha,\tau})}_{=\PP(\xi^2_{\tau+1}=\Lambda_\mathbf{x}(\xi^1_{\tau+1}),~|\xi^1_{\tau+1}|\leqslant K_1~|\Omega_{K,\alpha,\tau})}\geqslant\delta_{K_1}>0.
\end{equation}
And the first part of (i) is proven. It remains to choose the good $K_1\in(0,\tilde{K}]$ to get the second part. Set $p_K:=\PP(\Omega_{K,\alpha,\tau})$ and $\mu:=\mathcal{N}(0,I_d)$, then
\begin{align*}
\PP&\left(\Omega_{K,\alpha,\tau}^c\cup\left(\{\xi^2_{\tau+1}\neq\Lambda_\mathbf{x}(\xi^1_{\tau+1})~\text{ or }~|\xi^1_{\tau+1}|> K_1\}\cap\Omega_{K,\alpha,\tau}\right)\right)\\
&\quad\quad=
1-p_K+p_K\PP\left(\{\xi^2_{\tau+1}\neq\Lambda_\mathbf{x}(\xi^1_{\tau+1})~\text{ or }~|\xi^1_{\tau+1}|> K_1\}|\Omega_{K,\alpha,\tau}\right)\\
&\quad\quad\geqslant 1-p_K+p_K\PP(|\xi^1_{\tau+1}|> K_1|\Omega_{K,\alpha,\tau})\\
&\quad\quad\geqslant 1-p_K+p_K\mu(B(0,K_1)^c)=1-p_K+p_K(1-\mu(B(0,K_1))
\end{align*}
where the last inequality is due to Lemma \ref{lem:coupling_lem_step1} one more time.
Finally, it remains to choose $K_1\in(0,\tilde{K}]$ small enough in order to get $1-p_K+p_K(1-\mu(B(0,K_1))\geqslant\delta_1$.\\

(ii) If $\omega\in\Omega_{K,\alpha,\tau}$, by the previous construction and Lemma \ref{lem:coupling_lem_step1}, 
we have $|g_\tau(\omega)|=|Z_1(\omega)-Z_2(\omega)|\leqslant M_{K}$. And if $\omega\in\Omega_{K,\alpha,\tau}^c$ then $|g_\tau(\omega)|=|\xi(\omega)-\xi(\omega)|=0$ which concludes the proof of (ii).
\end{proof}

\vspace{2mm}

To fix the ideas let us recall what we mean by ``success of Step 1'' and ``failure of Step 1'' of trial $j$ ($j\geqslant1$)~:
\begin{align}
&\{\text{success of Step 1}\}=\Omega_{K,\alpha,\tau_{j-1}}\cap\{\xi^2_{\tau_{j-1}+1}=\Lambda_\mathbf{x}(\xi^1_{\tau_{j-1}+1}),~|\xi^1_{\tau_{j-1}+1}|\leqslant K_1\}\label{event:success_step1}\\
&\{\text{failure of Step 1}\}=\Omega_{K,\alpha,\tau_{j-1}}^c\cup\left(\{\xi^2_{\tau_{j-1}+1}\neq\Lambda_\mathbf{x}(\xi^1_{\tau_{j-1}+1})~\text{ or }~|\xi^1_{\tau_{j-1}+1}|> K_1\}\cap\Omega_{K,\alpha,\tau_{j-1}}\right)\label{event:failure_step1}
\end{align}
where $\mathbf{x}:=\left(X^1_{\tau_{j-1}},~X^2_{\tau_{j-1}},~\sum_{k=0}^{+\infty}a_k\xi^1_{\tau_{j-1}+1-k},~\sum_{k=0}^{+\infty}a_k\xi^2_{\tau_{j-1}+1-k}\right)$.

\subsubsection{Step 2 (sticking step)}\label{subsection:step2}

Step 2 of trial $j$ consists in trying to keep the paths fastened together on successive intervals $I_{j,\ell}$. More precisely, during trial $j$, we set
\begin{align}\label{def:step2_intervals}
&I_{j,0}:=\{\tau^{}_{j-1}+1\},\quad I_{j,1}:=\llbracket\tau^{}_{j-1}+2,\tau^{}_{j-1}+2c_2-1\rrbracket\nonumber\\
\text{ and }\quad\quad&\forall\ell\geqslant2,\quad I_{j,\ell}:=\left\llbracket\tau^{}_{j-1}+c_2s_\ell,\tau^{}_{j-1}+c_2s_{\ell+1}-1\right\rrbracket
\end{align}
where $c_2\geqslant2$ will be chosen further and with
\begin{equation}\label{eq:interval_size}
\forall\ell\geqslant2,\quad s_\ell=\left\{\begin{array}{cc}
2^\ell &\text{ under } (\mathbf{H_{poly}})\\
\ell &\text{ under } (\mathbf{H_{exp}}).
\end{array}\right.
\end{equation}

We denote \begin{equation}\label{def:failure_interval_number}
\ell^*_j:=\sup\{\ell\geqslant1~|~\forall n\in I_{j,\ell-1},~g_{n-1}=g^{(s)}_{n-1}\}
\end{equation} 
where $g^{(s)}_{n-1}$ is the successful-coupling drift defined by
\eqref{eq2:relation_step2}, i.e. $g^{(s)}_{n-1}=-\sum_{l=1}^{+\infty}a_lg_{n-1-l}$. In other words, $I_{j,\ell^*_j}$ is the interval where the failure occurs. If $\{\ell\geqslant1|\forall n\in I_{j,\ell-1},~g_{n-1}=g^{(s)}_{n-1}\}=\emptyset$, we adopt the convention $\ell^*_j=0$, it corresponds to the case where the failure occurs at Step 1.
When $\ell^*_j=+\infty$, trial $j$ is successful.
For a given positive $\alpha$ and $K>0$, we set
\begin{equation}\label{event_fail_after_l_attempts}
\mathcal{B}_{j,\ell}:=\Omega_{K,\alpha,\tau^{}_{j-1}}\cap\{\ell^*_j>\ell\}\quad\forall j\geqslant1,~\ell\geqslant0,
\end{equation}
which means that failure of Step 2 may occur at most after $\ell$ trials. With this notations we get 
\begin{equation}\label{eq:new_expr_success_proba}
\PP(\Delta\tau^{}_j=+\infty|\Omega_{K,\alpha,\tau^{}_{j-1}})=\PP(\text{success of Step 1 }|\Omega_{K,\alpha,\tau^{}_{j-1}})\prod_{\ell=1}^{+\infty}\PP(\mathcal{B}_{j,\ell}|\mathcal{B}_{j,\ell-1})
\end{equation}

where the event $\{\text{success of Step 1}\}$ is defined by \eqref{event:success_step1}.

\begin{rem} There is an infinite product in this expression of the successful-coupling probability. Hence, the size choice of the intervals $I_{j,\ell}$ defined in \eqref{eq:interval_size} will play a significant role in the convergence of the product to a positive limit. 
\end{rem}

In the following lemma, similarly to the above definitions, we consider for a stopping time $\tau$ the intervals $(I_{\tau,\ell})_{\ell\geqslant1}$, the integer $\ell^*_\tau$ and the events $\mathcal{B}_{\tau,\ell}$, replacing $\tau^{}_{j-1}$ by $\tau$.

\begin{lem}\label{lem:success_proba_step2}
Let $K>0$, assume $(\mathbf{H_1})$ and $(\mathbf{H_2})$. Let $\alpha>\frac{1}{2}\vee\left(\frac{3}{2}-\beta\right)$ under $(\mathbf{H_{poly}})$ or $\alpha>0$ different from $\zeta$ under $(\mathbf{H_{exp}})$. Let $\tau$ be a stopping time with respect to $(\mathcal{F}_n)_{n\in\Z}$ (defined in Subsection \ref{subsection:coupling_scheme}) and assume that the system is $(K,\alpha)$-admissible at time $\tau$, then there exists $C_K>0$ such that for $c_2\geqslant2$ large enough the successful drift $g^{(s)}$ satisfies
\[\text{for }\ell=1,\quad \|g^{(s)}\|_{I_{\tau,1}}\leqslant C_K\]
and
\[\forall\ell\geqslant2,\quad\|g^{(s)}\|_{I_{\tau,\ell}}=\left(\sum_{k=\tau+c_2s_\ell}^{\tau+c_2s_{\ell+1}-1}\left|g^{(s)}
_{k-1}\right|^2\right)^{1/2}\leqslant 2^{-\tilde{\alpha}\ell}\]
where $\tilde{\alpha}:=\left\{\begin{array}{ccc}
\min\{\alpha,\beta,\alpha+\beta-1\}-1/2-\varepsilon & \text{for all }\varepsilon>0 &\text{ under } (\mathbf{H_{poly}})\\
\min(\alpha,\zeta) & &\text{ under } (\mathbf{H_{exp}})
\end{array}\right.$.\\
Therefore, for all $\ell\geqslant1$, we can build thanks to Corollary \ref{cor:coupling_step2}
$\left((\xi^1_k)_{k\in I_{\tau,\ell}},(\xi^2_k)_{k\in I_{\tau,\ell}}\right)$ during Step 2 in such a way that
\[\PP(\mathcal{B}_{\tau,1}|\mathcal{B}_{\tau,0})\geqslant\delta_{K}^1
\quad\text{ and }\quad\forall\ell\geqslant2,\quad 
\PP(\mathcal{B}_{\tau,\ell}|\mathcal{B}_{\tau,\ell-1})\geqslant1-2^{-\tilde{\alpha}\ell}\]
where $\delta_{K}^1\in(0,1)$.\\
Moreover, if $2\leqslant\ell^*_\tau<+\infty$, there exists $C_\alpha>0$ independent from $K$ such that
\[\left(\sum_{k=\tau+c_2s_{\ell^*_\tau}}^{\tau+c_2s_{\ell^*_\tau+1}-1}|g_{k-1}|^2\right)^{1/2}\leqslant C_\alpha(\ell^*_\tau+1)\]
and if $\ell^*_\tau=1$, $\|g\|_{I_{\tau,1}}\leqslant C'_K$ for some constant $C'_K>0$.
\end{lem}

\begin{rem} $\rhd$ Under hypothesis $(\mathbf{H_{poly}})$ the condition $\alpha>\frac{1}{2}\vee\left(\frac{3}{2}-\beta\right)$ will ensure that \\$\min\{\alpha,\beta,\alpha+\beta-1\}-1/2>0$. \\
$\rhd$ In the polynomial case, for technical reasons $\tilde{\alpha}$ depends on $\varepsilon>0$.
This expression allows us to put together different cases and simplify the lemma. Indeed, if $(\alpha,\beta)\notin\{1\}\times(0,1]\cup(0,1]\times\{1\}$, we can take $\varepsilon=0$. 
\end{rem}

To prove this lemma we will use in the \textit{polynomial case} the following technical result which a more precise statement and a proof are given in Appendix \ref{appendix:proof_technical_lem}.

\begin{lem}[Technical lemma] \label{lem:technical}
Let $\alpha>0$ and $\beta>0$ such that $\alpha+\beta>1$. Then, there exists $C(\alpha,\beta)>0$ such that for every $\varepsilon>0$,
\[\forall n\geqslant0,\quad\sum_{k=0}^{n}(k+1)^{-\beta}(n+1-k)^{-\alpha}\leqslant C(\alpha,\beta)~
(n+1)^{-\min\{\alpha,\beta,\alpha+\beta-1\}+\varepsilon}.\]
When $(\alpha,\beta)\notin\{1\}\times(0,1]\cup(0,1]\times\{1\}$, we can take $\varepsilon=0$ in the previous inequality.
\end{lem}

We can now move on the proof of Lemma \ref{lem:success_proba_step2}.

\begin{proof} Let us prove the first part of the lemma, namely the upper-bound of the $\ell^2$ norm for the successful-coupling drift term on the intervals $I^{\tau}_\ell$. Indeed, the second part is just an application of Corollary \ref{cor:coupling_step2}.
Since the system is $(K,\alpha)$-admissible at time $\tau$, we get by \eqref{eq:adm_1}
\[\forall n\geqslant0,\left|\sum_{k=n+1}^{+\infty}a_kg_{\tau+n-k}\right|\leqslant v_n.\]
But, if Step 2 is successful, we recall that by \eqref{eq2:relation_step2} the successful drift satisfies
$~g^{(s)}_{\tau+n}=-\sum_{k=1}^{+\infty}a_kg_{\tau+n-k}$ for all $n\geqslant1$,
hence 
\[\forall n\geqslant1,~\sum_{k=0}^{n}a_kg^{(s)}_{\tau+n-k}=\underbrace{-\sum_{k=n+1}^{+\infty}a_kg_{\tau+n-k}}_{=:u_n}\quad\text{ and we set }\quad
u_0:=g^{(s)}_{\tau}.\]
Therefore thanks to Remark \ref{rem:reverse_relation_b}, this is equivalent to 
$~g^{(s)}_{\tau+n}=\sum_{k=0}^{n}b_ku_{n-k}$ for all $n\geqslant0$.
By the admissibility assumption, we have $\forall k\in\{0,\dots,n-1\},~ |u_{n-k}|\leqslant v_{n-k}$ and by Lemma \ref{lem:success_proba_step1} (ii), $|u_0|\leqslant M_{K}$.
Hence, we get
\begin{align}\label{eq:proof_lem_proba_success_step2}
\forall n\geqslant0,~|g^{(s)}_{\tau+n}|&\leqslant M_{K}\sum_{k=0}^{n}|b_k|v_{n-k}.
\end{align}

$\bullet$ \underline{Polynomial case:} Assume $(\mathbf{H_{poly}})$. Then for all $n\geqslant0$, $v_n=(n+1)^{-\alpha}$ with $\alpha>\frac{1}{2}\vee\left(\frac{3}{2}-\beta\right)$.

Here, \eqref{eq:proof_lem_proba_success_step2} is equivalent to 
\begin{align*}
\forall n\geqslant0,~|g^{(s)}_{\tau+n}|&\leqslant M_{K}C_\beta\sum_{k=0}^{n}(k+1)^{-\beta}(n+1-k)^{-\alpha}\\
&\leqslant M'_{K}~(n+1)^{-\min\{\alpha,\beta,\alpha+\beta-1\}+\varepsilon},
\end{align*}
for all $\varepsilon>0$ by applying the technical lemma \ref{lem:technical} and setting $M'_{K}:=C(\alpha,\beta)M_{K}C_\beta$.\\ We then set $\tilde{\alpha}:=\min\{\alpha,\beta,\alpha+\beta-1\}-1/2-\varepsilon$.

Hence, for all $\ell\geqslant2$,
\begin{align*}
\|g^{(s)}\|_{I_{\tau,\ell}}&=\left(\sum_{k=\tau+c_22^{\ell}}^{\tau+c_22^{\ell+1}-1}\left|g^{(s)}
_{k-1}\right|^2\right)^{1/2}\\
&\leqslant M'_K\left(\sum_{k=c_22^{\ell}}^{c_22^{\ell+1}-1}
k^{-2\tilde{\alpha}-1}\right)^{1/2}\\
&\leqslant M'_K\times\left(c_22^\ell\times(c_22^\ell)^{-2\tilde{\alpha}-1}\right)^{1/2}
= M'_K c_2^{-\tilde{\alpha}}2^{-\tilde{\alpha}\ell}.
\end{align*}
It remains to choose $c_2\geqslant~2\vee(M'_K)^{1/\tilde{\alpha}}$ to get the desired bound.\\

$\bullet$ \underline{Exponential case:} Assume $(\mathbf{H_{exp}})$. Then for all $n\geqslant0$, $v_n=e^{-\alpha n}$ with $\alpha>0$ and $\alpha\neq\zeta$.

Here, \eqref{eq:proof_lem_proba_success_step2} is equivalent to  
\begin{align*}
\forall n\geqslant0,~|g^{(s)}_{\tau+n}|&\leqslant M_{K}C_\zeta\sum_{k=0}^{n}e^{-\zeta k}e^{-\alpha(n-k)}\\
&\leqslant M'_{K}~e^{-\min\{\alpha,\zeta\}n},
\end{align*}
where we have set $M'_{K}:=M_{K}C_\zeta$. We then define $\tilde{\alpha}:=\min(\alpha,\zeta)$.\\
Hence, for all $\ell\geqslant2$,
\begin{align*}
\|g^{(s)}\|_{I^{(\tau)}_\ell}&=\left(\sum_{k=\tau+c_2\ell}^{\tau+c_2(\ell+1)-1}\left|g^{(s)}
_{k-1}\right|^2\right)^{1/2}\\
&\leqslant M'_{K}\left(\sum_{k=c_2\ell}^{c_2(\ell+1)-1}
e^{-2\tilde{\alpha}(k-1)}\right)^{1/2}\\
&\leqslant M'_{K}\times\left(\frac{e^{-2\tilde{\alpha}(c_2\ell-2)}}{2\tilde{\alpha}}\left[1-e^{-2\tilde{\alpha}c_2}\right]\right)^{1/2}\quad\text{ (by integral upper-bound)}\\
&\leqslant e^{-\tilde{\alpha}\ell}\quad\text{ (by choosing $c_2\geqslant 2$ large enough)}\\
&\leqslant 2^{-\tilde{\alpha}\ell}.
\end{align*}

For $\ell=1$, in both polynomial and exponential cases, the same approach gives us the existence of $C_K>0$ such that
\[\|g^{(s)}\|_{I_{\tau,1}}\leqslant C_K.\]
\end{proof}

By combining Lemma \ref{lem:success_proba_step1}, Lemma \ref{lem:success_proba_step2} and the expression 
\eqref{eq:new_expr_success_proba} we finally get Proposition \ref{prop:lower-bound_success_proba}.

\section{$(K,\alpha)$-admissibility}\label{section:admissibility}
\subsection{On condition \eqref{eq:adm_1}}

Let $\Delta t_3^{(j)}$ be the duration of Step 3 of trial $j$ for $j\geqslant1$.
The purpose of the next proposition is to prove that thanks to a calibration of $\Delta t_3^{(j)}$, one satisfies almost surely condition \eqref{eq:adm_1} at time $\tau_j$.

\begin{propo}\label{propo:adm1} 
Assume $(\mathbf{H_{1}})$ and $(\mathbf{H_{2}})$.
Let $\alpha\in\left(\frac{1}{2}\vee\left(\frac{3}{2}-\beta\right),\rho\right)$ if we are under $(\mathbf{H_{poly}})$ and $\alpha\in(0,\lambda)$ different from $\zeta$ if we are under $(\mathbf{H_{exp}})$. Assume that for all $j\geqslant 1$, 
\[\Delta t_3^{(j)}=\left\{\begin{array}{lll}
t_*\varsigma^j 2^{\theta \ell^*_j}&\text{ with }\theta>\left(2(\rho-\alpha)\right)^{-1}&\text{ under }(\mathbf{H_{poly}})\\
t_*+\varsigma^j+ \theta \ell^*_j&\text{ with }\theta>0&\text{ under }(\mathbf{H_{exp}})\\
\end{array}\right.\] 
where $\ell^*_j$ is defined in \eqref{def:failure_interval_number} and with $\varsigma>1$ arbitrary. Then for every $K>0$, there exists a choice of $t_*$ such that, for all $j\geqslant0$, condition \eqref{eq:adm_1} is a.s. true at time $\tau^{}_{j}$ on the event $\{\tau^{}_j<+\infty\}$.
In other words, for all $j\geqslant 0$,
\[\PP(\Omega^1_{\alpha,\tau^{}_j}|\{\tau^{}_j<+\infty\})=1.\]
\end{propo}

\begin{rem}\label{rem:exp_rate} $\rhd$ With Proposition \ref{propo:adm1} in hand, we are now in position to discuss the statement of Theorem \ref{thm:principal} (iii) as mentioned in Remark \ref{rem:thm_exp}. Adapting the proof of \eqref{bound_speed_p>1} (by taking an exponential Markov inequality), it appears that a fundamental tool to get exponential rate of convergence to equilibrium would be:
for all $j\geqslant1$, there exist $\lambda,C>0$ such that 
\begin{equation}\label{eq:control_moment_exp_rate}
\E[e^{\lambda\Delta\tau_j}|\mathcal{F}_{\tau_{j-1}}]\leqslant C
\end{equation} 
where $\Delta\tau_j:=\tau_j-\tau_{j-1}$. But, under $(\mathbf{H_{exp}})$, Proposition \ref{propo:adm1} shows that $\Delta\tau_j\geqslant\varsigma^j$ with $\varsigma>1$. This dependency on $j$ conflicts with the necessary control of the conditional expectation previously cited. \\
$\rhd$ Now, let us focus on the particular case of a finite memory: there exists $m\geqslant1$ such that $a_k=0$ for all $k> m$. As we saw in the second part of Subsection \ref{subsection:exponantial_rate}, $(|b_k|)$ decays exponentially fast. Then, an adjustment\footnote{In this case, we can fix the duration of Step 3 equal to $m$ since the memory only involves the $m$ previous times.} of the proof of Proposition \ref{propo:adm1} would lead to \eqref{eq:control_moment_exp_rate} and perhaps in such case, our robust but general approach could be simplified.
\end{rem}

\begin{proof}
Let us begin by the first coupling trial, in other words for $j=0$.
We recall that $g_k=0$ for all $k<\tau_0$ (see \eqref{eq:drift_before_coupling}), therefore 
\[\forall n\geqslant0,\left|\sum_{k=n+1}^{+\infty}a_kg_{\tau_0+n-k}\right|=0\leqslant v_n\]
and then condition \eqref{eq:adm_1} is a.s. true at time $\tau^{}_{0}$.
Now, we assume $j\geqslant1$ and we work on the event \[\{\tau^{}_j<+\infty\}\quad(\supset\{\tau^{}_m<+\infty\}\text{ for all }0\leqslant m\leqslant j-1).\]
Let us prove that on this event we have for all $n\geqslant0,~\left|\sum_{k=n+1}^{+\infty}a_kg_{\tau^{}_{j}+n-k}\right|\leqslant v_n.$
Set $u_n:=\sum_{k=n+1}^{+\infty}a_kg_{\tau^{}_{j}+n-k}$. Since $g_k=0$ for all $k<\tau_0$, we get
$~u_n=\sum_{k=n+1}^{n+\tau^{}_j-\tau_0}a_kg_{\tau^{}_{j}+n-k}.$
Let us now separate the right term into the contributions of the different coupling trials.
We get

\begin{align*}
u_n=\sum_{m=1}^{j}\left(\sum_{k=n+\tau^{}_j-\tau^{}_m+1}^{n+\tau^{}_j-\tau^{}_{m-1}}a_k~g_{\tau^{}_{j}+n-k}\right)
=\sum_{m=1}^{j}\underbrace{\left(\sum_{k=\tau^{}_{m-1}}^{\tau^{}_{m}-1}a_{n+\tau^{}_j-k}~g_{k}\right)
}_{(\star)_m}
\end{align*}

and $(\star)_m$ corresponds to the contribution of trial $m$, divided into two parts: success and failure.
We have now to distinguish two cases: \\

\underline{\textbf{First case:}} $\ell^*_m\geqslant1$, in other words the failure occurs during Step 2. We recall that in this case the system was automatically $(K,\alpha)$-admissible at time $\tau_{m-1}$, which will allow us to use Lemma \ref{lem:success_proba_step2} on $\tau_{m-1}$.

Then, since $g_k=0$ on $\llbracket\tau_{m-1}+c_2s_{\ell_m^*+1},\tau_m-1\rrbracket$ by definition of Step 3 of the coupling procedure,
\begin{align*}
(\star)_m&=\sum_{k=\tau^{}_{j
m-1}}^{\tau^{}_{m-1}+c_2s_{\ell_m^*+1}-1}a_{n+\tau^{}_j-k}~g_{k}\\
&=\underbrace{\sum_{k=0}^{c_2s_{\ell_m^*}-1}a_{n+\tau^{}_j-\tau^{}_{m-1}-k}~g_{\tau^{}_{m-1}+k}}_{\text{success}}
+\underbrace{\sum_{k=c_2s_{\ell_m^*}}^{c_2s_{\ell_m^*+1}-1}a_{n+\tau^{}_j-\tau^{}_{m-1}-k}~g_{\tau^{}_{m-1}+k}}_{\text{failure}}.
\end{align*}

We have now to make the distinction between the polynomial and the exponential case.\\

$\rhd$ Under $(\mathbf{H_{poly}})$: $s_\ell=2^{\ell}$, $|a_k|\leqslant C_\rho(k+1)^{-\rho}$ and then $a_k^2\leqslant C_\rho^2(k+1)^{-2\rho}$.\\

Using Cauchy-Schwarz inequality, the domination assumption on $(a_k)$, and the fact that \[n+\tau^{}_j-\tau^{}_{m-1}-k+1\geqslant n+\tau^{}_{m}-\tau^{}_{m-1}-k+1=\Delta t_3^{(m)}+c_22^{\ell_m^*+1}+n-k\] we get,
\begin{align*}
|\text{success}|&\leqslant\left(\sum_{k=0}^{c_22^{\ell_m^*}-1}a^2_{n+\tau^{}_j-\tau^{}_{m-1}-k}\right)^{1/2}\left(\sum_{k=0}^{c_22^{\ell_m^*}-1}\left|g_{\tau_{m-1}+k}\right|^2\right)^{1/2}\\
&\leqslant C_\rho\left(\sum_{k=0}^{c_22^{\ell_m^*}-1}(n+\tau^{}_j-\tau^{}_{m-1}-k+1)^{-2\rho}\right)^{1/2}\left(\sum_{k=0}^{c_22^{\ell_m^*}-1}\left|g_{\tau_{m-1}+k}\right|^2\right)^{1/2}\\
&\leqslant C_\rho\left(\sum_{k=0}^{c_22^{\ell_m^*}-1}(\Delta t_3^{(m)}+c_22^{\ell_m^*+1}+n-k)^{-2\rho}\right)^{1/2}\left(\sum_{k=0}^{c_22^{\ell_m^*}-1}\left|g_{\tau_{m-1}+k}\right|^2\right)^{1/2}\\
&\quad= C_\rho\left(\sum_{k=\Delta t_3^{(m)}+c_22^{\ell_m^*}+n+1}^{\Delta t_3^{(m)}+c_22^{\ell_m^*+1}+n}k^{-2\rho}\right)^{1/2}\left(\sum_{k=0}^{c_22^{\ell_m^*}-1}\left|g_{\tau_{m-1}+k}\right|^2\right)^{1/2}\\
&\leqslant C_\rho\left(c_22^{\ell^*_m}(n+\Delta t_3^{(m)})^{-2\rho}\right)^{1/2}\left(\sum_{k=0}^{c_22^{\ell_m^*}-1}\left|g_{\tau_{m-1}+k}\right|^2\right)^{1/2}.\\
\end{align*}
By the same arguments, we obtain
\begin{equation*}
|\text{failure}|
\leqslant C_\rho\left(c_22^{\ell^*_m}(n+\Delta t_3^{(m)})^{-2\rho}\right)^{1/2}\left(\sum_{k=c_22^{\ell_m^*}}^{c_22^{\ell_m^*+1}-1}\left|g_{\tau^{}_{m-1}+k}\right|^2\right)^{1/2}.
\end{equation*}

Hence, by the triangular inequality, we have 
\[\left|(\star)_m\right|\leqslant C_\rho\sqrt{c_2}2^{\ell^*_m/2}(n+\Delta t_3^{(m)})^{-\rho}\left[
\underset{\textbf{(1)}}{\left(\sum_{k=0}^{c_22^{\ell_m^*}-1}\left|g_{\tau^{}_{m-1}+k}\right|^2\right)^{1/2}}+
\underset{\textbf{(2)}}{\left(\sum_{k=c_22^{\ell_m^*}}^{c_2^{\ell_m^*+1}-1}\left|g_{\tau^{}_{m-1}+k}\right|^2\right)^{1/2}}\right].\]

Since $\ell^*_j\geqslant1$, by Lemma \ref{lem:success_proba_step1} (ii) and Lemma \ref{lem:success_proba_step2}, we have
\[\textbf{(1)}\leqslant M_K + \sum_{\ell=1}^{\ell_m^*-1}\|g^{(s)}\|_{I_\ell}\leqslant M_K+C_K+\sum_{\ell=2}^{+\infty}2^{-\tilde{\alpha}\ell}=:\tilde{C}_K\]
and
\[\textbf{(2)}\leqslant \max(C_\alpha,C'_K)(\ell^*_m+1).\]

Therefore

\begin{equation}\label{majo_poly}
\left|(\star)_m\right|\leqslant C^{(2)}_K~2^{\ell^*_m/2} (\ell^*_m+1)\left(n+\Delta t_3^{(m)}\right)^{-\rho}  
\end{equation}
where $C^{(2)}_K:=\max(C_\alpha,C'_K,\tilde{C}_K)$.

Moreover, recall that under $(\mathbf{H_{poly}})$
\begin{equation}\label{def:duration_step3}
\Delta t_3^{(m)}=t_*\varsigma^m2^{\theta\ell_m^*}~\text{ with }~\theta>(2(\rho-\alpha))^{-1}\text{ and }\varsigma>1.
\end{equation}

Plugging the definition of $\Delta t_3^{(m)}$ into \eqref{majo_poly} and using that for all $x,y>0,\quad (x+y)^{-\rho}\leqslant x^{-(\rho-\alpha)}y^{-\alpha}$ 
\begin{align}
\left|(\star)_m\right|&\leqslant C^{(2)}_K2^{\ell^*_m/2} (\ell^*_m+1)(n+\Delta t_3^{(m)})^{-\rho}  \nonumber\\
&\leqslant C^{(2)}_K(\ell_m^*+1)2^{(1/2-\theta(\rho-\alpha))\ell_m^*}(t_*\varsigma^m)^{-(\rho-\alpha)}(n+1)^{-\alpha}\label{eq:majo_poly2}.
 \end{align}

Since $\theta>(2(\rho-\alpha))^{-1}$ we have
\[C_{\alpha,K}=\sup_{\ell^*>0}C^{(2)}_K(\ell^*+1)2^{(1/2-\theta(\rho-\alpha))\ell^*}<+\infty,\]
and \eqref{eq:majo_poly2} yields
\begin{equation}\label{final_majo_poly}
\left|(\star)_m\right|\leqslant 
C_{\alpha,K} (t_*\varsigma^m)^{-(\rho-\alpha)}(n+1)^{-\alpha}\quad\text{ under }(\mathbf{H_{poly}}).
\end{equation}

\vspace{1cm}

$\rhd$ Under $(\mathbf{H_{exp}})$: $s_\ell=\ell$, $|a_k|\leqslant C_\lambda e^{-\lambda k}$ and then $a_k^2\leqslant C_\lambda^2e^{-2\lambda k}$.\\

Since the proof is almost the same in the exponential case, we will go faster and skip some details.\\
Using again Cauchy-Schwarz inequality, the domination assumption on $(a_k)$, and the fact that \[n+\tau^{}_j-\tau^{}_{m-1}-k\geqslant n+\tau^{}_{m}-\tau^{}_{m-1}-k=\Delta t_3^{(m)}+c_2(\ell_m^*+1)+n-k-1\] we get

\begin{align*}
|\text{success}|\leqslant\frac{C_\lambda e^\lambda}{\sqrt{2\lambda}}e^{-\lambda(n+\Delta t_3^{(m)})}\left(\sum_{k=0}^{c_22^{\ell_m^*}-1}\left|g_{\tau_{m-1}+k}\right|^2\right)^{1/2}
\end{align*}
and by the same arguments,
\begin{align*}
|\text{failure}|&\leqslant\frac{C_\lambda e^\lambda}{\sqrt{2\lambda}}e^{-\lambda(n+\Delta t_3^{(m)})}\left(\sum_{k=c_2\ell_m^*}^{c_2(\ell_m^*+1)-1}\left|g_{\tau^{}_{m-1}+k}\right|^2\right)^{1/2}.
\end{align*}

As in the polynomial case, by using Lemma \ref{lem:success_proba_step1} and \ref{lem:success_proba_step2} we get the existence of $C_K^{(3)}>0$ such that

\begin{equation}\label{majo_exp}
\left|(\star)_m\right|\leqslant C^{(3)}_K (\ell^*_m+1)e^{-\lambda(n+\Delta t_3^{(m)})}.
\end{equation}

Moreover, recall that under $(\mathbf{H_{exp}})$
\begin{equation}\label{def:duration_step3_exp}
\Delta t_3^{(m)}=t_*+\varsigma^m+\theta\ell_m^*~\text{ with }~\theta>0\text{ and }\varsigma>1.
\end{equation}

Plugging the definition of $\Delta t_3^{(m)}$ into \eqref{majo_exp}, we get
\begin{align*}
\left|(\star)_m\right|&\leqslant C^{(3)}_K (\ell^*_m+1)e^{-\lambda(n+\Delta t_3^{(m)})}  \\
&\leqslant C^{(3)}_K(\ell_m^*+1)e^{-(\lambda-\alpha)\Delta t_3^{(m)}}e^{-\alpha n}\\
 &\quad= C^{(3)}_K(\ell_m^*+1)e^{-(\lambda-\alpha)(t_*+\varsigma^m+\theta\ell_m^*)}e^{-\alpha n}.
 \end{align*}

We set 
\[C'_{\alpha,K}=\sup_{\ell^*>0}C^{(3)}_K(\ell^*+1)e^{-\theta(\lambda-\alpha)\ell^*}<+\infty.\]
And this gives us
\begin{equation}\label{final_majo_exp}
\left|(\star)_m\right|\leqslant 
C'_{\alpha,K} e^{-(\lambda-\alpha)(t_*+\varsigma^m)}e^{-\alpha n}\quad\text{ under }(\mathbf{H_{exp}}).
\end{equation}

\vspace{1cm}

\underline{\textbf{Second case:}} $\ell^*_m=0$, in other words, failure occurs during Step 1. This includes the case when the system is not $(K,\alpha)$-admissible at time $\tau_{m-1}$.\\

We have
\[(\star)_m=a_{n+\tau^{}_j-\tau^{}_{m-1}}~g_{\tau^{}_{m-1}}.\]
By Lemma \ref{lem:success_proba_step1} (ii), $|g_{\tau^{}_{m-1}}|\leqslant M_K$.\\
Moreover, since $n+\tau^{}_j-\tau^{}_{m-1}\geqslant n+\tau^{}_m-\tau^{}_{m-1}=n+\Delta t_3^{(m)}$, we obtain by using the same method as in the first case,
\begin{align}\label{case2}
\left|(\star)_m\right|&\leqslant \left\{\begin{array}{ll}
M_K(t_*\varsigma^m)^{-(\rho-\alpha)}(n+1)^{-\alpha}&\text{ under }(\mathbf{H_{poly}})\\
M_Ke^{-(\lambda-\alpha)(t_*+\varsigma^m)}e^{-\alpha n}&\text{ under }(\mathbf{H_{exp}})\\
\end{array}\right..
\end{align} 

\vspace{1cm}

By putting \eqref{final_majo_poly}, \eqref{final_majo_exp} and \eqref{case2} together, we finally get 
\begin{align*}
\left|(\star)_m\right|&\leqslant \left\{\begin{array}{ll}
\max(M_K,C_{\alpha,K})(t_*\varsigma^m)^{-(\rho-\alpha)}(n+1)^{-\alpha}&\text{ under }(\mathbf{H_{poly}})\\
\max(M_K,C'_{\alpha,K})e^{-(\lambda-\alpha)(t_*+\varsigma^m)}e^{-\alpha n}&\text{ under }(\mathbf{H_{exp}})\\
\end{array}\right..
\end{align*} 

Set $S_1=\sum_{m=1}^{+\infty}\varsigma^{-(\rho-\alpha)m}$ and $S_2=\sum_{m=1}^{+\infty}e^{-(\lambda-\alpha)\varsigma^m}$. By choosing $t_*$ large enough, we obtain for all $1\leqslant m\leqslant j$:
\begin{equation}\label{eq:majo_tentative_m}
\left|(\star)_m\right| \leqslant \left\{\begin{array}{ll}
\frac{1}{S_1}\varsigma^{-(\rho-\alpha)m}(n+1)^{-\alpha}&\text{ under }(\mathbf{H_{poly}})\\
\frac{1}{S_2}e^{-(\lambda-\alpha)\varsigma^m}e^{-\alpha n}&\text{ under }(\mathbf{H_{exp}})\\
\end{array}\right..
\end{equation}

Finally, by adding \eqref{eq:majo_tentative_m} for $m=1,..,j$ we have
\[\forall n\geqslant0,\quad|u_n|\leqslant\left\{\begin{array}{ll}
(n+1)^{-\alpha}&\text{ under }(\mathbf{H_{poly}})\\
e^{-\alpha n}&\text{ under }(\mathbf{H_{exp}})\\
\end{array}\right.\]
which concludes the proof.

\end{proof}

\subsection{Compact return condition \eqref{eq:adm_2}}\label{subsection:compact_return}
In the sequel, we set 
\begin{equation}\label{event:failure_trial_j}
\mathcal{E}_j:=\{\tau_j<\infty\}(=\{\tau_1<\infty,\dots,\tau_j<\infty\}).
\end{equation}
The aim of this subsection is to prove the following proposition: 

\begin{propo}\label{propo:compact_return} Assume $(\mathbf{H_1})$ and $(\mathbf{H_2})$. 
For all $\varepsilon>0$, there exists $K_\varepsilon>0$ such that
\begin{equation}\label{eq:compact_return}
\PP(\Omega_{K_\varepsilon,\tau_j}^2|\mathcal{E}_j)\geqslant 1-\varepsilon.
\end{equation}
\end{propo}

At this stage, we assume that $(\mathbf{H_{poly}})$ is true. Indeed,
the exponential case will immediately follow from the polynomial one since $(\mathbf{H_{exp}})$ implies $(\mathbf{H_{poly}})$.\\

Since for every events $A_1,A_2,A_3$ and $A_4$, we have $\PP(A_1\cap A_2\cap A_3\cap A_4)\geqslant\sum_{i=1}^4\PP(A_i)-3$
, it is enough to prove that for all $\varepsilon>0$, there exists $K_\varepsilon>0$ such that

\begin{equation}\label{eq:compact_return1}
\PP(|X^i_{\tau_j}|\leqslant K_\varepsilon|\mathcal{E}_j)\geqslant 1-\varepsilon\quad\text{ and }\quad \PP\left(\left.\left|\sum_{k=1}^{+\infty}a_k\xi^i_{\tau_j+1-k}\right|\leqslant K_\varepsilon\right|\mathcal{E}_j\right)\geqslant1-\varepsilon\quad\text{ for }i=1,2
\end{equation}
to get \eqref{eq:compact_return}. 
Let us first focus on the first part of \eqref{eq:compact_return1} concerning $|X^i_{\tau_j}|$ for $i=1,2$. Recall that the function $V:\R^d\to\R^*_+$ appearing in $(\mathbf{H_1})$ is such that $\lim\limits_{|x|\to+\infty}V(x)=+\infty$. For $K>0$ large enough we then have: $|x|\geqslant K$ $\Rightarrow$ $V(x)\geqslant K$. Therefore, for $i=1,2$ and $K_\varepsilon$ large enough, using Markov inequality we get 
\begin{align}\label{eq:majo_tempo}
&\PP(|X^i_{\tau_j}|\geqslant K_\varepsilon|\mathcal{E}_j)\leqslant\PP(V(X^i_{\tau_j})\geqslant K_\varepsilon|\mathcal{E}_j)\leqslant\frac{\E(V(X^i_{\tau_j})~|\mathcal{E}_j)}{K_\varepsilon}.
\end{align}

Hence, the first part of \eqref{eq:compact_return1} is true if there exists a constant $C$ such that for every $j\in\N$ and for every $K>0$,
\begin{equation}\label{eq:compact_return2}
\E(V(X^i_{\tau_j})~|\mathcal{E}_j)\leqslant C\quad\text{for }i=1,2.
\end{equation}

Indeed, plugging \eqref{eq:compact_return2} into \eqref{eq:majo_tempo} and taking $K_\varepsilon\geqslant\frac{C}{\varepsilon}$ yield the desired inequality. We see here that the independence of $C$ with respect to $K$ is essential. \\
For the sake of simplicity, we we will first use the following hypothesis to prove \eqref{eq:compact_return2}:

\begin{center}
$(\mathbf{H'_1})$: Let $\gamma\in(0,1)$. There exists $C_\gamma>0$ such that for all $j\in\N$, for every $K>0$ and for $i=1,2$,
\[\E\left[\left.\sum_{l=1}^{\Delta\tau_j}\gamma^{\Delta\tau_j-l}|\Delta^i_{\tau_{j-1}+l}|\quad\right|\mathcal{E}_j\right]<C_\gamma\]
where $\Delta\tau_j:=\tau_j-\tau_{j-1}$ and $\Delta^i$ is the stationary Gaussian sequence defined by \eqref{eq:moving_average}.
\end{center}

\begin{propo}\label{propo:majo_expectation}
Assume $(\mathbf{H_1})$, $(\mathbf{H_2})$ and $(\mathbf{H'_1})$. Let $(X^1_n,X^2_n)_{n\in\N}$ be a solution of \eqref{syst_couplage} with initial condition $(X^1_0,X^2_0)$ satisfying $\E(V(X^i_0))<\infty$ for $i=1,2$. Moreover, assume that $\tau_0=0$ and that $(\tau_j)_{j\geqslant1}$ is built in such a way that for all $j\geqslant1$,
$\PP(\mathcal{E}_j|\mathcal{E}_{j-1})\geqslant\delta_1>0$ (where $\delta_1$ is not depending on $j$) and 
$\Delta\tau_j\geqslant\frac{\log(\delta_1/2)}{\log(\gamma)}$. Then, there esists a constant $C$ such that for all $j\in\N$ and for every $K>0$,
\begin{equation*}
\E(V(X^i_{\tau_j})~|\mathcal{E}_j)\leqslant C\quad\text{for }i=1,2.
\end{equation*}
\end{propo}

\vspace{1cm}

\begin{rem}
$\rhd$ Actually, hypothesis $(\mathbf{H'_1})$ is true under $(\mathbf{H_{poly}})$ and will be proven Appendix \ref{appendix:H'_1}.\\
$\rhd$ The existence of $\delta_1>0$ independent from $j$ is proven in Subsection \ref{subsection:lower-bound_coupling_proba}.\\
 $\rhd$ To get $\Delta\tau_j\geqslant\frac{\log(\delta_1/2)}{\log(\gamma)}$, it is sufficient to choose $t_*$ large enough in the expression of $\Delta t_3^{(j)}$ (see Proposition \ref{propo:adm1}).\\
$\rhd$ Since in Theorem \ref{thm:principal} we made the assumption $\int_\X V(x)\Pi^*_\X\mu_0(dx)<+\infty$ and since an invariant distribution (extracted thanks to Theorem \ref{thm:existence_invariant_dist}) also satisfies $\int_\X V(x)\Pi^*_\X\mu_\star(dx)<+\infty$, we get that $\E(V(X^i_0))<\infty$ for $i=1,2$. Hence, we can set $\tau_0=0$.
\end{rem}

\begin{proof}
By $(\mathbf{H_1})$, there exist $\gamma\in(0,1)$ and $C>0$ such that for all $n\geqslant0$ and for $i=1,2$ we have 
\[V(X^i_{n+1})\leqslant\gamma V(X^i_n)+C(1+|\Delta^i_{n+1}|).\]
By applying this inequality at time $n=\tau_j-1$, and by induction, we immediately get:
\begin{equation}\label{eq:proof_majo_expectation}
V(X^i_{\tau_j})\leqslant\gamma^{\Delta\tau_j}V(X^i_{\tau_{j-1}})+C\sum_{l=1}^{\Delta\tau_j}\gamma^{\Delta\tau_j-l}(1+|\Delta^i_{\tau_{j-1}+l}|).
\end{equation}
By assumption $\Delta\tau_j\geqslant\frac{\log(\delta_1/2)}{\log(\gamma)}$ then $\gamma^{\Delta\tau_j}\leqslant\frac{\delta_1}{2}$. Moreover, since $\mathcal{E}_j\subset\mathcal{E}_{j-1}$
and $\PP(\mathcal{E}_j|\mathcal{E}_{j-1})\geqslant\delta_1$, we get 
\begin{align*}
\E[V(X^i_{\tau_{j-1}})~|\mathcal{E}_j]=\frac{1}{\PP(\mathcal{E}_j)}\E[V(X^i_{\tau_{j-1}})\mathds{1}_{\mathcal{E}_j}]
&\leqslant\underbrace{\frac{\PP(\mathcal{E}_{j-1})}{\PP(\mathcal{E}_j)}}_{=\PP(\mathcal{E}_j|\mathcal{E}_{j-1})^{-1}}\E[V(X^i_{\tau_{j-1}})~|\mathcal{E}_{j-1}]\\
&\leqslant \delta_1^{-1}\E[V(X^i_{\tau_{j-1}})~|\mathcal{E}_{j-1}].
\end{align*}
Therefore
\[\E[\gamma^{\Delta\tau_j}V(X^i_{\tau_{j-1}})~|\mathcal{E}_j]\leqslant\frac{\delta_1}{2}\delta_1^{-1}\E[V(X^i_{\tau_{j-1}})~|\mathcal{E}_{j-1}]=\frac{1}{2}\E[V(X^i_{\tau_{j-1}})~|\mathcal{E}_{j-1}].\]
Hence, by taking \eqref{eq:proof_majo_expectation}, we have 
\begin{align*}
\E[V(X^i_{\tau_j})~|\mathcal{E}_j]&\leqslant\frac{1}{2}\E[V(X^i_{\tau_{j-1}})~|\mathcal{E}_{j-1}]+C\sum_{l=0}^{+\infty}\gamma^{l}+C~\E\left[\left.\sum_{l=1}^{\Delta\tau_j}\gamma^{\Delta\tau_j-l}|\Delta^i_{\tau_{j-1}+l}|\quad\right|\mathcal{E}_j\right]\\
&=\frac{1}{2}\E[V(X^i_{\tau_{j-1}})~|\mathcal{E}_{j-1}]+\frac{C}{1-\gamma}+C~\E\left[\left.\sum_{l=1}^{\Delta\tau_j}\gamma^{\Delta\tau_j-l}|\Delta^i_{\tau_{j-1}+l}|\quad\right|\mathcal{E}_j\right].
\end{align*}
Hypothesis $(\mathbf{H'_1})$ allows us to say
\[\E[V(X^i_{\tau_j})~|\mathcal{E}_j]\leqslant\frac{1}{2}\E[V(X^i_{\tau_{j-1}})~|\mathcal{E}_{j-1}]+C\left(\frac{1}{1-\gamma}+C_\gamma\right).\]

By induction, we get the existence of a constant $\tilde{C}_\gamma>0$ such that
\[\E[V(X^i_{\tau_j})~|\mathcal{E}_j]\leqslant\left(\frac{1}{2}\right)^{j}\E[V(X^i_{\tau_{0}})~|\mathcal{E}_{0}]+\tilde{C}_\gamma.\] 
Since $\PP(\mathcal{E}_0)=1$, it comes
\[\E[V(X^i_{\tau_j})~|\mathcal{E}_j]\leqslant\left(\frac{1}{2}\right)^{j}\E[V(X^i_{\tau_{0}})]+\tilde{C}_\gamma.\] 
Since $\tau_0=0$ and we assumed that $\E[V(X^i_{0})]<\infty$, the proof is over.
\end{proof}

As a result of the proof of $(\mathbf{H'_1})$ (see Appendix \ref{appendix:H'_1}), we get the first part of \eqref{eq:compact_return1}. The second part can be also deduced from the proof of $(\mathbf{H'_1})$ thanks to Remark \ref{rem:second_part_compact_return}.

\section{Proof of Theorem \ref{thm:principal}}\label{section:proof_thm_principal}

Now we have all the necessary elements to prove the second part of the main theorem \ref{thm:principal} concerning the convergence in total variation to the unique invariant distribution (where the uniqueness will immediately follow from this convergence).\\

We recall that $\Delta\tau^{}_{j}$ denotes the duration of coupling trial $j$ and we set
\begin{equation}\label{eq:coupling_trial_successful}
j^{(s)}:=\inf\{j>0,\Delta\tau^{}_{j}=+\infty\}. 
\end{equation}
 $j^{(s)}$ corresponds to the trial where the coupling procedure is successful.
 The aim of this section is to bound $\PP(\tau^{}_{\infty}>n)~$ where
 $~\tau^{}_{\infty}=\inf\{n\geqslant0~|~X^1_k=X^2_k,~\forall k\geqslant n\},~$ since $~\|\mathcal{L}((X^1_k)_{k\geqslant n})-\mathcal{S}\mu_\star\|_{TV}\leqslant\PP(\tau^{}_{\infty}>n)$.
But, we have
 \[\PP(\tau^{}_\infty>n)=\PP\left(\sum_{k=1}^{+\infty}\Delta\tau^{}_k\mathds{1}_{\{j^{(s)}>k\}}>n\right)\]
 where $j^{(s)}$ is defined in \eqref{eq:coupling_trial_successful}.
It remains to bound the right term. Let $p\in(0,+\infty)$. \\
If $p\in(0,1)$, then by the Markov inequality and the simple inequality $|a+b|^p\leqslant|a|^p+|b|^p$, we get 
 \begin{align}\label{bound_speed_p<1}
 \PP\left(\sum_{k=1}^{+\infty}\Delta\tau^{}_k\mathds{1}_{\{j^{(s)}>k\}}>n\right)
 &\leqslant\frac{1}{n^p}\sum_{k=1}^{+\infty}\E[|\Delta\tau^{}_k|^p\mathds{1}_{\{j^{(s)}>k\}}]\nonumber\\
 &\leqslant\frac{1}{n^p}\sum_{k=1}^{+\infty}\E[\E[|\Delta\tau^{}_k|^p\mathds{1}_{\{\Delta\tau^{}_k<+\infty\}}~|~\{\tau^{}_{k-1}<+\infty\}]
 \mathds{1}_{\{\tau^{}_{k-1}<+\infty\}}].
 \end{align}
Else, if $p\geqslant1$, by Markov inequality and Minkowski inequality, we have
 \begin{align}\label{bound_speed_p>1}
 \PP\left(\sum_{k=1}^{+\infty}\Delta\tau^{}_k\mathds{1}_{\{j^{(s)}>k\}}>n\right)
 &\leqslant\frac{1}{n^p}\E\left[\left(\sum_{k=1}^{+\infty}\Delta\tau^{}_k\mathds{1}_{\{j^{(s)}>k\}}\right)^{p}\right]\nonumber\\
 &\leqslant\frac{1}{n^p}\left(\sum_{k=1}^{+\infty}\E[|\Delta\tau^{}_k|^p\mathds{1}_{\{j^{(s)}>k\}}]^{1/p}\right)^p\nonumber\\
 &\leqslant\frac{1}{n^p}\left(\sum_{k=1}^{+\infty}\E[\E[|\Delta\tau^{}_k|^p\mathds{1}_{\{\Delta\tau^{}_k<+\infty\}}~|~\{\tau^{}_{k-1}<+\infty\}]
 \mathds{1}_{\{\tau^{}_{k-1}<+\infty\}}]^{1/p}\right)^p.
 \end{align}

We define the event $\mathcal{A}_{k,\ell}:=\mathcal{B}^c_{k,\ell}\cap\mathcal{B}_{k,\ell-1}$ which corresponds to the failure of Step 2 after $\ell$ attempts at trial $k$.
Both in \eqref{bound_speed_p<1} and \eqref{bound_speed_p>1}, we separate the term $\E[|\Delta\tau^{}_k|^p\mathds{1}_{\{\Delta\tau^{}_k<+\infty\}}~|~\{\tau^{}_{k-1}<+\infty\}]$ through the events $\mathcal{A}_{k,\ell}$ which gives 
 \begin{equation}\label{majo_sur_A_k_l}
 \E[|\Delta\tau^{}_k|^p\mathds{1}_{\{\Delta\tau^{}_k<+\infty\}}~|~\{\tau^{}_{k-1}<+\infty\}]
 =\sum_{\ell=1}^{+\infty}\E[\mathds{1}_{\mathcal{A}_{k,\ell}}|\Delta\tau^{}_k|^p\mathds{1}_{\{\Delta\tau^{}_k<+\infty\}}~|~\{\tau^{}_{k-1}<+\infty\}].
 \end{equation}
Moreover, thanks to Lemma \ref{lem:success_proba_step2} and the definition of the events $\mathcal{A}_{k,\ell}$, we deduce that for $\ell\geqslant2$,
\begin{equation}\label{eq:majo_proba_A_k_l}
\PP(\mathcal{A}_{k,\ell}~|~\{\tau^{}_{k-1}<+\infty\})\leqslant
 2^{-\tilde{\alpha}\ell}
 \end{equation}
 where $\tilde{\alpha}:=\left\{\begin{array}{ccc}
\min\{\alpha,\beta,\alpha+\beta-1\}-1/2-\varepsilon & \text{for all }\varepsilon>0 &\text{ under } (\mathbf{H_{poly}})\\
\min(\alpha,\zeta) & &\text{ under } (\mathbf{H_{exp}}).
\end{array}\right.$\\

We have now to distinguish the polynomial case from the exponential one.\\

$\rhd$ Under $(\mathbf{H_{poly}})$: \\

We have a bound of type $\Delta\tau^{}_k\leqslant C_1\varsigma^k2^{\theta\ell}$ (due to the value of $\Delta t_3^{(k)}$, see Proposition \ref{propo:adm1})
 on the event $\mathcal{A}_{k,\ell}$ where $\varsigma>1$ is arbitrary. Indeed, on $\mathcal{A}_{k,\ell}$, we have
 \begin{align*}
 \Delta\tau^{}_k=\tau^{}_{k}-\tau^{}_{k-1} &\leqslant c_22^{\ell+1}+\Delta t_3^{(k)}\\
 &=c_22^{\ell+1}+t_*\varsigma^{k}2^{\theta\ell}\\
 &\leqslant C_1\varsigma^k2^{(\theta\vee1)\ell}~\text{ (for }C_1
 \text{ large enough)}.
 \end{align*}
 Hence, in \eqref{majo_sur_A_k_l} we get
 \begin{align*}\E[|\Delta\tau^{}_k|^p\mathds{1}_{\{\Delta\tau^{}_k<+\infty\}}~|~\{\tau^{}_{k-1}<+\infty\}]
 &\leqslant C_1^p\varsigma^{kp}\left(\sum_{\ell=1}^{+\infty}2^{(\theta\vee1) p\ell}~
 \PP(\mathcal{A}_{k,\ell}~|~\{\tau^{}_{k-1}<+\infty\})\right)\nonumber\\
 &\leqslant C_1^p\varsigma^{kp}\left(2^{(\theta\vee1) p}
 +\sum_{\ell=2}^{+\infty}2^{((\theta\vee1) p-\tilde{\alpha})\ell}\right)\quad\text{ using \eqref{eq:majo_proba_A_k_l}}\nonumber\\
 &\leqslant C\varsigma^{kp}\quad\text{ if $p\in\left(0,\frac{\tilde{\alpha}}{\theta\vee1}\right)$.}
 \end{align*}
Then for $p\in\left(0,\frac{\tilde{\alpha}}{\theta\vee1}\right)$,
 \begin{equation}\label{majo_k}
 \E[\E[|\Delta\tau^{}_k|^p\mathds{1}_{\{\Delta\tau^{}_k<+\infty\}}~|~\{\tau^{}_{k-1}<+\infty\}]
 \mathds{1}_{\{\tau^{}_{k-1}<+\infty\}}]\leqslant C\varsigma^{kp}\PP(j^{(s)}>k-1)
 \end{equation}
 and it remains to control $\PP(j^{(s)}>k-1)$. We have \[\PP(j^{(s)}>k-1)=\prod_{j=1}^{k-1}\PP(\mathcal{E}_j|\mathcal{E}_{j-1})=\prod_{j=1}^{k-1}(1-\PP(\mathcal{E}^c_j|\mathcal{E}_{j-1}))\]
where $\mathcal{E}_j$ is defined in \eqref{event:failure_trial_j}. By Proposition \ref{prop:lower-bound_success_proba} and \ref{propo:compact_return} applied for $\varepsilon=1/2$, we get for every $j\geqslant2$,
 \[\PP(\mathcal{E}^c_j|\mathcal{E}_{j-1})\geqslant\PP(\Delta\tau_j=+\infty|\Omega_{K_{1/2},\alpha,\tau_{j-1}})\PP(\Omega_{K_{1/2},\alpha,\tau_{j-1}}|\mathcal{E}_{j-1})\geqslant\frac{\delta_0}{2}\]
where $\delta_0>0$ depends on $K_{1/2}$. 
Therefore,
 $~\PP(j^{(s)}>k-1)\leqslant \left(1-\frac{\delta_0}{2}\right)^{k-1}~$
and by \eqref{majo_k}
 \begin{equation}
 \E[\E[|\Delta\tau^{}_k|^p\mathds{1}_{\{\Delta\tau^{}_k<+\infty\}}~|~\{\tau^{}_{k-1}<+\infty\}]
 \mathds{1}_{\{\tau^{}_{k-1}<+\infty\}}]\leqslant C\varsigma^{kp}\left(1-\frac{\delta_0}{2}\right)^{k-1}.
 \end{equation}
 Finally, by choosing $1<\varsigma<\left(1-\frac{\delta_0}{2}\right)^{-1/p}$, we get using \eqref{bound_speed_p<1} or \eqref{bound_speed_p>1} that for all $p\in\left(0,\frac{\tilde{\alpha}}{\theta\vee1}\right)$, there exists $C_p>0$ such that 
\begin{equation} \label{last_majo}
 \PP(\tau_\infty>n)\leqslant\PP\left(\sum_{k=1}^{+\infty}\Delta\tau^{}_k\mathds{1}_{\{j^{(s)}>k\}}>n\right)\leqslant C_pn^{-p}.
 \end{equation}
 It remains to optimize the upper-bound $\frac{\tilde{\alpha}}{\theta\vee1}$ for $p$. Since $\tilde{\alpha}:=\min\{\alpha,~\beta,~\alpha+\beta-1\}-1/2-\varepsilon$ with $\varepsilon>0$ as small as necessary and since by Proposition \ref{propo:adm1}: $\theta>(2(\rho-\alpha))^{-1} \text{ and } \alpha\in\left(\frac{1}{2}\vee\left(\frac{3}{2}-\beta\right),\rho\right),$ we finally get \eqref{last_majo} for all $p\in\left(0,v(\beta,\rho)\right)$ where 
 \[v(\beta,\rho)=\sup\limits_{\alpha\in\left(\frac{1}{2}\vee\left(\frac{3}{2}-\beta\right),\rho\right)}\min\{1,2(\rho-\alpha)\}(\min\{\alpha,~\beta,~\alpha+\beta-1\}-1/2).\]
This concludes the proof of Theorem \ref{thm:principal} in the polynomial case.\\

$\rhd$ Under $(\mathbf{H_{exp}})$: \\

The proof is almost the same. The only differences are that we use the following bound 
 \begin{align*}
 \Delta\tau^{}_k=\tau^{}_{k}-\tau^{}_{k-1} &\leqslant c_2(\ell+1)+\Delta t_3^{(k)}\\
 &=c_2(\ell+1)+t_*+\varsigma^{k}+\theta\ell\\
 &\leqslant C_1\varsigma^k\theta\ell~\text{ (for }C_1
 \text{ large enough)}
 \end{align*}
 on the events $\mathcal{A}_{k,\ell}$ and the upperbound $\PP(\mathcal{A}_{k,\ell}~|~\{\tau^{}_{k-1}<+\infty\})\leqslant
 2^{-\tilde{\alpha}\ell}$ given in \eqref{eq:majo_proba_A_k_l}.
 And then we get for all $p>0$ the existence of $C_p>0$ such that
 \begin{equation} 
 \PP(\tau_\infty>n)\leqslant\PP\left(\sum_{k=1}^{+\infty}\Delta\tau^{}_k\mathds{1}_{\{j^{(s)}>k\}}>n\right)\leqslant C_pn^{-p}.
 \end{equation}
 by choosing $1<\varsigma<\left(1-\frac{\delta_0}{2}\right)^{-1/p}$
and the proof of Theorem \ref{thm:principal} is over.\\

%To end this paper, let us make a comment about the particular case where the dynamical system \eqref{SDS} is reduced to: $X_{n+1}=AX_{n}+\sigma\Delta_{n+1}$ where $A$ and $\sigma$ are some given matrices. As for (fractional) Ornstein-Uhlenbeck processes in a continuous setting, the study of linear dynamics can be achieved with specific methods. Here, the sequence $X$ benefits of a Gaussian structure so that the convergence in distribution could be studied through the covariance of the process. One can also remark that since for two paths $X$ and $\tilde{X}$ built with the same noise, we have: $\tilde{X}_{n+1}-X_{n+1}=A(\tilde{X}_{n}-X_{n})$, a simple induction leads to $\E[|\tilde{X}_{n}-X_{n}|^2]\leqslant\rho(A^*A)^n\E[|\tilde{X}_{0}-X_{0}|^2]$ where $\rho(A^*A)$ is the spectral radius of the Hermitian matrix $A^*A$. So without going into the details, if $\rho(A^*A)<1$, such bounds lead to geometric rates of convergence in Wasserstein distance and also in total variation distance (on this topic, see e.g. \cite{panloup2018sub}). To conclude, it is worth noting that our coupling strategy (and the related rates in Theorem \ref{thm:principal}) is more adapted to a setting where the drift term in $F$ does not contract everywhere as in linear dynamics.

\section*{Acknowledgements}

I am grateful to my PhD advisors Fabien Panloup and Laure Coutin for suggesting the problem, for helping me in the research process and for their valuable comments. I also gratefully acknowledge the reviewers for their helpful suggestions for improving the paper.

\appendix

\section{Proof of Theorem \ref{thm:euler_scheme}}\label{appendix:proof_euler_scheme}

The beginning of the following proof makes use of ideas developped in \cite{cohen2011approximation, cohen2014approximation}. \\Let us recall that $F_h(x,w):=x+hb(x)+\sigma(x)w$.

\begin{proof}
Set $V(x)=|x|$.
Let us begin by proving that $(\mathbf{H_1})$ holds with $V$ for $F_h$ with $h>0$ small enough.
We have:
\[|F_h(x,w)|^2=|x|^2+h^2|b(x)|^2+2h\langle x,b(x)\rangle+2\langle x,\sigma(x)w\rangle+2h\langle b(x),\sigma(x)w\rangle+|\sigma(x)w|^2.\]
Then, using the inequality $|\langle a,b\rangle|\leqslant\frac{1}{2}(\varepsilon|a|^2+\frac{1}{\varepsilon}|b|^2)$ for all $\varepsilon>0$, we get
\[
|\langle x,\sigma(x)w\rangle|\leqslant\frac{1}{2}(\varepsilon|x|^2+\frac{1}{\varepsilon}|\sigma(x)w|^2)
~\text{ and }~|\langle b(x),\sigma(x)w\rangle|\leqslant\frac{1}{2}(\varepsilon|b(x)|^2+\frac{1}{\varepsilon}|\sigma(x)w|^2).
\]
Moreover, assumptions \textbf{(L1)} and \textbf{(L2)} on $b$ give the existence of $\tilde{\beta}\in\R,~\alpha>0$ and $\tilde{C}>0$ such that
\begin{equation*}
|\langle b(x),x\rangle|\leqslant\tilde{\beta}-\tilde{\alpha}|x|^2~\text{ et }~
|b(x)|^2\leqslant\tilde{C}(1+|x|^2).
\end{equation*}
Hence, we finally have
\begin{align*}
|F_h(x,w)|^2&\leqslant |x|^2+\tilde{C}h^2(1+|x|^2)+2h(\tilde{\beta}-\tilde{\alpha}|x|^2)+\varepsilon|x|^2+\frac{1}{\varepsilon}|\sigma(x)w|^2\\
&~~~~~+\tilde{C}h\varepsilon(1+|x|^2)+\frac{h}{\varepsilon}|\sigma(x)w|^2+|\sigma(x)w|^2\\
&\leqslant|x|^2+2h(\tilde{\beta}-\tilde{\alpha}|x|^2)+\tilde{C}(\varepsilon+h\varepsilon+h^2)(1+|x|^2)+
\left(1+\frac{h+1}{\varepsilon}\right)|\sigma(x)w|^2.\\
\end{align*}
Now, set $\varepsilon=h^2$ and choose $0<h<\min\left\{\sqrt{1+\frac{\tilde{\alpha}}{\tilde{C}}}-1,~\frac{1}{\tilde{\alpha}}\right\}$. Then, we have $\tilde{C}(\varepsilon+h\varepsilon+h^2)\leqslant\tilde{\alpha}h$ and $0<1-\tilde{\alpha}h<1$.
Therefore,  
\[|F_h(x,w)|^2\leqslant|x|^2+h(\tilde{\gamma}-\tilde{\alpha}|x|^2)+\left(1+\frac{h+1}{\varepsilon}\right)|\sigma(x)w|^2\]
where $\tilde{\gamma}=2\tilde{\beta}+\tilde{\alpha}$.
Then
\begin{equation}\label{eq:1proof_H1}
|F_h(x,w)|^2\leqslant(1-\tilde{\alpha} h)|x|^2+h\tilde{\gamma}+\left(1+\frac{h+1}{\varepsilon}\right)|\sigma(x)w|^2.
\end{equation}
By assumption $\sigma$ is a bounded function on $\R^d$. Then, there exists $C'>0$ depending on $h$ and $\sigma$ such that
\[|F_h(x,w)|^2\leqslant(1-\tilde{\alpha} h)|x|^2+C'\left(1+|w|^2\right).\]
Using the classical inequality $\sqrt{a+b}\leqslant \sqrt{a}+\sqrt{b}$, we finally get the existence of 
$\gamma\in(0,1)$ and $C>0$ such that for all $(x,w)\in\R^d\times\R^d$
\begin{equation}
|F_h(x,w)|\leqslant\gamma|x|+C\left(1+|w|\right)
\end{equation}
which achieves the proof of $(\mathbf{H_1})$.\\
We now turn to the proof of $(\mathbf{H_2})$.
Let $K>0$ and take $\mathbf{x}=(x,x',y,y')\in B(0,K)^4$. Here we take $\tilde{K}=K$.
Hence, let us now define $\Lambda_{\mathbf{x}}$.
For all $u\in B(0,K)$, we set
\begin{equation}\label{eq:Lambda_center_space}
\Lambda_{\mathbf{x}}(u)=A(\mathbf{x})u+B(\mathbf{x})
\end{equation}
with $A(\mathbf{x}):=\sigma^{-1}(x')\sigma(x)$ and $B(\mathbf{x}):=\sigma^{-1}(x')(x-x'+h(b(x)-b(x')))+\sigma^{-1}(x')\sigma(x)y-y'$.\\
Then, \eqref{eq:Lambda_center_space} is equivalent to $\tilde{F}_h(x,u,y)=\tilde{F}_h(x',\Lambda_{\mathbf{x}}(u),y')$ for all $u\in B(0,K)$.
Hence, for all $u\in B(0,K)$,
\begin{equation}
J_{\Lambda_{\mathbf{x}}}(u)=\sigma^{-1}(x')\sigma(x).
\end{equation}
Since $\sigma,\sigma^{-1}$ and $b$ are continuous, there exist $C_K>0$ and $m_K>0$ independent from $\mathbf{x}$ such that for all $u\in B(0,K)$,
\begin{align*}
|\det(J_{\Lambda_{\mathbf{x}}}(u))|\geqslant C_{K}\\
|\Lambda_{\mathbf{x}}(u)-u|\leqslant m_K.
\end{align*}
For the sake of simplicity, let us set $\tilde{\Lambda}_{\mathbf{x}}(u)=\Lambda_{\mathbf{x}}(u)-B(\mathbf{x})$ and extend $\tilde{\Lambda}_{\mathbf{x}}$ to $\R^d$.
Now, let $K_1>0$ be independent of $\mathbf{x}$ such that $\sup\limits_{u\in B(0,K)}|A(\mathbf{x})u|<K_1$ and set for all $u\in \R^d\setminus B(0,K_1)$, $\tilde{\Lambda}_{\mathbf{x}}(u)=u$. Hence, $\tilde{\Lambda}_{\mathbf{x}}$ is a $\mathcal{C}^1$-diffeomorphism from $B(0,K)$ to $\mathcal{E}_K:=\{A(\mathbf{x})u~|u\in B(0,K)\}$ and from $\R^d\setminus B(0,K_1)$ to itself. It remains to extend it with a $\mathcal{C}^1$-diffeomorphism from $B(0,K_1)\setminus B(0,K)$ to $B(0,K_1)\setminus\mathcal{E}_K$. To this end, we consider $q$ the positive definite quadratic form associated to the ellipsoid $\mathcal{E}_K$ and we denote by $\mathcal{B}':=(e'_1,\dots,e'_d)$ the orthonormal basis which diagonalizes $q$, so that if $x=(x'_1,\dots,x'_d)$ in $\mathcal{B}'$ we have $q(x)=\sum_{i=1}^d \lambda_i (x'_i)^2$ with $\lambda_i>0$ for all $i\in\{1,\dots,d\}$. Let $\mathcal{B}:=(e_1,\dots,e_d)$ be the canonical basis and $\varphi:\R^d\to\R^d$ be the linear application such that $\varphi(e_i)=\frac{1}{\sqrt{\lambda_i}}e'_i$ for all $i\in\{1,\dots,d\}$.
\begin{rem} This application $\varphi$ gives also a $\mathcal{C}^1$-diffeomorphism from $B(0,K)$ to $\mathcal{E}_K$ by construction and $|u|=K\Longleftrightarrow \sqrt{|q(\varphi(u))|}=K$.
\end{rem}
Now, set for all $u\in B(0,K_1)\setminus B(0,K)$,
\[\tilde{\Lambda}_{\mathbf{x}}(u):=\left[\left(1-\frac{|u|-K}{K_1-K}\right)\frac{K}{|u|}+\left(\frac{|u|-K}{K_1-K}\frac{K_1}{|\varphi(u)|}\right)\right]\varphi(u).\]
This is just an interpolation between $\frac{K}{|u|}\varphi(u)$ and $\frac{K_1}{|\varphi(u)|}\varphi(u)$.
It is a $\mathcal{C}^1$-diffeomorphism from $B(0,K_1)\setminus B(0,K)$ to $B(0,K_1)\setminus\mathcal{E}_K$ and the inverse is given by, for all $v\in B(0,K_1)\setminus \mathcal{E}_K$:
\[\tilde{\Lambda}^{-1}_{\mathbf{x}}(v):=\left[\left(1-\frac{|v|-\alpha(|v|)}{K_1-\alpha(|v|)}\right)\frac{K}{\sqrt{q(v)}}+\left(\frac{|v|-\alpha(|v|)}{K_1-\alpha(|v|)}\frac{K_1}{|\varphi^{-1}(v)|}\right)\right]\varphi^{-1}(v)\]
with $\alpha(|v|):=|v|/\sqrt{q(v)}$.
Finally, one can check that we have all the elements to conclude that $\tilde{F}_h$ satisfies $(\mathbf{H_2})$.
\end{proof}

\begin{rem}\label{rem:relax_sigma_assumption}
$\rhd$ If we relax the boundedness assumption on $\sigma$ and assume that $|\sigma(x)|\leqslant C(1+|x|^{\kappa})$ for some $\kappa\in(0,1)$, only the proof of $(\mathbf{H_1})$ is changed. The beginning of the proof is exactly the same. From \eqref{eq:1proof_H1}, we use the classical Young inequality $|ab|\leqslant\frac{1}{p}\frac{|a|^p}{\varepsilon^p}+\frac{1}{q}|b|^q\varepsilon^q$ with $p=\frac{1}{\kappa}$ and $q=\frac{1}{1-\kappa}$ for the term $|\sigma(x)w|^2$. Then, it sufficies to use $|\sigma(x)|\leqslant C(1+|x|^{\kappa})$ and to calibrate $\varepsilon$ to get an inequality of the type: $|F(x,w)|^2\leqslant\gamma|x|^2+C(1+|w|^{\frac{2}{1-\kappa}})$. We conclude that $(\mathbf{H_1})$ holds with $V(x):=|x|^{1-\kappa}$.\\
$\rhd$ Let us consider the family of functions given by $F(x,w)=f(b(x)+\sigma(x)w)$. Provided that $f$ is well defined on $\R^d$, $\sigma$ is continuously invertible and $\sigma^{-1}$ and $b$ are continuous on $\R^d$, we can build a function $\Lambda_{\mathbf{x}}$ which satisfies $(\mathbf{H_2})$ exactly in the same way as in the preceding proof.
\end{rem}

\section{Explicit formula for the sequence $(b_k)_{k\geqslant0}$ defined in Proposition \ref{prop:inverse_T}}\label{appendix:formula_b}

\begin{thm}\label{thm:inversion}
Let $(u_n)_{n\in\N}$ and $(v_n)_{n\in\N}$ be two sequences such that for $n\in\N$,
\begin{equation}\label{f_fonction_g}
u_n=\sum_{k=0}^{n}a_kv_{n-k}
\end{equation}
then we have:
\begin{equation}\label{g_fonction_f}
v_n=\sum_{k=0}^{n}b_ku_{n-k}
\end{equation}
where 
\[b_0:=\frac1{a_0}~~\text{ and }~~\forall k\geqslant1,~b_k:=\sum_{p=1}^k\frac{(-1)^p}{a_0^{p+1}}\left(\sum_{\substack{k_1,\dots,k_p\geqslant1\\ k_1+\dots+k_p=k}}
\prod_{i=1}^pa_{k_i}\right).\]
\end{thm}

\begin{proof}
It sufficies to reverse a triangular Toeplitz matrix. Indeed, equation \eqref{f_fonction_g} is equivalent to:
\begin{equation}\label{sys_matriciel}
\forall n\in\N,~~\left(
\begin{matrix}
f_0\\f_1\\ \vdots \\f_n
\end{matrix}\right)
=\left(\begin{matrix}
a_0&0&\hdots&0\\
a_1&\ddots&\ddots&\vdots\\
\vdots&\ddots&\ddots&0\\
a_{n-1}&\hdots&a_1&a_0
\end{matrix}\right)
\left(\begin{matrix}
g_0\\g_1\\ \vdots \\g_n
\end{matrix}\right).
\end{equation}
Denote by $A$ the matrix asociated to the system. Denote by $N$ the following nilpotent matrix:
\[N=\left(\begin{matrix}
0&\hdots&\hdots&0\\
1&\ddots&&\vdots\\
&\ddots&\ddots&\vdots\\
(0)&&1&0
\end{matrix}\right).\]
Then, $A=a_0I_n+a_1N+\dots+a_{n-1}N^{n-1}$ and
we are looking for $B$ such that
\\$B=b_0I_n+b_1N+\dots+b_{n-1}N^{n-1}$
and $AB=I_n$.
Denote by
\[S(z)=\sum_{k\geqslant0}a_kz^k~~\text{ and }~~S^{-1}(z)=\sum_{k\geqslant0}b_kz^k,\]
we are interested in the $(n-1)$ first coefficients of $S^{-1}(z)$.

And formally,
\begin{align*}
S^{-1}(z)=\frac{1}{S(z)}&=\frac{1}{a_0}\left(\frac{1}{1+\sum_{k\geqslant1}\frac{a_k}{a_0}z^k}\right)\\
&=\frac{1}{a_0}\sum_{p\geqslant0}\frac{(-1)^p}{a_0^p}\left(\sum_{k\geqslant1}a_kz^k\right)^p\\
&=\frac{1}{a_0}+\sum_{p\geqslant1}\frac{(-1)^p}{a_0^{p+1}}\sum_{k\geqslant p}
\left(\sum_{\substack{k_1,\dots,k_p\geqslant1\\ k_1+\dots+k_p=k}}a_{k_1}a_{k_2}\dots a_{k_p}\right)z^k\\
&=\frac{1}{a_0}+\sum_{k\geqslant1}\sum_{p=1}^k\left(\frac{(-1)^p}{a_0^{p+1}}
\sum_{\substack{k_1,\dots,k_p\geqslant1\\ k_1+\dots+k_p=k}}a_{k_1}a_{k_2}\dots a_{k_p}\right)z^k.\\
\end{align*}
Finally, we identify the desired coefficients.
\end{proof}

\section{Particular case: when the sequence $(a_k)_{k\geqslant0}$ is log-convex}\label{appendix:a_log_convex}

This section is based on a work made by N.Ford, D.V.Savostyanov and N.L.Zamarashkin in \cite{ford2014decay}.\\

\begin{lem}\label{lem:toeplitz} Let $(a_n)_{n\in\N}$ be a log-convex sequence in the following sense 
\[ a_n\geqslant0~~\text{ for } n\geqslant0\quad \text{ and }\quad a_n^2\leqslant a_{n-1}a_{n+1}~~\text{ for }n\geqslant1.\]
If $a_0>0$, then the sequence $(b_n)_{n\in\N}$ defined by
\begin{equation*}
b_0=\frac1{a_0} ~~\text{ and }~~ \forall n\geqslant1,~~ b_n=-\frac{1}{a_0}\sum_{k=1}^{n}a_kb_{n-k}
\end{equation*}
satisfies 
\begin{equation}
\forall n\geqslant1,~~~b_n\leqslant0 ~~\text{ and }~~|b_n|\leqslant b_0a_{n}
\end{equation}
\end{lem}

\begin{rem}\label{rem:toeplitz}
The sequence $a_n=(n+1)^{-\rho}$ is log-convex for all $\rho>0$, then the corresponding $(b_n)_{n\in\N}$ is such that $\forall n\in\N,\quad |b_n|\leqslant(n+1)^{-\rho}$. 
\end{rem}

\begin{proof}
Without loss of generality, we assume that $a_0=1$.\\
$\bullet$ First, following Theorem 4 of \cite{ford2014decay} we can prove by strong induction that for all $n\geqslant1$, $b_{n}\leqslant0$.\\
$\bullet$ The second property satisfied by $(b_n)$ directly follows from the first one. 
 Let $n\geqslant1$, as we just saw $b_{n}\leqslant0$ therefore 
$|b_{n}|=-b_{n}$. But,
\begin{align*}
 -b_{n}&=\sum_{k=1}^{n-1}a_k(\underbrace{b_{n-k}}_{\leqslant0})+a_{n}b_0\\
 &\leqslant b_0a_{n}
\end{align*}
and the lemma is proven.
\end{proof}
\section{Proof of Proposition \ref{propo:b_rate}}\label{appendix:proof_propo_b_rate}
We recall that $\rho=3/2-H$ where $H\in(0,1/2)$ is the Hurst parameter and 
 $(b_n)_{n\in\N}$ is defined by
\begin{equation}\label{inverse_a_k}
b_0=\frac{1}{a^H_0}\quad \text{ and for }n\geqslant1,\quad b_n=-\frac{1}{a^H_0}\sum_{k=1}^na^H_kb_{n-k}.
\end{equation}

and for all $k\geqslant 1$,
\[\frac{a^H_k}{a^H_0}=\left(\left(2k+1\right)^{1-\rho}-\left(2k-1\right)^{1-\rho}\right).\]

We want to show that $|b_n|\leqslant C_b (n+1)^{-(2-\rho)}$ by induction. To this end we only need to prove that for $n$ large enough,
\begin{equation}
S_n:=\sum_{k=1}^n\left(\left(2k-1\right)^{1-\rho}-\left(2k+1\right)^{1-\rho}\right)(n+1-k)^{-(2-\rho)}\leqslant(n+1)^{-(2-\rho)}.
\end{equation}

For the sake of simplicity we assume that $n$ is even.
\begin{align}
S_n&=\sum_{k=1}^{n/2}\left(\left(2k-1\right)^{1-\rho}-\left(2k+1\right)^{1-\rho}\right)(n+1-k)^{-(2-\rho)}+\sum_{k=n/2+1}^n\left(\left(2k-1\right)^{1-\rho}-\left(2k+1\right)^{1-\rho}\right)(n+1-k)^{-(2-\rho)}\nonumber\\
S_n&=:S^{(1)}_n+S^{(2)}_n.
\end{align}

$\rhd$ We begin with $S^{(1)}_n$. Summation by parts:
\begin{align}\label{S1}
S^{(1)}_n&=\left(\frac{n}{2}+1\right)^{-(2-\rho)}(1-(n+1)^{1-\rho})-\sum_{k=1}^{n/2-1}(1-(2k+1)^{1-\rho})\left((n-k)^{-(2-\rho)}-(n+1-k)^{-(2-\rho)}\right)\nonumber\\
&=\left(\frac{n}{2}+1\right)^{-(2-\rho)}(1-(n+1)^{1-\rho})-\left(\left(\frac{n}{2}+1\right)^{-(2-\rho)}-n^{-(2-\rho)}\right)\nonumber\\
&\quad\quad\quad\quad\quad\quad+\sum_{k=1}^{n/2-1}(2k+1)^{1-\rho}\left((n-k)^{-(2-\rho)}-(n+1-k)^{-(2-\rho)}\right)\nonumber\\
&=n^{-(2-\rho)}-\left(\frac{n}{2}+1\right)^{-(2-\rho)}(n+1)^{1-\rho}+\sum_{k=1}^{n/2-1}(2k+1)^{1-\rho}\left((n-k)^{-(2-\rho)}-(n+1-k)^{-(2-\rho)}\right)\nonumber\\
S^{(1)}_n&=n^{-(2-\rho)}-\left(\frac{n}{2}\right)^{-(2-\rho)}(n+1)^{1-\rho}+\sum_{k=1}^{n/2}(2k+1)^{1-\rho}\left((n-k)^{-(2-\rho)}-(n+1-k)^{-(2-\rho)}\right).
\end{align}
We set $\tilde{S}_n:=\sum_{k=1}^{n/2}(2k+1)^{1-\rho}\left((n-k)^{-(2-\rho)}-(n+1-k)^{-(2-\rho)}\right)$.
Then,
\begin{align*}
\tilde{S}_n &=\frac{1}{n}\sum_{k=1}^{n/2}\left(\frac{2k+1}{n}\right)^{1-\rho}\left(\left(1-\frac{k}{n}\right)^{-(2-\rho)}-\left(1-\frac{k-1}{n}\right)^{-(2-\rho)}\right)\\
&=\frac{2^{1-\rho}}{n}\sum_{k=1}^{n/2}\left(\frac{k+1/2}{n}\right)^{1-\rho}\left(\left(1-\frac{k}{n}\right)^{-(2-\rho)}-\left(1-\frac{k-1}{n}\right)^{-(2-\rho)}\right).
\end{align*}

Moreover,

\begin{align}\label{somme_riemann}
&\sum_{k=1}^{n/2}\left(\frac{k+1/2}{n}\right)^{1-\rho}\left(\left(1-\frac{k}{n}\right)^{-(2-\rho)}-\left(1-\frac{k-1}{n}\right)^{-(2-\rho)}\right)
\nonumber\\
&\quad=(2-\rho)\left(\int_{0}^{1/2}\left(x+\frac{1}{2n}\right)^{1-\rho}(1-x)^{-(3-\rho)}{\rm d}x-\sum_{k=1}^{n/2}\int_{\frac{k-1}{n}}^{\frac{k}{n}}\left(\left(x+\frac{1}{2n}\right)^{1-\rho}-\left(\frac{k+1/2}{n}\right)^{1-\rho}\right)(1-x)^{-(3-\rho)}{\rm d}x\right)
\end{align}

and
\begin{align}\label{valeur_integrale}
\int_{0}^{1/2}\left(x+\frac{1}{2n}\right)^{1-\rho}(1-x)^{-(3-\rho)}{\rm d}x&=
\left[\frac{(1-x)^{\rho-2}\left(x+\frac{1}{2n}\right)^{2-\rho}}{(2-\rho)\left(1+\frac{1}{2n}\right)}\right]_0^{1/2}\nonumber\\
&=\frac{1}{2-\rho}\left(1+\frac{1}{2n}\right)^{-1}\left[\left(1+\frac{1}{n}\right)^{2-\rho}-\left(\frac{1}{2n}\right)^{2-\rho}\right].
\end{align}

Hence by putting together \eqref{somme_riemann} and \eqref{valeur_integrale} we get

\begin{align}\label{majo_S_tilde}
\tilde{S}_n&=\frac{2^{1-\rho}}{n}\left(\left(1+\frac{1}{2n}\right)^{-1}\left[\left(1+\frac{1}{n}\right)^{2-\rho}-\left(\frac{1}{2n}\right)^{2-\rho}\right]\right.\nonumber\\
&\quad\quad\quad\left.-(2-\rho)\sum_{k=1}^{n/2}\int_{\frac{k-1}{n}}^{\frac{k}{n}}\left(\left(x+\frac{1}{2n}\right)^{1-\rho}-\left(\frac{k+1/2}{n}\right)^{1-\rho}\right)(1-x)^{-(3-\rho)}{\rm d}x\right).
\end{align}

We deduce from \eqref{S1} and \eqref{majo_S_tilde} that 
\begin{equation}\label{majo_final_S1}
S^{(1)}_n\leqslant n^{-(2-\rho)}-\left(\frac{n}{2}\right)^{-(2-\rho)}(n+1)^{1-\rho}+\frac{2^{1-\rho}}{n}\left(1+\frac{1}{2n}\right)^{-1}\left[\left(1+\frac{1}{n}\right)^{2-\rho}-\left(\frac{1}{2n}\right)^{2-\rho}\right].
\end{equation}

$\rhd$ Now, we look after $S_n^{(2)}$:\\

As before, using the fact that

\begin{equation}\label{valeur_integrale2}
\int_{0}^{1/2}\left(1-x+\frac{1}{2n}\right)^{-\rho}x^{-(2-\rho)}{\rm d}x=
\left[\frac{\left(1-x+\frac{1}{2n}\right)^{1-\rho}x^{\rho-1}}{(\rho-1)\left(1+\frac{1}{2n}\right)}\right]_0^{1/2}=\frac{1}{\rho-1}\left(1+\frac{1}{2n}\right)^{-1}\left(1+\frac{1}{n}\right)^{1-\rho}
\end{equation}

we get

\begin{equation}\label{majo_S2}
S^{(2)}_n=\frac{2^{1-\rho}}{n}\left(\left(1+\frac{1}{2n}\right)^{-1}\left(1+\frac{1}{n}\right)^{1-\rho}-(\rho-1)\sum_{k=1}^{n/2}\int_{\frac{k-1}{n}}^{\frac{k}{n}}\left(x^{-(2-\rho)}-\left(\frac{k}{n}\right)^{-(2-\rho)}\right)\left(1-x+\frac{1}{2n}\right)^{-\rho}{\rm d}x\right).
\end{equation}

Now, for all $k\in\{1,\dots,n/2\}$, we set \[I_k:=\int_{\frac{k-1}{n}}^{\frac{k}{n}}\left(x^{-(2-\rho)}-\left(\frac{k}{n}\right)^{-(2-\rho)}\right)\left(1-x+\frac{1}{2n}\right)^{-\rho}{\rm d}x.\]

Thanks to the substitution $t=x-\frac{k-1}{n}$, we have

\begin{equation*}
I_k=\int_0^{1/n}\left(\left(t+\frac{k-1}{n}\right)^{-(2-\rho)}-\left(\frac{k}{n}\right)^{-(2-\rho)}\right)\left(1+\frac{1}{2n}-t-\frac{k-1}{n}\right)^{-\rho}{\rm d}t
\end{equation*}

Taylor-Lagrange expansion:

\begin{equation*}
\bullet\quad \left(t+\frac{k-1}{n}\right)^{-(2-\rho)}-\left(\frac{k}{n}\right)^{-(2-\rho)}=\left(\frac{1}{n}-t\right)(2-\rho)\left(\frac{k}{n}\right)^{-(3-\rho)}+\frac{1}{2}\left(\frac{1}{n}-t\right)^2(2-\rho)(3-\rho)\xi^{-(4-\rho)}
\end{equation*}
\quad \quad \quad \quad with $\xi\in]t+(k-1)/n,k/n[$.

\begin{equation*}
\bullet\quad \left(1+\frac{1}{2n}-t-\frac{k-1}{n}\right)^{-\rho}=\left(1+\frac{1}{2n}-\frac{k-1}{n}\right)^{-\rho}+t\rho\left(1+\frac{1}{2n}-c\right)^{-\rho-1}\quad\quad\quad\quad\quad\quad\quad\quad\quad\quad\quad\text{ }
\end{equation*}
\quad \quad \quad \quad with $c\in](k-1)/n,t+(k-1)/n[$.\\

Therefore, we deduce that
\begin{equation*}
I_k\geqslant\int_0^{1/n}\left(\frac{1}{n}-t\right)(2-\rho)\left(\frac{k}{n}\right)^{-(3-\rho)}\left(1+\frac{1}{2n}-\frac{k-1}{n}\right)^{-\rho}{\rm d}t
=\frac{2-\rho}{2n^2}\left(\frac{k}{n}\right)^{-(3-\rho)}\left(1+\frac{1}{2n}-\frac{k-1}{n}\right)^{-\rho}.
\end{equation*}

Then we add the inequality for $k$ from $1$ to $n/2$,

\begin{equation}\label{mino_somme_integrale}
\sum_{k=1}^{n/2}I_k\geqslant\frac{2-\rho}{2n}\times\underbrace{\frac{1}{n}\sum_{k=1}^{n/2}\left(\frac{k}{n}\right)^{-(3-\rho)}\left(1+\frac{1}{2n}-\frac{k-1}{n}\right)^{-\rho}}_{U_n}.
\end{equation}

We easily show that
\begin{equation}
U_n\geqslant\int_0^{1/2}\left(y+\frac{1}{n}\right)^{-(3-\rho)}\left(1+\frac{3}{2n}-y\right)^{-\rho}{\rm d}y=:J_n.
\end{equation}
By integration by parts on $J_n$ we get:
\begin{align*}
J_n&=\left[\frac{-\left(y+\frac{1}{n}\right)^{-(2-\rho)}}{2-\rho}\left(1+\frac{3}{2n}-y\right)^{-\rho}\right]_0^{1/2}+\frac{\rho}{2-\rho}\int_0^{1/2}\left(y+\frac{1}{n}\right)^{-(2-\rho)}\left(1+\frac{3}{2n}-y\right)^{-\rho-1}{\rm d}y\\
&=\frac{1}{2-\rho}\underbrace{\left[\left(\frac{1}{n}\right)^{-(2-\rho)}\left(1+\frac{3}{2n}\right)^{-\rho}-\left(\frac{1}{2}+\frac{1}{n}\right)^{-(2-\rho)}\left(\frac{1}{2}+\frac{3}{2n}\right)^{-\rho}\right]}_{\underset{n\to+\infty}{\sim}n^{2-\rho}}\\
&\quad\quad\quad+\frac{\rho}{2-\rho}\underbrace{\int_0^{1/2}\left(y+\frac{1}{n}\right)^{-(2-\rho)}\left(1+\frac{3}{2n}-y\right)^{-\rho-1}{\rm d}y}_{\underset{n\to+\infty}{\longrightarrow}\int_0^{1/2}y^{-(2-\rho)}\left(1-y\right)^{-\rho-1}{\rm d}y}.
\end{align*}
Hence, for $n$ large enough, we have 
\begin{equation}\label{mino_J_n}
U_n\geqslant J_n\geqslant\frac{1}{2(2-\rho)}n^{2-\rho}
\end{equation}
By combining \eqref{mino_somme_integrale} and \eqref{mino_J_n} we get for $n$ large enough
\begin{equation*}
\sum_{k=1}^{n/2}I_k\geqslant\frac{1}{4}n^{1-\rho}.
\end{equation*}
Finally we get for $S^{(2)}_n$ the following upper-bound for $n$ large enough,

\begin{equation}\label{majo_final_S2}
S^{(2)}_n\leqslant\frac{2^{1-\rho}}{n}\left(\left(1+\frac{1}{2n}\right)^{-1}\left(1+\frac{1}{n}\right)^{1-\rho}-\frac{\rho-1}{4}n^{1-\rho}\right).
\end{equation}

By putting together \eqref{majo_final_S1} and \eqref{majo_final_S2} and when factoring by
$(n+1)^{-(2-\rho)}$, we get for $S_n$

\begin{equation}\label{majo_final_S}
S_n\leqslant(n+1)^{-(2-\rho)}\left(1-\frac{1}{n}\right)^{-(2-\rho)}\times u_n
\end{equation}
with 
\begin{align*}
u_n&=1-2^{2-\rho}(n+1)^{1-\rho}\\
&\quad\quad\quad+2^{1-\rho}n^{1-\rho}\left[\left(1+\frac{1}{2n}\right)^{-1}\left(\left(1+\frac{1}{n}\right)^{2-\rho}-\left(\frac{1}{2n}\right)^{2-\rho}\right)+\left(1+\frac{1}{2n}\right)^{-1}\left(1+\frac{1}{n}\right)^{1-\rho}-\frac{\rho-1}{4}n^{1-\rho}\right]
\end{align*}

Lastly, we have the following asymptotic expansion:

\[\left(1-\frac{1}{n}\right)^{-(2-\rho)}\times u_n=1-\frac{2^{1-\rho}(\rho-1)}{4}n^{2-2\rho}+o\left(\frac{1}{n}\right)\]

Since $\rho\in(1,3/2)$ we have $2-2\rho\in(-1,0)$ therefore for $n$ large enough we conclude that
\[S_n\leqslant(n+1)^{-(2-\rho)}.\]

\section{Proof of Theorem \ref{thm:existence_invariant_dist}}\label{appendix:proof_thm_existence}

Let $x_0\in\X$ and $\mu=\delta_{x_0}\times\PP_w$. We have $\Pi_\W^*\mu=\PP_w$ therefore by Property \ref{prop:stability} we get $\forall k\in\N$, $\Pi_\W^*(\mathcal{Q}^k\mu)=\PP_w$.
Moreover, we clearly have $\int_{\X}\psi(x)(\Pi_\X^*\mu)({\rm d}x)=\psi(x_0)<+\infty$.\\
We now set for all $n\in\N^*$, \[R_n\mu=\frac1{n}\sum_{k=0}^{n-1}\mathcal{Q}^{k}\mu.\]
The aim is to prove that the sequence $(R_n\mu)_{n\in\N^*}$ is tight.
First, let us prove that $(\Pi_\X^*R_n\mu)_{n\in\N^*}$ is tight.\\
By Definition \ref{def:lyapunov_function} (ii), we have $\forall k\geqslant0$:
\[\int_{\X\times\W}\psi(x)\mathcal{Q}^{k+1}\mu({\rm d}x,{\rm d}w)-\alpha
\int_{\X}\psi(x)(\Pi_\X^*\mathcal{Q}^k\mu)({\rm d}x)\leqslant\beta.\]
By adding for $k$ from $0$ to $n-1$, dividing by $n$ and reordering the terms, we get:
\begin{align}\label{ineq1:thm_existence_inv_dist}
(1-\alpha)&
\int_{\X}\psi(x)(\Pi_\X^*R_n\mu)({\rm d}x)\nonumber\\+&\frac1{n}\sum_{k=0}^{n-1}\left(\int_{\X\times\W}\psi(x)\mathcal{Q}^{k}\mu({\rm d}x,{\rm d}w)-
\int_{\X}\psi(x)(\Pi_\X^*\mathcal{Q}^k\mu)({\rm d}x)\right)\nonumber\\
&+\frac1{n}\int_{\X\times\W}\psi(x)\mathcal{Q}^{n+1}\mu({\rm d}x,{\rm d}w)
-\frac1{n}\int_{\X}\psi(x)(\Pi_\X^*\mu)({\rm d}x)~~\leqslant\beta.
\end{align}
Since we are in a Polish space (here $\X\times\W$) we can ``disintegrate'' $\mathcal{Q}^k\mu$ for all $k\in\{0,\dots,n-1\}$ (see \cite{arnold2013random} for background):
\[\mathcal{Q}^k\mu({\rm d}x,{\rm d}w)=(\mathcal{Q}^k\mu)^x({\rm d}w)(\Pi_\X^*\mathcal{Q}^k\mu)({\rm d}x).\]
By integrating first with respect to $w$ and then with respect to $x$, we get: \[\int_{\X\times\W}\psi(x)\mathcal{Q}^{k}\mu({\rm d}x,{\rm d}w)-
\int_{\X}\psi(x)(\Pi_\X^*\mathcal{Q}^k\mu)({\rm d}x)=0.\]
Let us return to \eqref{ineq1:thm_existence_inv_dist},
\begin{align*}
(1-\alpha)\int_{\X}\psi(x)(\Pi_\X^*R_n\mu)({\rm d}x)\leqslant~&\beta+\frac1{n}\left(\int_{\X}\psi(x)(\Pi_\X^*\mu)({\rm d}x)
-\int_{\X\times\W}\psi(x)\mathcal{Q}^{n+1}\mu({\rm d}x,{\rm d}w)\right)\\
&=\beta+\frac1{n}\int_{\X}\psi(x)(\Pi_\X^*\mu-\Pi_\X^*(\mathcal{Q}^{n+1}\mu))({\rm d}x).
\end{align*}
Set $A_{n+1}=\int_{\X}\psi(x)(\Pi_\X^*\mathcal{Q}^{n+1}\mu)({\rm d}x)$.
By Definition \ref{def:lyapunov_function} (ii) and by induction, we have 
\[0\leqslant\frac{A_{n+1}}{n}\leqslant\frac{\beta}{n}\sum_{k=0}^n\alpha^k+\frac{\alpha^{n+1}}{n}A_0
=\frac{\beta}{n}\frac{1-\alpha^{n+1}}{1-\alpha}+\frac{\alpha^{n+1}}{n}\psi(x_0).\]
Hence we deduce that $\lim\limits_{n\to+\infty}\frac{A_{n+1}}{n}=0$.
Then, there exists $C>0$ such that $\forall n\in\N^*$:
\[(1-\alpha)\int_{\X}\psi(x)(\Pi_\X^*R_n\mu)({\rm d}x)\leqslant C\]
and then 
\[\sup_{n\geqslant1}\int_{\X}\psi(x)(\Pi_\X^*R_n\mu)({\rm d}x)\leqslant\frac{C}{1-\alpha}.\]
Let $\delta>0$ and $K_\delta:=\{x\in\X~|~\psi(x)\leqslant\delta\}=\psi^{-1}([0,\delta])$. By Definition \ref{def:lyapunov_function} (ii), $K_\delta$ is a compact set.
For all $x\in\X$, we have $\mathds{1}_{K_\delta^c}(x)\leqslant\frac{\psi(x)}{\delta}$, so
\[\forall n\in\N^*,~~(\Pi_\X^*R_n\mu)(K_\delta^c)\leqslant\frac{C}{\delta(1-\alpha)}.\]
By setting $\frac{\varepsilon}{2}=\frac{C}{\delta(1-\alpha)}$, we deduce that $(\Pi_\X^*R_n\mu)_{n\in\N^*}$ is tight.\\
Let us now go back to the tightness of $(R_n\mu)_{n\in\N^*}$.
Let $K$ be a compact set of $\W$ such that $\PP_w(K^c)<\frac{\varepsilon}{2}$, this is possible since $\W$ is Polish. We then get
\begin{align*}
R_n\mu((K_\delta\times K)^c)\leqslant & R_n\mu(K_\delta^c\times\W)+R_n\mu(\X\times K^c)\\
&=(\Pi_\X^*R_n\mu)(K_\delta^c)+(\Pi_\W^*R_n\mu)(K^c)\\
&=(\Pi_\X^*R_n\mu)(K_\delta^c)+(\PP_w)(K^c)\\
&\leqslant \frac{\varepsilon}{2}+\frac{\varepsilon}{2}=\varepsilon.
\end{align*}
Finally, $(R_n\mu)_{n\in\N^*}$ is tight. Let $\mu_{\star}$ be one of its accumulation points.
By the Krylov-Bogoliubov criterium we deduce that $\mu_{\star}$ is an invariant distribution for $\mathcal{Q}$.

\section{Proof of Lemma \ref{lem:technical}}\label{appendix:proof_technical_lem}

In this section we will prove a slightly more precise result than Lemma \ref{lem:technical} which is the following: for all $\alpha,\beta>0$ such that $\alpha+\beta>1$, there exists $C(\alpha,\beta)>0$ such that for all $n\geqslant0$,
\begin{equation}
\sum_{k=0}^{n}(k+1)^{-\beta}(n+1-k)^{-\alpha}\leqslant C(\alpha,\beta)~\left\{\begin{array}{lll}
(n+1)^{-\beta}\ln(n) & \text{ if }& \alpha=1 \text{ and } \beta\leqslant1\\
(n+1)^{-\alpha}\ln(n) & \text{ if } & \beta=1 \text{ and } \alpha\leqslant1\\
(n+1)^{-\min\{\alpha,\beta,\alpha+\beta-1\}} & \text{ else }
\end{array}\right..
\end{equation}
For the sake of simplicity, we will prove this result when $n$ is odd. If $n$ is even, the proof is almost the same. Set $N:=\frac{n+1}{2}$. Then, we get 
\begin{align}\label{eq:lem_tec_decomp}
\sum_{k=0}^{n}(k+1)^{-\beta}(n+1-k)^{-\alpha}
&=\sum_{k=1}^{n+1}k^{-\beta}(n+2-k)^{-\alpha}\nonumber\\
&=\sum_{k=1}^{N}k^{-\beta}(n+2-k)^{-\alpha}+\sum_{k=N+1}^{n+1}k^{-\beta}(n+2-k)^{-\alpha}\nonumber\\
&=\sum_{k=1}^{N}k^{-\beta}(n+2-k)^{-\alpha}+\sum_{k=1}^{N}k^{-\alpha}(n+2-k)^{-\beta}\nonumber\\
\sum_{k=0}^{n}(k+1)^{-\beta}(n+1-k)^{-\alpha}&= S_N(\beta,\alpha)+S_N(\alpha,\beta)
\end{align}
by setting $S_N(\alpha,\beta):=\sum_{k=1}^{N}k^{-\alpha}(n+2-k)^{-\beta}=\sum_{k=1}^{N}k^{-\alpha}(2N-(k-1))^{-\beta}$.\\

$\rhd$ If $\alpha\in(0,1)$, we have $\alpha+\beta-1\leqslant\beta$ and $S_N(\alpha,\beta)\leqslant \tilde{C}(\alpha,\beta)(n+1)^{-(\alpha+\beta-1)}$. \\Indeed,
\begin{align*}
S_N(\alpha,\beta)
&=\sum_{k=1}^{N}k^{-\alpha}(2N-(k-1))^{-\beta}\\
&=N^{-(\alpha+\beta-1)}\times \frac{1}{N}\sum_{k=1}^{N}\left(\frac{k}{N}\right)^{-\alpha}\left(2- \frac{k-1}{N}\right)^{-\beta}\\
&\leqslant N^{-(\alpha+\beta-1)}\times  \frac{1}{N}\sum_{k=1}^{N}\left(\frac{k}{N}\right)^{-\alpha}\left(2- \frac{k}{N}\right)^{-\beta}\\
\end{align*}
and 
\[\lim\limits_{N\to+\infty}\frac{1}{N}\sum_{k=1}^{N}\left(\frac{k}{N}\right)^{-\alpha}\left(2- \frac{k}{N}\right)^{-\beta}=\int_{0}^1x^{-\alpha}(2-x)^{-\beta}dx\]
where the integral is well defined since $\alpha\in(0,1)$. 
Therefore, since $N=\frac{n+1}{2}$, we deduce that there exists $\tilde{C}(\alpha,\beta)>0$ such that 
$S_N(\alpha,\beta)\leqslant \tilde{C}(\alpha,\beta)(n+1)^{-(\alpha+\beta-1)}$.\\

$\rhd$ If $\alpha>1$, we have $\alpha+\beta-1>\beta$ and $S_N(\alpha,\beta)\leqslant \tilde{C}(\alpha,\beta)(n+1)^{-\beta}$. \\Indeed,
\begin{align*}
S_N(\alpha,\beta)
&=\sum_{k=1}^{N}k^{-\alpha}(2N-(k-1))^{-\beta}\\
&\leqslant (2N-(N-1))^{-\beta}\sum_{k=1}^{N}k^{-\alpha}\\
&\leqslant(N+1)^{-\beta}\sum_{k=1}^{+\infty}k^{-\alpha}.
\end{align*}
Therefore, as before we deduce that there exists $\tilde{C}(\alpha,\beta)>0$ such that 
$S_N(\alpha,\beta)\leqslant \tilde{C}(\alpha,\beta)(n+1)^{-\beta}$.\\

$\rhd$ If $\alpha=1$, in the same way as in the case $\alpha>1$, we get
\begin{align*}
S_N(\alpha,\beta)&\leqslant(N+1)^{-\beta}\sum_{k=1}^{N}\frac{1}{k}\\
&\leqslant \tilde{C}(N+1)^{-\beta}\ln(N)
\end{align*}
Therefore, there exists $\tilde{C}(\alpha,\beta)>0$ such that $S_N(\alpha,\beta)\leqslant \tilde{C}(\alpha,\beta)(n+1)^{-\beta}\ln(n)$.\\

Finally, we get that for all $\alpha>0$ and $\beta>0$ such that $\alpha+\beta>1$,
\begin{equation}
S_N(\alpha,\beta)\leqslant \tilde{C}(\alpha,\beta)\left\{\begin{array}{lll}
(n+1)^{-\min\{\alpha,\beta,\alpha+\beta-1\}} & \text{ if } & \alpha\neq1 \\
(n+1)^{-\beta}\ln(n) & \text{ if }& \alpha=1 \\
\end{array}\right.
\end{equation}
Putting this inequality into \eqref{eq:lem_tec_decomp} we finally get the desired inequality and the proof is finished.

\section{Proof of Hypothesis $(\mathbf{H'_1})$}\label{appendix:H'_1}

We recall that we want to prove that under $(\mathbf{H_{poly}})$, the following hypothesis is true:

\begin{center}
$(\mathbf{H'_1})$: Let $\gamma\in(0,1)$. There exists $C_\gamma>0$ such that for all $j\in\N$, for every $K>0$ and for $i=1,2$,
\[\E\left[\left.\sum_{l=1}^{\Delta\tau_j}\gamma^{\Delta\tau_j-l}|\Delta^i_{\tau_{j-1}+l}|\quad\right|\mathcal{E}_j\right]<C_\gamma\]
where $\Delta\tau_j:=\tau_j-\tau_{j-1}$ and $\Delta^i$ is the stationary Gaussian sequence defined in Equation \eqref{eq:moving_average}.
\end{center}

\begin{rem}\label{rem:second_part_compact_return} 
Since the proof of this assumption will exclusively use the domination assumption on $(a_k)_{k\geqslant0}$ and since $(\tilde{a}_k)_{k\geqslant0}:=(a_{k+1})_{k\geqslant0}$ satisfies the same domination assumption, we will also get that for $i=1,2$,
\[\E\left[\left.\sum_{l=1}^{\Delta\tau_j}\gamma^{\Delta\tau_j-l}|\tilde{\Delta}^i_{\tau_{j-1}+l}|\quad\right|\mathcal{E}_j\right]<C_\gamma\]
where $\tilde{\Delta}^i_{\tau_{j-1}+l}=\sum_{k=0}^{+\infty}a_{k+1}\xi^i_{\tau_{j-1}+l-k}$. Hence, we will get that for $i=1,2$
\[\E\left[\left.|\tilde{\Delta}^i_{\tau_j}|\right|\mathcal{E}_j\right]=\E\left[\left.\left|\sum_{k=0}^{+\infty}a_{k+1}\xi^i_{\tau_j-k}\right|~\right|\mathcal{E}_j\right]< C_\gamma.\]
Then, by the Markov inequality we finally get the second part of Equation \eqref{eq:compact_return1}.
\end{rem}

We now turn to the proof of $(\mathbf{H'_1})$.
We work on the set $\mathcal{E}_j=\{\tau_{j}<+\infty\}$. We have 
\[\sum_{l=1}^{\Delta\tau_j}\gamma^{\Delta\tau_j-l}|\Delta^i_{\tau_{j-1}+l}|=\sum_{u=\tau_{j-1}+1}^{\tau_j}\gamma^{\tau_j-u}|\Delta^i_u|.\]
But,
\[|\Delta^i_u|=\left|\sum_{-\infty}^{k=u}a_{u-k}\xi^i_{k}\right|=\left|\sum_{m=0}^{j}\Lambda^i_m(u)\right|.\]
where
\begin{align}
\Lambda^i_m(u)&=\sum_{k=\tau_{m-1}+1}^{\tau_m}a_{u-k}\xi^i_{k}\quad \text{pour }m\in\{1,\dots,j-1\},
\label{eq:cut_noise1} \\
\Lambda^i_0(u)&=\sum_{-\infty}^{k=\tau_0}a_{u-k}\xi^i_{k}\quad \text{ et }\quad\Lambda^i_{j}(u)=\sum_{k=\tau_{j-1}+1}^{u}a_{u-k}\xi^i_{k}. \label{eq:cut_noise2}
\end{align}
With these notations, we get the following upper-bound 
\begin{equation}\label{ineq1_H'1}
\sum_{l=1}^{\Delta\tau_j}\gamma^{\Delta\tau_j-l}|\Delta^i_{\tau_{j-1}+l}|\leqslant\sum_{m=0}^{j}\sum_{u=\tau_{j-1}+1}^{\tau_j}\gamma^{\tau_j-u}|\Lambda^i_m(u)|.
\end{equation}

The goal of the following lemmas is to get an upper-bound of the quantity $\E[\underset{u\in\llbracket\tau_{j-1}+1,\tau_j\rrbracket}{\sup}|\Lambda^i_m(u)|~|\mathcal{E}_j]$ when $m\in\{0,\dots,j-1\}$.

\begin{lem}\label{lem:IPP_majo} Assume $(\mathbf{H_{poly}})$.
Let $t_0,t_1\in\Z$ and $u\in\N$ such that $t_0<t_1<u$. Let $(\xi_k)_{k\in\Z}$ be a sequence with values in $\R^d$.
Then,
\begin{align*}
\left|\sum_{k=t_0}^{t_1}a_{u-k}\xi_k\right|&\leqslant C_\rho(u+1-t_0)^{-\rho}\left|\sum_{k=t_0}^{t_1}\xi_k\right|+C_\kappa\sum_{k=1}^{t_1-t_0}\left|\sum_{l=k+t_0}^{t_1}\xi_l\right|(u+1-t_0-k)^{-\kappa}\\
&\leqslant C_\rho(u+1-t_0)^{-\rho}\left|\sum_{k=t_0}^{t_1}\xi_k\right|+C_\kappa\sum_{k=1}^{t_1-t_0}\left|\sum_{l=k+t_0}^{t_1}\xi_l\right|(u+1-t_0-k)^{-(\rho+1)}.
\end{align*}

\begin{rem} The last inequality just follows from the fact that $\kappa\geqslant\rho+1$ by assumption.
\end{rem}
\end{lem}

\begin{proof}
The proof is essentially based on a summation by parts argument. We set
\[\sum_{k=t_0}^{t_1}a_{u-k}\xi_k=\sum_{k=0}^{t_1-t_0}\underbrace{a_{u-t_0-k}}_{=:a'_k}\underbrace{\xi_{t_0+k}}_{=:\xi'_k}\]
and
\[\tilde{B}_k:=\sum_{l=k}^{t_1-t_0}\xi'_l\quad\text{ for }k\in\llbracket0,t_1-t_0\rrbracket.\]
We then have
\begin{align*}
\sum_{k=0}^{t_1-t_0}a'_k\xi'_k&=\sum_{k=0}^{t_1-t_0-1}a'_k(\tilde{B}_k-\tilde{B}_{k+1})+a'_{t_1-t_0}\xi'_{t_1-t_0}\\
&=\sum_{k=0}^{t_1-t_0}a'_k\tilde{B}_k-\sum_{k=1}^{t_1-t_0}a'_{k-1}\tilde{B}_k\\
&=a'_0\tilde{B}_0+\sum_{k=1}^{t_1-t_0}(a'_k-a'_{k-1})\tilde{B}_k\\
\sum_{k=0}^{t_1-t_0}a'_k\xi'_k&=a_{u-t_0}\left(\sum_{k=t_0}^{t_1}\xi_k\right)+\sum_{k=1}^{t_1-t_0}\left(\sum_{l=k+t_0}^{t_1}\xi_l\right)[a_{u-t_0-k}-a_{u-t_0-(k-1)}].
\end{align*}
Finally, by using triangular inequality and $(\mathbf{H_{poly}})$ we deduce that
\[\left|\sum_{k=t_0}^{t_1}a_{u-k}\xi_k\right|\leqslant C_\rho(u+1-t_0)^{-\rho}\left|\sum_{k=t_0}^{t_1}\xi_k\right|+C_\kappa\sum_{k=1}^{t_1-t_0}\left|\sum_{l=k+t_0}^{t_1}\xi_l\right|(u+1-t_0-k)^{-\kappa}.\]
\end{proof}

In the next lemma we adopt the convention $\sum_\emptyset=1$. Moreover, recall that by Proposition \ref{propo:adm1}, we have for every $j\in\N^*$, $\Delta\tau_j\geqslant\varsigma^j$ for an arbitrary $\varsigma>1$.

\begin{lem} \label{lem2_preuve_H'1} Assume $(\mathbf{H_{poly}})$.
We suppose that $\tau_0=0$ and that there exists $\delta_1\in(0,1)$ such that for all $m\geqslant1$ and $K>0$,
$\PP(\mathcal{E}_m|\mathcal{E}_{m-1})\geqslant\delta_1$. Then, for $i=1,2$, for all $p>1$ and for every $\varepsilon\in(0,\rho-1/2)$, there exists $C>0$ such that for all $j\geqslant1$, $m\in\{0,\dots,j-1\}$ and $K>0$,
\begin{equation}\label{eq1:lem2_preuve_H'1}
\E[\underset{u\in\llbracket\tau_{j-1}+1,\tau_j\rrbracket}{\sup}|\Lambda^i_m(u)|~|\mathcal{E}_j]\leqslant C\frac{\left(\sum_{l=m+1}^{j-1}\varsigma^l\right)^{1/2-\rho+\varepsilon}}{\delta_1^{\frac{j-m}{p}}}.
\end{equation}
Consequently, there exist $\eta\in(0,1)$ and $C>0$ such that for all $j\geqslant1$ and $m\in\{0,\dots,j-1\}$,
\begin{equation}\label{eq2:lem2_preuve_H'1}
\E[\underset{u\in\llbracket\tau_{j-1}+1,\tau_j\rrbracket}{\sup}|\Lambda^i_m(u)|~|\mathcal{E}_j]\leqslant C\eta^{~j-m}.
\end{equation}
\end{lem}

\begin{proof}
First of all, let us prove that \eqref{eq1:lem2_preuve_H'1} induces \eqref{eq2:lem2_preuve_H'1}.
Let $\alpha_1\in(0,+\infty)$ such that $\varsigma=\delta_1^{-\alpha_1}$. One just have to remark that for $j\geqslant2$ and $m\in\{1,\dots,j-2\}$,
\[\left(\sum_{l=m+1}^{j-1}\varsigma^l\right)^{1/2-\rho+\varepsilon}\delta_1^{\frac{m-j}{p}}\leqslant\delta_1^{(\alpha_1(\rho-1/2-\varepsilon)-1/p)(j-m)}\]
We choose for instance $\varepsilon=\frac{1}{2}(\rho-1/2)$ and $p>\frac{2}{\alpha_1(\rho-1/2)}$ in such a way that \[\alpha_1(\rho-1/2-\varepsilon)-1/p>0.\]
We then deduce \eqref{eq2:lem2_preuve_H'1}.\\
Now, it remains to show \eqref{eq1:lem2_preuve_H'1}. For clarity, we set \[E_j:=\llbracket \tau_{j-1}+1,\tau_j\rrbracket.\] 
Using that for $m\geqslant1$, $\mathcal{E}_{m}\subset\mathcal{E}_{m-1}$ and $\PP(\mathcal{E}_m|\mathcal{E}_{m-1})\geqslant\delta_1\in(0,1)$ and Hölder inequality we deduce the following inequalities, 
\begin{align*}
\E[\underset{u\in E_j}{\sup}|\Lambda^i_m(u)|~|\mathcal{E}_j]&\leqslant \E[\underset{u\in E_j}{\sup}|\Lambda^i_m(u)|^p~|\mathcal{E}_j]^{1/p}
=\left(\E[\underset{u\in E_j}{\sup}|\Lambda^i_m(u)|^p\mathds{1}_{\mathcal{E}_{j}}]\frac{1}{\PP(\mathcal{E}_j)}\right)^{1/p}\\
&\leqslant\left(\E[\underset{u\in E_j}{\sup}|\Lambda^i_m(u)|^p\mathds{1}_{\mathcal{E}_{j-1}}]\frac{1}{\PP(\mathcal{E}_j)}\right)^{1/p}
=\left(\E[\underset{u\in E_j}{\sup}|\Lambda^i_m(u)|^p~|\mathcal{E}_{j-1}]\frac{\PP(\mathcal{E}_{j-1})}{\PP(\mathcal{E}_j)}\right)^{1/p}\\
&\leqslant\delta_1^{-1/p}\E[\underset{u\in E_j}{\sup}|\Lambda^i_m(u)|^p~|\mathcal{E}_{j-1}]^{1/p}\\
&\leqslant(\delta_1^{-1})^{\frac{j-m}{p}}\E[\underset{u\in E_j}{\sup}|\Lambda^i_m(u)|^p~|\mathcal{E}_{m}]^{1/p}\quad \text{ (by induction).}
\end{align*}

It remains to prove the existence of $C$ such that for all $j\geqslant1$, $m\in\{0,\dots,j-1\}$ and $K>0$,
\[
\E[\underset{u\in E_j}{\sup}|\Lambda^i_m(u)|^p~|\mathcal{E}_m]^{1/p}\leqslant C\left(\sum_{l=m+1}^{j-1}\varsigma^l\right)^{1/2-\rho+\varepsilon}
\]
with again the convention $\sum_\emptyset=1$.
We separate the end of the proof into three cases.\\

\textbf{Case 1:} $j\geqslant3$ and $m\in\{1,\dots,j-2\}$.\\
By Lemma \ref{lem:IPP_majo}, applied with $t_0=\tau_{m-1}+1$ and $t_1=\tau_{m}$
\[|\Lambda^i_m(u)|\leqslant C_\rho(u-\tau_{m-1})^{-\rho}\left|\sum_{k=\tau_{m-1}+1}^{\tau_{m}}\xi^i_k\right|+C_\kappa\sum_{k=1}^{\tau_{m}-\tau_{m-1}-1}\left|\sum_{l=k+\tau_{m-1}+1}^{\tau_{m}}\xi^i_l\right|(u-\tau_{m-1}-k)^{-(\rho+1)}.\]
But, $u-\tau_{m-1}\geqslant\tau_{j-1}-\tau_{m-1}\geqslant\sum_{l=m+1}^{j-1}\varsigma^l~$ and $~u-\tau_{m-1}-k\geqslant\tau_{j-1}-\tau_{m}+1\geqslant\sum_{l=m+1}^{j-1}\varsigma^l$.\\ Let $\varepsilon\in(0,\rho-1/2)$, we then have 
\begin{align*}
|\Lambda^i_m(u)|\leqslant\left(\sum_{l=m+1}^{j-1}\varsigma^l\right)^{1/2-\rho+\varepsilon}
&\left[C_\rho(u-\tau_{m-1})^{-(1/2+\varepsilon)}\left|\sum_{k=\tau_{m-1}+1}^{\tau_{m}}\xi^i_k\right|\right.\\&\left.\quad+C_\kappa\sum_{k=1}^{\tau_{m}-\tau_{m-1}-1}\left|\sum_{l=k+\tau_{m-1}+1}^{\tau_{m}}\xi^i_l\right|(u-\tau_{m-1}-k)^{-(3/2+\varepsilon)}\right].
\end{align*}
We denote by $\tilde{\Lambda}^i_m(u)$ the above quantity between brackets. Hence
\[\E[\underset{u\in E_j}{\sup}|\Lambda^i_m(u)|^p~|\mathcal{E}_m]^{1/p}\leqslant\left(\sum_{l=m+1}^{j-1}\varsigma^l\right)^{1/2-\rho+\varepsilon}\E[\underset{u\in E_j}{\sup}|\tilde{\Lambda}^i_m(u)|^p~|\mathcal{E}_m]^{1/p}.\]
We now have to prove the existence of $C$ such that
\[\E[\underset{u\in E_j}{\sup}|\tilde{\Lambda}^i_m(u)|^p~|\mathcal{E}_m]^{1/p}\leqslant C\quad\text{ for all }p\in(1,+\infty).\]
We write $\mathcal{E}_m=\cup_{\ell\geqslant0}\mathcal{A}_{m,\ell}$ with 
\begin{equation}\label{event_fail_at_trial_ell}
\mathcal{A}_{m,\ell}=\mathcal{B}^c_{m,\ell}\cap\mathcal{B}_{m,\ell-1}
\end{equation} where $\mathcal{B}_{m,\ell}$ is defined in Equation \eqref{event_fail_after_l_attempts}. In other words, $\mathcal{A}_{m,0}$ is the failure of Step 1 of tentative $m$ and for $\ell\geqslant1$, $\mathcal{A}_{m,\ell}$ is the event ``Step 2 of trial $m$ fails after exactly $\ell$ attempts''. \\
Let $\ell\in\N$, we begin by studying $\E[\underset{u\in E_j}{\sup}|\tilde{\Lambda}^i_m(u)|^p~|\mathcal{A}_{m,\ell}]^{1/p}$. Since $u>\tau_m$, 
\[|\tilde{\Lambda}^i_m(u)|\leqslant C_\rho(\Delta\tau_m)^{-(1/2+\varepsilon)}\left|\sum_{k=\tau_{m-1}+1}^{\tau_{m-1}+\Delta\tau_m}\xi^i_k\right|+C_\kappa\sum_{k=1}^{\Delta\tau_{m}-1}\left|\sum_{l=k+\tau_{m-1}+1}^{\tau_{m-1}+\Delta\tau_m}\xi^i_l\right|(\Delta\tau_m-k)^{-(3/2+\varepsilon)}.\]
By Minkowski inequality and the fact that $\Delta\tau_m=:\Delta(m,\ell)$ is constant on $\mathcal{A}_{m,\ell}$, we get
\begin{align}\label{ineq1'_H'1}
\E[\underset{u\in E_j}{\sup}|\tilde{\Lambda}^i_m(u)|^p~|\mathcal{A}_{m,\ell}]^{1/p}\leqslant C_\rho(&\Delta(m,\ell))^{-(1/2+\varepsilon)}\E\left[\left|\sum_{k=\tau_{m-1}+1}^{\tau_{m-1}+\Delta(m,\ell)}\xi^i_k\right|^p~|\mathcal{A}_{m,\ell}\right]^{1/p}\nonumber\\ &+C_\kappa\sum_{k=1}^{\Delta(m,\ell)-1}(\Delta(m,\ell)-k)^{-(3/2+\varepsilon)}\E\left[\left|\sum_{l=k+\tau_{m-1}+1}^{\tau_{m-1}+\Delta(m,\ell)}\xi^i_l\right|^p~|\mathcal{A}_{m,\ell}\right]^{1/p}.
\end{align}
Moreover, using Cauchy-Schwarz inequality,
\begin{align}\label{term1}
\E\left[\left|\sum_{k=\tau_{m-1}+1}^{\tau_{m-1}+\Delta(m,\ell)}\xi^i_k\right|^p~|\mathcal{A}_{m,\ell}\right]^{1/p}&=\E\left[\left|\sum_{k=\tau_{m-1}+1}^{\tau_{m-1}+\Delta(m,\ell)}\xi^i_k\right|^p\mathds{1}_{\mathcal{A}_{m,\ell}}~|\mathcal{E}_{m-1}\right]^{1/p}\PP(\mathcal{A}_{m,\ell}|\mathcal{E}_{m-1})^{-1/p}\nonumber\\
&\leqslant\E\left[\left(\sum_{k=\tau_{m-1}+1}^{\tau_{m-1}+\Delta(m,\ell)}\xi^i_k\right)^{2p}~|\mathcal{E}_{m-1}\right]^{1/2p}\PP(\mathcal{A}_{m,\ell}|\mathcal{E}_{m-1})^{-1/2p}\nonumber\\
&\leqslant c_p\sqrt{\Delta(m,\ell)}~\PP(\mathcal{A}_{m,\ell}|\mathcal{E}_{m-1})^{-1/2p}.
\end{align}
In the last inequality we use the fact that $\sum_{k=\tau_{m-1}+1}^{\tau_{m-1}+\Delta(m,\ell)}\xi^i_k$ is independent from $\mathcal{E}_{m-1}$ and that its law is $\mathcal{N}(0,\Delta(m,\ell))$.
In the same way, we obtain 
\begin{equation}\label{term2}
\E\left[\left|\sum_{l=k+\tau_{m-1}+1}^{\tau_{m-1}+\Delta(m,\ell)}\xi^i_l\right|^p~|\mathcal{A}_{m,\ell}\right]^{1/p}\leqslant c_p\sqrt{\Delta(m,\ell)-k}~\PP(\mathcal{A}_{m,\ell}|\mathcal{E}_{m-1})^{-1/2p}.
\end{equation}
We deduce from \eqref{term1} and \eqref{term2} that in \eqref{ineq1'_H'1}
\begin{align}\label{ineq2_H'1}
\E[\underset{u\in E_j}{\sup}|\tilde{\Lambda}^i_m(u)|^p~|\mathcal{A}_{m,\ell}]^{1/p}&\leqslant\left(c_pC_\rho(\Delta(m,\ell))^{-\varepsilon}+c_pC_\kappa\sum_{k=1}^{\Delta(m,\ell)-1}(\Delta(m,\ell)-k)^{-(1+\varepsilon)}\right)\PP(\mathcal{A}_{m,\ell}|\mathcal{E}_{m-1})^{-1/2p}\nonumber\\
&\leqslant C_{p,\varepsilon}\PP(\mathcal{A}_{m,\ell}|\mathcal{E}_{m-1})^{-1/2p}.
\end{align}
Then by using the inequality $(a+b)^{1/p}\leqslant a^{1/p}+b^{1/p}$ for $p>1$ and \eqref{ineq2_H'1} for $\ell$ from $0$ to $+\infty$, we get
\begin{align*}
\E[\underset{u\in E_j}{\sup}|\tilde{\Lambda}^i_m(u)|^p~|\mathcal{E}_{m}]^{1/p}
&=\left(\sum_{\ell\geqslant0}\E[\underset{u\in E_j}{\sup}|\tilde{\Lambda}^i_m(u)|^p~|\mathcal{A}_{m,\ell}]\PP(\mathcal{A}_{m,\ell}|\mathcal{E}_{m})\right)^{1/p}\\
&\leqslant \sum_{\ell\geqslant0}\E[\underset{u\in E_j}{\sup}|\tilde{\Lambda}^i_m(u)|^p~|\mathcal{A}_{m,\ell}]^{1/p}\PP(\mathcal{A}_{m,\ell}|\mathcal{E}_{m})^{1/p}\\
&\leqslant C_{p,\varepsilon}\sum_{\ell\geqslant0}\PP(\mathcal{A}_{m,\ell}|\mathcal{E}_{m-1})^{-1/2p}\left(\frac{\PP(\mathcal{A}_{m,\ell}|\mathcal{E}_{m-1})}{\PP(\mathcal{E}_{m}|\mathcal{E}_{m-1})}\right)^{1/p}\\
&\leqslant C_{p,\varepsilon}\delta_1^{-1/p}\sum_{\ell\geqslant0}\PP(\mathcal{A}_{m,\ell}|\mathcal{E}_{m-1})^{1/2p}.
\end{align*}
But $\mathcal{A}_{m,\ell}\subset\mathcal{E}_m\subset\mathcal{E}_{m-1}$ hence for $\ell>0$, \[\PP(\mathcal{A}_{m,\ell}|\mathcal{E}_{m-1})=\PP(\mathcal{B}_{m,\ell-1}|\mathcal{E}_{m-1})\PP(\mathcal{B}^c_{m,\ell}|\mathcal{B}_{m,\ell-1})\leqslant\PP(\mathcal{B}^c_{m,\ell}|\mathcal{B}_{m,\ell-1}),\] 
Therefore, for all $\varepsilon\in(0,\rho-1/2)$ and $p\in(1,+\infty)$, by applying Lemma \ref{lem:success_proba_step2}, this gives us the existence of $C$ such that
\[\E[\underset{u\in E_j}{\sup}|\tilde{\Lambda}^i_m(u)|^p~|\mathcal{E}_{m}]^{1/p}\leqslant C_{p,\varepsilon}\delta_1^{-1/p}(2+\sum_{\ell\geqslant2}2^{-\tilde{\alpha}\ell/2p}) <C.\]
The first case is now achieved.\\

\textbf{Case 2:} Let $j\geqslant2$ and $m=j-1$.\\
The proof is almost exactly the same as in case 1. We simply use the following controls 
\[u-\tau_{j-2}>\Delta\tau_{j-1}\quad\text{ and }\quad u-\tau_{j-2}-k>\Delta\tau_{j-1}-k\]
and we do not introduce $\varepsilon$ which is useless here since $\sum_{l=m+1}^{j-1}=\sum_\emptyset$.\\

\textbf{Case 3:} Let $j\geqslant1$ and $m=0$. By assumption $\tau_0=0$, then 
\[\Lambda^i_0(u)=\sum_{-\infty}^{k=0}a_{u-k}\xi^i_k.\]
By Lemma \ref{lem:IPP_majo}, for all $M>0$,
\begin{align*}
|\sum_{-M}^{k=0}a_{u-k}\xi^i_k|&\leqslant C_\rho(u+1+M)^{-\rho}\left|\sum_{k=-M}^{0}\xi^i_k\right|+C_\kappa\sum_{k=1}^{M}\left|\sum_{l=k-M}^{0}\xi^i_l\right|(u+1+M-k)^{-(\rho+1)}\\
&\leqslant C_\rho(1+M)^{-\rho}\left|\sum_{k=-M}^{0}\xi^i_k\right|+C_\kappa\sum_{k=1}^{M}\left|\sum_{l=k-M}^{0}\xi^i_l\right|(M+1-k)^{-(\rho+1)}\\
&=C_\rho(1+M)^{-\rho}\left|\sum_{k=-M}^{0}\xi^i_k\right|+C_\kappa\sum_{k=0}^{M-1}\left|\sum_{l=-k}^{0}\xi^i_l\right|(k+1)^{-(\rho+1)}\\
&=C_\rho(1+M)^{-\rho}\left|\sum_{k=0}^{M}\xi^i_{-k}\right|+C_\kappa\sum_{k=0}^{M-1}\left|\sum_{l=0}^{k}\xi^i_{-l}\right|(k+1)^{-(\rho+1)}.
\end{align*}
Since $\rho>1/2$, by means of Borel-Cantelli Lemma and the fact that $\sum_{k=0}^{M}\xi^i_{-k}\sim\mathcal{N}(0,M+1)$, we can show that $\lim\limits_{M\to+\infty}(1+M)^{-\rho}\left(\sum_{k=0}^{M}\xi^i_{-k}\right)=0$ a.s.\\
We then get
\[|\Lambda^i_0(u)|\leqslant C_\kappa\sum_{k=0}^{+\infty}\left|\sum_{l=0}^{k}\xi^i_{-l}\right|(k+1)^{-(\rho+1)}.\]
Set $W^i_{k}=\sum_{l=0}^{k-1}\xi^i_{-l}$ for $k>0$ and $W^i_0=\xi^i_0$. Using Minkowski inequality, we have for all $p\in(1,+\infty)$ and for all $\varepsilon\in(0,\rho-1/2)$
\begin{align*}
\E[\underset{u\in E_j}{\sup}|\Lambda^i_0(u)|^p~|\mathcal{E}_{0}]^{1/p}=\E[\underset{u\in E_j}{\sup}|\Lambda^i_0(u)|^p]^{1/p}&\leqslant
C_\kappa\sum_{k=0}^{+\infty}(k+1)^{-(\rho+1/2-\varepsilon)}\E\left[\left(\frac{|W^i_{k+1}|}{(k+1)^{1/2+\varepsilon}}\right)^p~\right]^{1/p}\\
&\leqslant C_\kappa\E\left[\left(\sup_{k\geqslant0}\frac{|W^i_{k+1}|}{(k+1)^{1/2+\varepsilon}}\right)^p~\right]^{\frac{1}{p}}\sum_{k=0}^{+\infty}(k+1)^{-(\rho+1/2-\varepsilon)}\\
&\leqslant C \E\left[\left(\sup_{k\geqslant0}\frac{|W^i_{k+1}|}{(k+1)^{1/2+\varepsilon}}\right)^p~\right]^{\frac{1}{p}}
\end{align*}
because $\rho+1/2-\varepsilon>1$.
It remains to prove that for $\varepsilon\in(0,\rho-1/2)$ and for $p\in(1,+\infty)$
\[\E\left[\left(\sup_{k\geqslant0}\frac{|W^i_{k+1}|}{(k+1)^{1/2+\varepsilon}}\right)^p~\right]^{\frac{1}{p}}<+\infty.\]
Thanks to a summation by parts, we can show that
\[\frac{W^i_{k+1}}{(k+1)^{1/2+\varepsilon}}=\sum_{l=0}^k\frac{\xi^i_{-l}}{(l+1)^{1/2+\varepsilon}}+\sum_{l=1}^k W^i_l\left(\frac{1}{(l+1)^{1/2+\varepsilon}}-\frac{1}{l^{1/2+\varepsilon}}\right).\]
Hence, using that $\left|\frac{1}{(l+1)^{1/2+\varepsilon}}-\frac{1}{l^{1/2+\varepsilon}}\right|\leqslant(1/2+\varepsilon)\frac{1}{l^{3/2+\varepsilon}}$, we get
\[\frac{|W^i_{k+1}|}{(k+1)^{1/2+\varepsilon}}\leqslant\left|\sum_{l=0}^k\frac{\xi^i_{-l}}{(l+1)^{1/2+\varepsilon}}\right|+(1/2+\varepsilon)\sum_{l=1}^k \frac{|W^i_l|}{l^{3/2+\varepsilon}}.\]
Therefore
\[\sup_{k\geqslant0}\frac{|W^i_{k+1}|}{(k+1)^{1/2+\varepsilon}}\leqslant\sup_{k\geqslant0}\left|\sum_{l=0}^k\frac{\xi^i_{-l}}{(l+1)^{1/2+\varepsilon}}\right|+(1/2+\varepsilon)\sum_{l=1}^{+\infty} \frac{|W^i_l|}{l^{3/2+\varepsilon}}.\]
We use again Minkowski inequality, which gives
\[\left\Vert\sup_{k\geqslant0}\frac{|W^i_{k+1}|}{(k+1)^{1/2+\varepsilon}}\right\Vert_p\leqslant
\left\Vert\sup_{k\geqslant0}\left|\sum_{l=0}^k\frac{\xi^i_{-l}}{(l+1)^{1/2+\varepsilon}}\right|~\right\Vert_p+(1/2+\varepsilon)\sum_{l=1}^{+\infty} \frac{\|W^i_l\|_p}{l^{3/2+\varepsilon}}.\]
On the one hand, we have $\|W^i_l\|_p\leqslant c_p\sqrt{l}$ because $W^i_l\sim\mathcal{N}(0,l)$. On the other hand, set $N_k:=\sum_{l=0}^{k}\frac{\xi^i_{-l}}{(l+1)^{1/2+\varepsilon}}$. This is a martingale with distribution
$~\mathcal{N}\left(0,\sum_{l=0}^k(l+1)^{-(1+2\varepsilon)}\right)$ therefore $\|N_k\|_p\leqslant c_p$ with $c_p$ independent from $k$. Hence, $(N_k)_{k\in\N}$ converges a.s. and in $L^p$ into $N_\infty\in L^p$.
We then deduce by Doob's inequality that 
\[\|\sup_{k\geqslant0}|N_k|~\|_p\leqslant\left(\frac{p-1}{p}\right)\|N_\infty\|_p.\]
Finalement,
\[\left\Vert\sup_{k\geqslant0}\frac{|W^i_{k+1}|}{(k+1)^{1/2+\varepsilon}}\right\Vert_p<+\infty.\]
which achieves the third case.
\end{proof}

\begin{propo}\label{propo:proof_H'1}
Assume $(\mathbf{H_{poly}})$. We suppose that $\tau_0=0$ and that for all $m\geqslant1$ and $K>0$,
$\PP(\mathcal{E}_m|\mathcal{E}_{m-1})\geqslant\delta_1\in(0,1)$. Then $(\mathbf{H'_1})$ holds true.
\end{propo}

\begin{proof}
First, thanks to \eqref{ineq1_H'1}, we have 
\begin{equation*}
\sum_{l=1}^{\Delta\tau_j}\gamma^{\Delta\tau_j-l}|\Delta^i_{\tau_{j-1}+l}|\leqslant\sum_{m=0}^{j}\sum_{u=\tau_{j-1}+1}^{\tau_j}\gamma^{\tau_j-u}|\Lambda^i_m(u)|.
\end{equation*}
The aim is to bound every term in the right-hand side. For $m\in\{0,\dots,j-1\}$ and for all 
$u\in E_j:=\llbracket\tau_{j-1}+1,\tau_j\rrbracket$,  
\[|\Lambda^i_m(u)|\leqslant \underset{u\in E_j}{\sup}|\Lambda^i_m(u)|.\]
Since the right-hand side does not depend on $u$ anymore, we deduce that for all $m\in\{0,\dots,j-1\}$
\[\sum_{u=\tau_{j-1}+1}^{\tau_j}\gamma^{\tau_j-u}|\Lambda^i_m(u)|\leqslant\underset{u\in E_j}{\sup}|\Lambda^i_m(u)|\sum_{w=0}^{+\infty}\gamma^{w}=\underset{u\in E_j}{\sup}|\Lambda^i_m(u)|\frac{1}{1-\gamma}.\]
Hence, Lemma \ref{lem2_preuve_H'1}, gives that for all $m\in\{0,\dots,j-1\}$
\[\E\left[\left.\sum_{u=\tau_{j-1}+1}^{\tau_j}\gamma^{\tau_j-u}|\Lambda^i_m(u)|~\right|\mathcal{E}_j\right]\leqslant
\frac{C}{1-\gamma}\eta^{~j-m}\]
where $\eta\in(0,1)$. Consequently,
\begin{equation}\label{ineq3_H'1}
\E\left[\left.\sum_{m=0}^{j-1}\sum_{u=\tau_{j-1}+1}^{\tau_j}\gamma^{\tau_j-u}|\Lambda^i_m(u)|~\right|\mathcal{E}_j\right]\leqslant
\frac{C}{(1-\gamma)(1-\eta)}.
\end{equation}
In inequality \eqref{ineq1_H'1}, it then remains to bound the term with $\Lambda^i_j(u)$.
By substitution, we obtain for $m=j$
\[\sum_{u=\tau_{j-1}+1}^{\tau_j}\gamma^{\tau_j-u}|\Lambda^i_j(u)|=\sum_{v=1}^{\Delta\tau_j}\gamma^{\Delta\tau_j-v}|\Lambda^i_j(v+\tau_{j-1})|.\]
As in the proof of Lemma \ref{lem2_preuve_H'1}, we use the decomposition of $\mathcal{E}_j$ through the events $\mathcal{A}_{j,\ell}$ and that $\Delta\tau_j=:\Delta(j,\ell)$ is constant on $\mathcal{A}_{j,\ell}$:
\begin{equation}\label{eq3_preuve_H'1}
\E[\sum_{u=\tau_{j-1}+1}^{\tau_j}\gamma^{\tau_j-u}|\Lambda^i_j(u)|~|\mathcal{E}_j]=\sum_{\ell\geqslant0}\sum_{v=1}^{\Delta(j,\ell)}\gamma^{\Delta(j,\ell)-v}\E[|\Lambda^i_j(v+\tau_{j-1})|~|\mathcal{A}_{j,\ell}]\PP(\mathcal{A}_{j,\ell}|\mathcal{E}_j).
\end{equation}
Using that $\mathcal{A}_{j,\ell}\subset\mathcal{E}_j\subset\mathcal{E}_{j-1}$ and Cauchy-Schwarz inequality, one notes that 
\begin{align}\label{ineq3'_H'1}
\E[|\Lambda^i_j(v+\tau_{j-1})|~|\mathcal{A}_{j,\ell}]\PP(\mathcal{A}_{j,\ell}|\mathcal{E}_j)&\leqslant \E[|\Lambda^i_j(v+\tau_{j-1})|^2~|\mathcal{E}_j]^{1/2}\PP(\mathcal{A}_{j,\ell}|\mathcal{E}_j)^{1/2}\nonumber\\
&\leqslant\frac{\E[|\Lambda^i_j(v+\tau_{j-1})|^4~|\mathcal{E}_{j-1}]^{1/4}}{\PP(\mathcal{E}_{j}|\mathcal{E}_{j-1})^{1/4}}\PP(\mathcal{A}_{j,\ell}|\mathcal{E}_j)^{1/2}\nonumber\\
&\leqslant\frac{\underset{v\in\N^*}{\sup}\E[|\Lambda^i_j(v+\tau_{j-1})|^4~|\mathcal{E}_{j-1}]^{1/4}}{\PP(\mathcal{E}_{j}|\mathcal{E}_{j-1})^{1/4}}\PP(\mathcal{A}_{j,\ell}|\mathcal{E}_j)^{1/2}.
\end{align}
But $\PP(\mathcal{E}_j|\mathcal{E}_{j-1})\geqslant\delta_1>0$ and by Lemma \ref{lem:success_proba_step2}, we have for all $\ell\geqslant2$,
\begin{equation}\label{ineq3''_H'1}
\PP(\mathcal{A}_{j,\ell}|\mathcal{E}_j)=\frac{\PP(\mathcal{A}_{j,\ell}|\mathcal{E}_{j-1})}{\PP(\mathcal{E}_j|\mathcal{E}_{j-1})}\leqslant\delta_1^{-1}\PP(\mathcal{B}_{j,\ell-1}|\mathcal{E}_{j-1})\PP(\mathcal{B}^c_{j,\ell}|\mathcal{B}_{j,\ell-1})\leqslant\delta_1^{-1}2^{-\tilde{\alpha}\ell}.
\end{equation}

We now use \eqref{ineq3'_H'1} and \eqref{ineq3''_H'1} into \eqref{eq3_preuve_H'1} and this gives the existence of $C_{\delta_1,\gamma}$ such that
\begin{equation}\label{ineq4_H'1}
\E[\sum_{u=\tau_{j-1}+1}^{\tau_j}\gamma^{\tau_j-u}|\Lambda^i_j(u)|~|\mathcal{E}_j]\leqslant C_{\delta_1,\gamma}\underset{v\in\N^*}{\sup}\E[|\Lambda^i_j(v+\tau_{j-1})|^4~|\mathcal{E}_{j-1}]^{1/4}.
\end{equation}
It only remains to show that
\[\underset{v\in\N^*}{\sup}\E[|\Lambda^i_j(v+\tau_{j-1})|^4~|\mathcal{E}_{j-1}]^{1/4}<+\infty.\]
By Lemma \ref{lem:IPP_majo} and the definition of $\Lambda^i_j$ in \eqref{eq:cut_noise2},
\begin{align*}
|\Lambda^i_j(v+\tau_{j-1})|&=\left|\sum_{k=\tau_{j-1}+1}^{v+\tau_{j-1}}a_{v+\tau_{j-1}-k}\xi^i_k\right|\\
&\leqslant C_\rho v^{-\rho}\left|\sum_{k=\tau_{j-1}+1}^{v+\tau_{j-1}}\xi^i_k\right|+C_\kappa\sum_{k=1}^{v-1}\left|\sum_{l=k+\tau_{j-1}+1}^{v+\tau_{j-1}}\xi^i_l\right|(v-k)^{-(\rho+1)}.
\end{align*}
We again apply Minkowski inequality
\begin{align*}
\E[|\Lambda^i_j(v+&\tau_{j-1})|^4~|\mathcal{E}_{j-1}]^{1/4}\\
&\leqslant C_\rho v^{-\rho}\E\left[\left.\left|\sum_{k=\tau_{j-1}+1}^{v+\tau_{j-1}}\xi^i_k\right|^4~\right|\mathcal{E}_{j-1}\right]^{1/4}+C_\kappa\sum_{k=1}^{v-1}(v-k)^{-(\rho+1)}\E\left[\left.\left|\sum_{l=k+\tau_{j-1}+1}^{v+\tau_{j-1}}\xi^i_l\right|^4~\right|\mathcal{E}_{j-1}\right]^{1/4}\\
&=C_\rho v^{-\rho}\E\left[\left|\sum_{k=1}^{v}\xi^i_k\right|^4\right]^{1/4}+C_\kappa\sum_{k=1}^{v-1}(v-k)^{-(\rho+1)}\E\left[\left|\sum_{l=k+1}^{v}\xi^i_l\right|^4\right]^{1/4}\\
&\leqslant c_4\left(C_\rho v^{-\rho+1/2}+C_\kappa\sum_{k=1}^{v-1}(v-k)^{-(\rho+1/2)}\right)
\end{align*}
where $c_4$ is related to the $4^{th}$ moment of a centered and reduced Gaussian random variable.
Since $\rho+1/2>1$, we immediately deduce that 
\begin{equation}\label{ineq5_H'1}
\underset{v\in\N^*}{\sup}\E[|\Lambda^i_j(v+\tau_{j-1})|^4~|\mathcal{E}_{j-1}]^{1/4}<+\infty.
\end{equation}
We put together \eqref{ineq3_H'1},\eqref{ineq4_H'1} and \eqref{ineq5_H'1} to conclude the proof of $(\mathbf{H'_1})$.
\end{proof}
\bibliographystyle{plain}
{\small
\bibliography{biblio_article}}

\end{document}